\documentclass[review,hidelinks,onefignum,onetabnum]{siamart220329}

\usepackage{lipsum}
\usepackage{amsfonts}
\usepackage{graphicx}
\usepackage{epstopdf}
\usepackage{algorithmic}
\ifpdf
  \DeclareGraphicsExtensions{.eps,.pdf,.png,.jpg}
\else
  \DeclareGraphicsExtensions{.eps}
\fi

\usepackage{enumitem}
\setlist[enumerate]{leftmargin=.5in}
\setlist[itemize]{leftmargin=.5in}

\newsiamremark{remark}{Remark}
\newsiamremark{hypothesis}{Hypothesis}
\crefname{hypothesis}{Hypothesis}{Hypotheses}
\newsiamthm{claim}{Claim}

\headers{ALM for Training RNNs}{Yue Wang, Chao Zhang, and Xiaojun Chen}

\title{An Augmented Lagrangian Method for Training Recurrent Neural Networks\thanks{Submitted on 28 December 2023.
\funding{This work was supported in part by Hong Kong Research Grants Council C5036-21E and National Science Foundation of China (No. 12171027).}}}

\author{Yue Wang\thanks{Department of Applied Mathematics, Hong Kong Polytechnic University, Hong
        Kong, China (\email{yueyue.wang@connect.polyu.hk}).}
\and Chao Zhang\thanks{School of Mathematics and Statistics, Beijing Jiaotong University, Beijing 100044, China (\email{zc.njtu@163.com}).}
\and Xiaojun Chen\thanks{Department of Applied Mathematics, Hong Kong Polytechnic University,
Hong Kong, China (\email{maxjchen@polyu.edu.hk}).}}

\usepackage{amsopn}

\usepackage[section]{placeins}
\usepackage{braket, amsfonts, amsopn,amsmath}
\usepackage{algorithmic}
\usepackage{tabularx}
\usepackage{makecell}
\usepackage{graphicx, epstopdf}
\usepackage[caption=false]{subfig}
\usepackage{multirow}
\usepackage{array}
\usepackage[scr=boondox, scrscaled=0.85]{mathalfa}
\usepackage{float}
\usepackage[caption=false]{subfig}%<-Preamble
\usepackage{caption}
\let\sss= \scriptscriptstyle % make subscript small

\newcolumntype{R}{>{$}r<{$}} % 
\newcolumntype{V}[1]{>{[\;}*{#1}{R@{\;\;}}R<{\;]}} %
\ifpdf
\hypersetup{
  pdftitle={ALMRNN},
  pdfauthor={Yue Wang, Chao Zhang, and Xiaojun Chen}
}
\fi

\begin{document}
\maketitle

%% ------------------------------------------------------------------
%% ABSTRACT
%% ------------------------------------------------------------------
\begin{abstract}
Recurrent Neural Networks (RNNs) are widely used to model sequential data in a wide range of areas, such as natural language processing, speech recognition, machine translation, and time series analysis. In this paper, we model the training process of RNNs with the ReLU activation function as a constrained optimization problem with a smooth nonconvex objective function and piecewise smooth nonconvex constraints.
We prove that any feasible point of the optimization problem satisfies the no nonzero abnormal multiplier constraint qualification (NNAMCQ), and any local minimizer is a Karush-Kuhn-Tucker  (KKT) point of the problem.
Moreover, we propose an augmented Lagrangian method (ALM) and design an efficient block coordinate descent (BCD) method to solve the subproblems of the ALM.
The update of each block of the BCD method has a closed-form solution.
The stop criterion for the inner loop is easy to check and can be stopped in finite steps.
Moreover, we show that the BCD method can generate a directional stationary point of the subproblem.
Furthermore, we establish the global convergence of the ALM to a KKT point of the constrained optimization problem. Compared with the state-of-the-art algorithms, numerical results demonstrate the efficiency and effectiveness of the ALM for training RNNs.
\end{abstract}

\begin{keywords}
recurrent neural network, nonsmooth nonconvex optimization, augmented Lagrangian method, block coordinate descent
\end{keywords}

\begin{MSCcodes}
65K05, 90B10, 90C26, 90C30
\end{MSCcodes}

%% ------------------------------------------------------------------
%% MAIN SECTIONS
%% ------------------------------------------------------------------
\section{Introduction}
\label{sec:intro}
Recurrent Neural Networks (RNNs) have been applied in a wide range of areas, such as speech recognition \cite{graves2013speech, lstm2014lsa}, natural language processing \cite{mikolov2010recurrent, sundermeyer2012lstm} and nonlinear time series forecasting \cite{arneric2014garch, mirmirani2004comparison}.
In this paper, we focus on the Elman RNN architecture \cite{elman1990finding}, one of the earliest and most fundamental RNNs, and use Elman RNNs to deal with the regression task with the least squares loss function.

Given input data ${x}_{t} \in \mathbb{R}^{n}$ and output data ${y}_{t} \in \mathbb{R}^{m}$,  $t=1, \ldots, T$,
a widely used minimization problem for training RNNs
 is represented as  (see \cite[pp. 381]{goodfellow2016deep})
\begin{equation}
\label{pb:SAA_rnn}
     \min_{{A}, {W}, {V}, {b}, {c}}  \ \frac{1}{T} \sum_{t=1}^{T} \left\|y_{t}- \bigg({A} \sigma\Big({W}\big(...\sigma({V}{x}_{1}+{b})...\big)+{V}{x}_{t}+{b}\Big)+{c}\bigg)\right\|^2,
\end{equation}
where ${W} \in \mathbb{R}^{r \times r}$, ${V} \in \mathbb{R}^{r \times n}$ and ${A} \in \mathbb{R}^{m \times r}$ are unknown {weight} matrices,  ${b} \in \mathbb{R}^{r}$ and ${c} \in \mathbb{R}^{m}$ are unknown bias vectors, and {$\sigma:\mathbb{R}\to\mathbb{R}$} is a 
nonsmooth  
activation function that {is applied component-wise on vectors and} transforms the previous information and the input data ${x}_{t}$ into the hidden layer at time $t$. The training process by \cref{pb:SAA_rnn} can be interpreted as looking for proper {weight} matrices ${A},$ ${W},$ ${V},$
and  bias vectors ${b},$ ${c}$ in RNNs to minimize the difference between the true value $y_t$ and the output from RNNs across all time steps.
{It is worth mentioning that the Elman RNNs {in (\ref{pb:SAA_rnn})} shares the same weight matrices and bias vectors at different time steps \cite[pp. 374]{goodfellow2016deep}.}

When the traditional backpropagation through time (BPTT) method is used to train RNNs, the highly nonlinear and nonsmooth composition function presented in \cref{pb:SAA_rnn} poses significant challenges.  Gradient descent methods (GDs), as well as stochastic gradient descent-based methods (SGDs), are widely used to train RNNs in practice \cite{chandriah2021rnn, zhang2016multi}, but  the ``gradient" of the loss function associated with the weighted matrices via the ``chain rule" is calculated even if the ``chain rule" does not hold. The ``gradients" might exponentially increase to a very large value or shrink to zero as time $t$ increases, which makes RNNs training
with large time length $T$ very challenging \cite{bengio1994learning}.
To overcome this shortcoming, various techniques have been developed, such as gradient clipping \cite{mikolov2010recurrent}, gradient descent with Nesterov momentum \cite{bengio2013advances}, initialization with small values \cite{pascanu2013difficulty}, adding sparse regularization \cite{bengio2009learning}, and so on.
Because the essence of the above methods is to restrict the initial values of weighted matrices or gradients, they are sensitive to the choice of initial values \cite{le2015simple}. Moreover, GDs and SGDs for training RNNs lack rigorous convergence analysis.

The objective function in (\ref{pb:SAA_rnn}) is nonsmooth nonconvex and has a highly composite structure. In this paper, we equivalently reformulate (\ref{pb:SAA_rnn})
as a constrained optimization problem with a simple smooth objective function by utilizing auxiliary variables to represent the composition structures and treating these representations as constraints. Moreover, we propose an augmented Lagrangian method (ALM) for the constrained optimization problem with $\ell_2$-norm regularization, and design a block coordinate descent (BCD) method to solve the subproblem of the ALM at every iteration. The solution of the subproblems of the BCD method is very easy to compute with a closed-form.
Utilizing auxiliary variables to reformulate highly nonlinear composite structured problems as constrained optimization problems has been adopted for training Deep Neural Networks (DNNs) \cite{carreira2014distributed, cui2020multicomposite, liu2021linearly, liu2022inexact, zhang2017convergent}.
However, these algorithms for DNNs cannot be used for RNNs directly because of the difference between their architectures. In fact, RNNs share the same weighted matrices and bias vectors across different layers, whereas DNNs have distinct weighted matrices and bias vectors in different layers. In  DNNs, the weighted matrices and bias vectors can be updated layer by layer, allowing for the separation of the gradient calculation across different layers. However, in RNNs, the weighted matrices and bias vectors need to be updated simultaneously. Therefore, it is necessary to establish effective algorithms tailored to the characteristics of RNNs. To the best of our knowledge, the proposed ALM in this paper is the first first-order optimization method for training RNNs with solid convergence results.

Recently, several augmented Lagrangian-based methods have been proposed for nonconvex nonsmooth problems with composite structures. In \cite{chen2017augmented}, Chen et al. proposed an ALM for non-Lipschitz nonconvex programming, which requires the constraints to be smooth.
Hallak and Teboulle in \cite{hallak2023adaptive} transformed a comprehensive class of optimization problems into constrained problems with smooth constraints and nonsmooth nonconvex objective functions, and
proposed a novel adaptive augmented Lagrangian-based method to solve the constrained problem. The assumption on the smoothness of constraints in \cite{chen2017augmented,hallak2023adaptive}  is not satisfied for the optimization problem arising in training RNNs with nonsmooth activation functions considered in this paper.

Our contributions are summarized as follows:
\begin{itemize}
    \item {We prove that the solution set of the constrained problem with $\ell_2$ regularization is nonempty and compact. Furthermore, we prove that  any feasible point of the constrained optimization problem satisfies the no nonzero abnormal multiplier constraint qualification (NNAMCQ), which immediately guarantees any local minimizer of the constrained problems is a Karush-Kuhn-Tucker (KKT) point. }

    \item We show that any accumulation point of the sequence generated by the BCD method is a directional stationary point of the subproblem. Moreover, we show that in the $k$-th iteration of the ALM, the {stopping} criterion of the BCD method for solving the subproblem can be satisfied within  $O\big(1/(\epsilon_{k-1})^2\big)$ finite steps for any $\epsilon_{k-1}>0$.
    \item  We show that there exists an accumulation point of the sequence generated by the ALM for solving the constrained optimization problem with regularization and any accumulation point of the sequence is a KKT point.
             \item We compare the performance of the ALM with several state-of-the-art methods for both synthetic and real datasets. The numerical results verify that our ALM outperforms other algorithms in terms of forecasting accuracy for both the training sets and the test sets.
\end{itemize}

The rest of the paper is organized as follows.
In \cref{sec:formulation}, we equivalently reformulate problem (\ref{pb:SAA_rnn}) as a nonsmooth nonconvex constrained minimization problem with a simple smooth objective function. Then we show that the solution set of the constrained problem with regularization is nonempty and bounded, and give the first-order necessary optimality conditions for the constrained problem and the regularized problem.
We propose the ALM for the constrained problem with regularization, as well as the BCD method for the subproblems of the ALM in \cref{sec:alg}.  We establish the convergence results of the BCD method, and the ALM in \cref{sec:cv}. Finally, we conduct numerical experiments on both the synthetic  and  real data  in \cref{sec:ne}, which
demonstrate the effectiveness and efficiency of the ALM  for the reformulated optimization problem.

{\bf Notation and terminology.}
Let $\mathbb{N}_{+}$ denote the set of positive integers. For column vectors $\pi_1, \pi_2,\ldots,\pi_{l}$, let us denote by $\boldsymbol{\pi} := (\pi_1;\pi_2;\ldots;\pi_{ l}) = (\pi_1^{\sss \top},\pi_2^{\sss \top},\ldots, \pi_{l}^{\sss \top})^{\sss \top}$ a column vector. For a given matrix $D\in \mathbb{R}^{k \times l}$, we denote by $D_{.j}$ the $j$-th column of $D$ and use
${\rm{vec}} (D)
= (D_{.1};D_{.2};\ldots;D_{.l})\in \mathbb{R}^{kl}$ to represent a column vector.
For a given vector $g$, we use ${\rm{diag}}(g)$ to represent the diagonal matrix, whose $(i,i)$-entry is the $i$-th component $g_i$ of $g$.
We use ${\boldsymbol{e}_{l}}$ to represent the  vector of all ones in $\mathbb{R}^{l}$.
For $\nu \in \mathbb{R}$, $\lceil \nu \rceil$ refers to the smallest integer that is greater than $\nu$.
  For a given $N\in \mathbb{N}_+$, we denote  $[N]:=\{1,2,\ldots,N\}.$
We use $\|\cdot\|$ and $\| \cdot\|_{\infty}$ to denote the $\ell_2$-norm and infinity norm of a vector or a matrix, respectively. We denote by $\|\cdot\|_{F}$ the Frobenius norm of a matrix.

Let $f : \mathbb{R}^{n_1} \to \mathbb{R}$ be a
proper lower semicontinuous function defined on
$\mathbb{R}^{n_1}$. The notation $x^{k} \stackrel{f}{\rightarrow} \bar{x}$ means that $x^k \to \bar x$ and $f(x^k) \to f(\bar x)$.
The Fr\'echet subdifferential $\hat{\partial} f(x)$ and the limiting subdifferential {$\partial f(x)$} of $f$ at $\bar{x} \in \mathbb{R}^{n_1}$ are defined as
\begin{eqnarray*}
    \hat{\partial} f(\bar{x}) := \left\{g \in \mathbb{R}^{n_1} : \liminf_{x \to \bar{x}, x \neq \bar{x}}\frac{f(x)-f(\bar{x})-\langle g, x -\bar{x}\rangle}{\|x-\bar{x}\|} \geq 0 \right\}, \\
    \partial f(\bar{x}) := \left\{g \in \mathbb{R}^{n_1} : \exists x^{k} \stackrel{f}{\rightarrow} \bar{x}, g^{k} \to g \ \text{with} \ g^k \in \hat{\partial} f(x^k),\ \forall k\right\},
\end{eqnarray*} by {\cite[Definition 1.1]{kruger2003frechet}} and {\cite[Definition 8.3, pp. 301]{rockafellar2009variational}}, respectively.
{A point $\bar x$ is said to be a Fr\'echet stationary point of $\min f(x)$ if $0\in \hat \partial f(\bar x)$, and $\bar x$ is said to be a limiting stationary point of $\min f(x)$ if $0\in \partial f(\bar x)$.}
By {\cite[pp. 30]{clarke1990optimization}}, the usual (one-side) directional derivative of $f$ at $x$ in the direction $d \in \mathbb{R}^{n_1}$ is
\begin{displaymath}
    f'(x; d) := \lim_{\lambda \downarrow 0} \frac{ f(x+\lambda d)-f(x)}{\lambda},
\end{displaymath}
when the limit exists.
According to \cite[Definition 2.1]{peng2020computation},
we say that a point $\bar x \in \mathbb{R}^{n_1}$ is a d(irectional)-stationary point of $\min f(x)$ if $$f'({\bar x};d) \geq 0, \quad \forall d \in \mathbb{R}^{n_1}.$$

\section{Problem reformulation and optimality conditions}
\label{sec:formulation}
%{{
%At the beginning of this section, we first discuss the activation function $\sigma$ in problem \eqref{pb:SAA_rnn}.}
%{In deep learning, commonly used activation functions can be classified into smooth nonconvex activation functions including sigmoid, sigmoid linear unit (SiLU), and tanh, and nonsmooth convex activation functions such as Rectified Linear Unit (ReLU), and Leaky rectified linear unit (Leaky ReLU). }
%{Subsequently, based on ReLU, further activation function have been developed, such as convex and smooth activation function Exponential Linear Unit (ELU). 
%However, using the sigmoid and tanh activation functions to train RNNs may exacerbate the issues of vanishing and exploding gradients. The ReLU activation function, however, can effectively mitigate these problems \cite{pascanu2013difficulty}. }%加入更多参考文献
%{Therefore, our work mainly focuses on the ReLU activation function. 
%It is worth mentioning that our ALM method can also be applied to handle smooth convex activation functions such as ELU and smooth nonconvex activation functions like sigmoid. Detailed analysis is in Section \ref{subsec:exten_act_model}.
%\textcolor{red}{I suggest to delete this paragraph.}}

{For simplicity, we focus on}  the activation function $\sigma {: \mathbb{R} \to \mathbb{R}}$ as the ReLU function, i.e.,
\begin{eqnarray}\label{ReLU}
\quad\sigma({u}) = \max\{{u}, 0\} = ({u})_{+}.
\end{eqnarray}
{Our model, algorithms and theoretical analysis developed in this paper can be generalized to %\textcolor{magenta}{more general activation functions such as}
the leaky ReLU and the ELU activation functions.
%and the sigmoid activation function.
Detailed analysis for the extensions will be given in section \ref{subsec:exten_act_model}.}

\subsection{Problem reformulation}
We utilize auxiliary variables $\mathbf{h}$, $\mathbf{u}$ and denote vectors $\mathbf{w}, \mathbf{a}, \mathbf{z}, \mathbf{s}$ as
\begin{eqnarray*}
&&{\mathbf{h}}=({h}_1; {h}_2; ...; {h}_T)\in \mathbb{R}^{rT},
\quad
{\mathbf{u}}=({u}_1; {u}_{2}; ...; {u}_T) \in \mathbb{R}^{rT},\\
&&{\mathbf{w}}=(\text{vec}({W}); \text{vec}({V}); {b}) \in \mathbb{R}^{N_{\mathbf{w}}},
\quad
{\mathbf{a}}=(\text{vec}({A}); {c}) \in \mathbb{R}^{N_{\mathbf{a}}},\\
&&{\mathbf{z}}=({\mathbf{w}}; {\mathbf{a}})\in \mathbb{R}^{N_{\mathbf{w}}+N_{{\mathbf{a}}}},\hspace{0.2in}
\quad
{\mathbf{s}}=({\mathbf{z}};{\mathbf{h}};{\mathbf{u}})
\in \mathbb{R}^{ N_{\mathbf{w}} + N_{{\mathbf{a}}} + 2rT},
\end{eqnarray*}
where $N_{\mathbf{w}}=r^2+r n + r$ and $N_{{\mathbf{a}}}=m r+ m$.

{We reformulate problem \eqref{pb:SAA_rnn} as the following constrained optimization problem:}
\begin{equation}
\begin{split}
\label{prob:cor}
     \min_{\begin{array}{c} {\mathbf{s}} \end{array}} \quad & \frac{1}{T} \sum_{t=1}^{T}\left\|y_{t}-({A}{h}_t+{c})\right\|^2\\
      \text{s.t.} \quad &{u}_t = {W}{h}_{t-1}+{V}{x}_{t}+{b}, \\
      &
      {h}_0 = {0}, \
      {h}_t = ({u}_t)_+, \  t=1, 2, ..., T.
\end{split}
\end{equation}
{Problems \eqref{pb:SAA_rnn} and  \eqref{prob:cor} are equivalent}
in the sense that if $(A^*, W^*, V^*, b^*, c^*)$ is a {global} solution of
(\ref{pb:SAA_rnn}), then   ${\mathbf{s}^*}=({\mathbf{z}^*};{\mathbf{h^*}};{\mathbf{u^*}}) $ is a {global} solution of
(\ref{prob:cor}) where  ${\mathbf{z}^*}$ is defined by $(A^*, W^*, V^*, b^*, c^*)$ {
and ${\mathbf{h^*}
}, {\mathbf{u^*}}$ satisfy the constraints of (\ref{prob:cor}) with $W^*, V^*, b^*$.
}
Conversely, if
${\mathbf{s}^*}$ is a {global} solution of
(\ref{prob:cor}), then ${\mathbf{z}^*}$ is a {global} solution of  (\ref{pb:SAA_rnn}).

Let us denote the mappings $\Phi: \mathbb{R}^{r} \mapsto \mathbb{R}^{m \times N_{\mathbf{a}}}$ and $\Psi: \mathbb{R}^{rT} \mapsto  \mathbb{R}^{rT \times N_{\bf{w}}}$
as
\begin{eqnarray}\label{PhiPsi}
    \Phi({h}_t)=\begin{bmatrix}{h}_{t}^{\top} \otimes I_{m} & I_{m}\end{bmatrix}, \quad
    \Psi({\mathbf{h}})=\begin{bmatrix}
    {0}_{r}^{\top} \otimes I_{r} & x_{1}^{\top} \otimes I_{r} & I_{r} \\
    {h}_{1}^{\top} \otimes I_{r} & x_{2}^{\top} \otimes I_{r} & I_{r} \\
    \vdots & \vdots & \vdots \\
    {h}_{T-1}^{\top} \otimes I_{r} & x_{T}^{\top} \otimes I_{r} & I_{r}
    \end{bmatrix},
\end{eqnarray}
where $\otimes$ represents the Kronecker product, $I_r$ and $I_m$ are the identity matrices with dimensions $r$ and $m$ respectively, and $0_r$ is the zero vector with dimension $r$.
Thus, the objective function and constraints in problem \cref{prob:cor} can be represented as
\begin{equation}\label{func:cons}
\begin{array}{c}
{\displaystyle \ell({\mathbf{s}}) :=   \frac{1}{T} \sum_{t=1}^{T}\left\|y_{t}-\Phi({h}_t)\mathbf{a}\right\|^2,}\\
 {\mathcal{C}_1(\mathbf{s}) := \mathbf{u} - \Psi(\mathbf{h})\mathbf{w}=0, \quad \quad \mathcal{C}_2(\mathbf{s}) := \mathbf{h} -(\mathbf{u})_+ = 0.}
 \end{array}
\end{equation}

To mitigate the overfitting, we further add a regularization term
\begin{eqnarray}\label{P}
P({\mathbf{s}}):= \lambda_1 \|{A}\|^2_F + \lambda_2 \|{W}\|^2_F + \lambda_3\|{V}\|^2_F + \lambda_4\|{b}\|^2 + \lambda_5\|{c}\|^2 + \lambda_6 \|\mathbf{u}\|^2
\end{eqnarray}
with {$\lambda_i>0, i=1,2,\ldots,6$} in the objective of problem \cref{prob:cor}, and consider the following problem:
\begin{equation}
\begin{aligned}
\label{prob:ror}
     \min  & \quad \mathcal{R}({\mathbf{s}})  := \ell({\mathbf{s}}) + P({\mathbf{s}})\\
           \text{s.t.} & \quad {\mathbf{s}} \in \mathcal{F} := \{{\mathbf{s}}: {\cal C}_1({\mathbf{s}})=0, \ {\cal C}_2({\mathbf{s}})=0\}.
\end{aligned}
\end{equation}

\subsection{Optimality conditions}
Problem (\ref{prob:cor}) and problem (\ref{prob:ror}) have the same feasible set
$\mathcal{F}.$  The constraint function  ${\cal C}_1$ is continuously differentiable, while the other constraint function ${\cal C}_2$ is linear in $\mathbf{h}$ and piecewise linear in $\mathbf{u}.$  We denote by $J\mathcal{C}_1(\mathbf{s})$  the Jacobian matrix of the function $\mathcal{C}_1$ at $\mathbf{s}$, and by  $J_{\mathbf{z}}\mathcal{C}_1(\mathbf{s})$, $J_{\mathbf{h}}\mathcal{C}_1(\mathbf{s})$,  $J_{\mathbf{u}}\mathcal{C}_1(\mathbf{s})$  the Jacobian matrix of function $\mathcal{C}_1$ at ${\mathbf{s}}$ with respect to the block $\mathbf{z}$, $\mathbf{h}$ and $\mathbf{u}$, respectively.
Similarly,  we use
$J_{\mathbf{h}}\mathcal{C}_2(\mathbf{s})$ to represent the Jacobian matrix of ${\cal C}_2$ at $\mathbf{s}$ with respect to $\mathbf h$.  Moreover, for a fixed vector
$\zeta\in \mathbb{R}^{rT}$,
we use
$\partial \big(\zeta^{\sss \top} {\cal C}_2(\mathbf{s})\big)$ to denote the limiting subdifferential of
$\zeta^{\sss \top}{\cal C}_2$ at $\mathbf{s}$ and
$\partial_{\mathbf{u}}\big(\zeta^{\sss \top} {\cal C}_2(\mathbf{s})\big)$ to denote {the} limiting subdifferential of
$\zeta^{\sss \top}{\cal C}_2$ at $\mathbf{s}$ with respect to $\mathbf u$.

The following lemma shows that the NNAMCQ \cite[Definition 4.2, pp. 1451]{ye2000multiplier}  holds at any feasible point ${\mathbf{s}} \in \mathcal{F}.$ {The proofs of all lemmas are given in  \cref{appendixa}.}
\begin{lemma}
\label{lemma: NNAMCQ}
The NNAMCQ holds at any $\mathbf{s} \in \mathcal{F}$, i.e., there exist no nonzero vectors $\xi = (\xi_{1}; \xi_{2};...; \xi_{T}) \in \mathbb{R}^{rT}$ and $\zeta = (\zeta_{1}; \zeta_{2};...; \zeta_{T})\in \mathbb{R}^{rT}$ such that
\begin{align}
\label{subdiff_2cons}
    {0} \in J\mathcal{C}_1(\mathbf{s})^{\top} \xi + \partial \big(\zeta^{\sss \top} {\cal C}_2(\mathbf{s})\big).
\end{align}
\end{lemma}

\begin{definition}
\label{def:KKT point}
We say that $\mathbf{s} \in \mathcal{F}$ is a
KKT point of problem \cref{prob:cor} if there exist $ \xi \in \mathbb{R}^{rT}$ and $\zeta \in \mathbb{R}^{rT}$ such that
%\begin{align}
 $$   {0} \in \nabla\ell({\mathbf{s}}) + J\mathcal{C}_1(\mathbf{s})^{\top} \xi +
    \partial \big(\zeta^{\sss \top} {\cal C}_2(\mathbf{s})\big).$$
%\end{align}
We say that $\mathbf{s} \in \mathcal{F}$ is a
KKT point of problem \cref{prob:ror} if there exist $ \xi \in \mathbb{R}^{rT}$ and $\zeta \in \mathbb{R}^{rT}$ such that
%\begin{align}
    $$    {0} \in \nabla\mathcal{R}(\mathbf{s}) + J\mathcal{C}_1(\mathbf{s})^{\top} \xi +
    \partial \big(\zeta^{\sss \top} {\cal C}_2(\mathbf{s})\big).$$
%\end{align}
\end{definition}

Now we can establish the first order necessary conditions for problem \cref{prob:cor} and problem \cref{prob:ror}.
\begin{theorem}
 (i)   If $\bar{\mathbf{s}}$ is a local solution of problem \cref{prob:cor}, then $\bar{\mathbf{s}}$ is a KKT point of problem \cref{prob:cor}.  (ii)   If $\bar{\mathbf{s}}$ is a local solution of problem \cref{prob:ror}, then $\bar{\mathbf{s}}$ is a KKT point of problem \cref{prob:ror}.
\end{theorem}
\begin{proof}
    Note that the objective functions of problem \cref{prob:cor} and problem \cref{prob:ror} are continuously differentiable. The constraint functions
    $\mathcal{C}_{1}$ is continuously differentiable, and $\mathcal{C}_{2}$ is Lipschitz continuous at  any feasible point $\mathbf{s} \in \mathcal{F}$. By \cref{lemma: NNAMCQ},  NNAMCQ holds at any $\bar{s}\in {\cal F}$. Therefore, the conclusions of this theorem hold according to \cite[Remark 2 and Theorem 5.2]{ye2000multiplier}.
\end{proof}

\subsection{Nonempty and compact solution set of \cref{prob:ror}}
Let ${\cal S}_1$ be the solution set of problem \cref{prob:ror}, and denote the level set
\begin{eqnarray}
\label{level set theta}
    \mathcal{D}_{\cal R} (\rho):= \{\mathbf{s}\in {\cal F}: \mathcal{R}(\mathbf{s}) \leq \rho \}
\end{eqnarray}
with a nonnegative scalar $\rho$.
\begin{lemma}
\label{lemma:solutionset_nonempty} {For any $\rho > {\cal R}(0)$, the level set $D_{\cal R}(\rho)$ is nonempty and compact. }
Moreover, the solution set $\mathcal{S}_{1}$ of \cref{prob:ror} is nonempty and compact.
\end{lemma}

\section{ALM with BCD method for \cref{prob:ror}}

%{The augmented Lagrangian framework is attractive in constrained optimization, because it is able to reconstruct the constrained problems into tracktable reformulations and reduce the possibility of ill conditioning by replacing constraints with} \textcolor{red}{quadratic penalty  terms and explicit Lagrange multiplier estimates in the objective \cite{}.}
%{The Lagrangian-based approach has been well developed for convex optimization \cite{}, while few works of the schemes and convergence results of the augmented Lagrangian-based approach have been addressed  for optimization problems with nonsmooth objective and smooth but nonconvex constraints \cite{}. 
% No existing literature has touched the regularized constrained problem \cref{prob:ror} 
%  arising from training RNNs due to the nonsmooth nonconvex constraints.} 

{To solve the regularized constrained problem \cref{prob:ror},  we develop  in this section an ALM. {The} subproblems of ALM are approximately solved by a BCD  method whose update of each block owns a closed-form expression. This is not an easy task due to the nonsmooth nonconvex constraints.}
The framework of the ALM %{for solving \cref{prob:ror}}  
is given in \cref{Alg:ALM}, in which the updating schemes for Lagrangian multipliers and penalty parameters are motivated by \cite{chen2017augmented}. It is worth mentioning that in  \cite{chen2017augmented}, the constraints are smooth. In problem \cref{prob:ror}, the constraints are nonsmooth nonconvex. 
{For solving the subproblems in the ALM, we design the BCD method in Algorithm \ref{Alg:BCD}
%More effort are needed to design the algorithm 
and provide the closed-form expression for the update of each block in the BCD. 
Due to the nonsmooth nonconvex constraints in  \cref{prob:ror}, the convergence analysis is complex, which will be given in section \ref{sec:cv}.}

The augmented Lagrangian (AL) function of problem \cref{prob:ror} is
\begin{eqnarray}
\label{alm function for 2constraints}
    & &\quad\ \mathcal{L}(\mathbf{s},\xi,\zeta,\gamma)
    \\
    &\quad\quad& :=
    \mathcal{R}(\mathbf{s})+\langle \xi,  \mathbf{u}-\Psi(\mathbf{h})\mathbf{w}\rangle + \langle \zeta,  \mathbf{h}-(\mathbf{u})_{+}\rangle
    +
    \tfrac{\gamma}{2}\left\|\mathbf{u}-\Psi(\mathbf{h})\mathbf{w}\right\|^2 + \tfrac{\gamma}{2}\left\|\mathbf{h}-(\mathbf{u})_{+}\right\|^2
    \nonumber\\
    &\quad\quad&\ =  {\cal R}({\mathbf{s}}) + \frac{\gamma}{2} {\left\|{\mathbf{u}}-\Psi(\mathbf{h})\mathbf{w} + \frac{\xi}{\gamma}\right\|}^2
    + \frac{\gamma}{2} {\left\|\mathbf{h}-(\mathbf{u})_{+}+ \frac{\zeta}{\gamma}\right\|}^2 - \frac{{\|\xi\|}^2}{2\gamma} - \frac{{\|\zeta\|}^2}{2\gamma},
    \nonumber
\end{eqnarray}
where $\xi = (\xi_{1}; \xi_{2};...; \xi_{T})\in \mathbb{R}^{rT}$ and $\zeta = (\zeta_{1}; \zeta_{2};...; \zeta_{T})\in \mathbb{R}^{rT}$ are the Lagrangian multipliers, and $\gamma > 0$ is the penalty parameter for the two quadratic penalty terms of constraints $\mathbf{u}=\Psi(\mathbf{h})\mathbf{w}$ and $\mathbf{h}=(\mathbf{u})_{+}$.
For convenience, we will also write ${\cal L}({\bf z}, {\bf h}, {\bf u} ,\xi,\zeta,\gamma)$ to represent ${\cal L}({\bf{s}},\xi,\zeta,\gamma)$ when the blocks of $\bf s$ are emphasized.

We develop some basic results {in the following two lemmas} relating to the AL function ${\cal L}$.
%that will be used later on.
{The explicit formulas for the gradients of $\mathcal{L}$ with respect to $\mathbf{z}$ and $\mathbf{h}$ in Lemma \ref{gradient} (iii) and (iv)  will be used for obtaining the closed-form updates for the $\mathbf{z}$ and $\mathbf{h}$ blocks in the BCD method, respectively. The  Lipschitz constants $L_1(\xi,\zeta,\gamma,\hat{r})$ and $L_2(\xi,\zeta,\gamma,\hat{r})$ in Lemma \ref{lem3.2} are essential to design a practical stopping condition \eqref{cond:cri_BCD} of the BCD method in Algorithm \ref{Alg:BCD}. The results will also be used for the convergence results of the BCD method in Theorems \ref{lemma:stop_iter} and  \ref{theo:BCD_blockSeq_conv}.}
%{The proofs of the two lemmas are given in Appendix \ref{appendixa}.} 
\begin{lemma}
\label{gradient}
For any fixed $\gamma, \xi$ and $\zeta$, the following statements hold.
\begin{enumerate}[label=(\roman*)]

\item \label{lem_ele:func_lb} The AL function ${\cal L}$ is lower bounded that satisfies
\begin{eqnarray*}
{\cal L}({\bf s},\xi,\zeta,\gamma) \ge -\frac{\|\xi\|^2}{2 \gamma} - \frac{\|\zeta\|^2}{2 \gamma}\quad \mbox{for all}\  {\bf s}.
\end{eqnarray*}

\item \label{lem_ele:levelset_pro} For any $\hat {\bf {s}}$ and $\hat{\Gamma} \ge \hat{r} :=  {\cal L} (\hat{\bf s},\xi,\zeta,\gamma)$, the level set
\begin{eqnarray*}
\Omega_{{\cal L}}(\hat \Gamma) := \{{\bf s}\ : {\cal L} ({\bf s},\xi,\zeta,\gamma) \le \hat{\Gamma} \}
\end{eqnarray*}
is nonempty and compact.

\item \label{lem_ele:grad_z} The AL function  ${\cal L}$ is continuously differentiable with respect to $\bf{z}$, and the gradient with respect to ${\bf z}$ is
\begin{eqnarray*}
\nabla_{\bf z} {\cal L}({\bf{z}},{\bf{h}},{\bf{u}},\xi,\zeta,\gamma) =
\left[\begin{array}{l}
\hat{Q}_1({\bf{s}},\xi,\zeta,\gamma){\bf w} +
\hat{q}_1({\bf{s}},\xi,\zeta,\gamma)\\
\hat{Q}_2({\bf{s}},\xi,\zeta,\gamma) \mathbf{a} + \hat{q}_2({\bf{s}},\xi,\zeta,\gamma)
\end{array}
\right],
\end{eqnarray*}
where
\begin{eqnarray*}
& &\hat{Q}_1({\bf{s}},\xi,\zeta,\gamma) = \gamma \Psi({\bf{h}})^{\sss \top} \Psi({\bf{h}})  + 2 \Lambda_1,\quad
\hat{q}_1({\bf{s}},\xi,\zeta,\gamma) = - \Psi({\bf{h}})^{\sss \top} (\xi + \gamma  {\bf u})\\
& & \hat{Q}_2({\bf{s}},\xi,\zeta,\gamma) = \frac{2}{T} \sum_{t=1}^T \Phi(h_t)^{\sss \top} \Phi(h_t)  + 2  \Lambda_2,\quad
\hat{q}_2({\bf{s}},\xi,\zeta,\gamma) = - \frac{2}{T} \sum_{t=1}^T \Phi(h_t)^{\sss \top} y_t\\
& &\Lambda_1 = {\rm{diag}}\Big(\big(\lambda_2 \boldsymbol{e}_{\sss{r^{_2}}}; \lambda_3 \boldsymbol{e}_{\sss rn}; \lambda_4 \boldsymbol{e}_{\sss r}\big)\Big),\quad  \Lambda_2 =   {\rm{diag}}\Big(\big({\lambda_1 {\boldsymbol{e}_{rm}}; {{\lambda_5 \boldsymbol{e}_{m}}}}\big)\Big).
\end{eqnarray*}

\item \label{lem_ele:grad_h} The AL function  ${\cal L}$ is continuously differentiable with respect to $\bf{h}$, and the gradient with respect to $\mathbf{h}$ is
\begin{align*}
&\nabla_{\bf h} {\cal L}({\bf{z}},{\bf{h}},{\bf{u}},\xi,\zeta,\gamma)\\
=\ &\big( \nabla_{h_1} {\cal L}({\bf z}, {\bf h}, {\bf u},\xi,\zeta,\gamma); \nabla_{h_2} {\cal L}({\bf z}, {\bf h}, {\bf u},\xi,\zeta,\gamma); \ldots; \nabla_{h_T} {\cal L}({\bf z}, {\bf h}, {\bf u},\xi,\zeta,\gamma) \big),
\end{align*}
where
\begin{eqnarray*}
& &\nabla_{h_t} {\cal L}({\bf z}, {\bf h}, {\bf u},\xi,\zeta,\gamma) =  \left\{\begin{array}{ll}
D_1({\bf{s}},\xi,\zeta,\gamma) {h_t} - d_{1t}({\bf{s}},\xi,\zeta,\gamma),&\ {\rm{if}}\ t \in  [T-1],\\
D_2({\bf{s}},\xi,\zeta,\gamma) {h_T}  - d_{2T} ({\bf{s}},\xi,\zeta,\gamma),&\ {\rm{if}}\ t=T,
\end{array}
\right.\\
& & D_1({\bf{s}},\xi,\zeta,\gamma) = \gamma W^{\sss \top} W + \tfrac{2}{T} A^{\sss \top} A + \gamma I_{r} ,\\
& & D_2({\bf{s}},\xi,\zeta,\gamma) = \tfrac{2}{T} A^{\sss \top} A + \gamma I_r,\\
& & d_{1t}({\bf{s}},\xi,\zeta,\gamma) = W^{\sss \top} \left(\xi_{t+1} + \gamma (u_{t+1}-V x_{t+1}-b)\right) + \gamma (u_t)_+ -\zeta_t + \tfrac{2}{T} A^{\sss \top}(y_t -c),\\
& &
d_{2T}({\bf{s}},\xi,\zeta,\gamma) = \gamma (u_T)_+ - \zeta_T + \tfrac{2}{T} A^{\sss \top} (y_T - c).
\end{eqnarray*}

\end{enumerate}
\end{lemma}

\begin{lemma}\label{lem3.2}
For any ${\bf z}, {\bf h}, {\bf u}, {\bf h}', {\bf u}'$ in the level set  $\Omega_{{\cal L} }(\hat{r})$, we have
\begin{gather}
\|\nabla_{\bf{z}} {\cal L}({\bf z},{\bf h}', {\bf u}',\xi,\zeta,\gamma)  - \nabla_{\bf{z}} {\cal L}({\bf z},{\bf h}, {\bf u},\xi,\zeta,\gamma)\| \le L_1(\xi, \zeta, \gamma, {\hat r})
\left\|
\begin{array}{l}{\bf h}' - {\bf h}\\
{\bf u}' - {\bf u}
\end{array}
\right\|, \label{grd_lip_z}\\
\|\nabla_{\bf{h}} {\cal L}({\bf z},{\bf h}, {\bf u}',\xi,\zeta,\gamma)  - \nabla_{\bf{h}} {\cal L}({\bf z},{\bf h}, {\bf u},\xi,\zeta,\gamma)\| \le L_2(\xi, \zeta, \gamma, {\hat r})
\left\|
{\bf u}' - {\bf u}
\right\|, \label{grd_lip_h}
\end{gather}
where
\begin{eqnarray}\label{L12}
L_1(\xi,\zeta,\gamma,\hat{r}) = \sqrt{2} \max\{\gamma\delta_1, \delta_2 + \delta_3 +\delta_4\},\
L_2(\xi,\zeta,\gamma,\hat{r}) = \gamma \delta_5,
\end{eqnarray}
with $X := (x_1;x_2;...;x_T) \in \mathbb{R}^{nT}$,
\begin{eqnarray*}
& & \delta = \hat r + \frac{\|\xi\|^2}{2 \gamma} + \frac{\|\zeta\|^2}{2 \gamma},\
\delta_0 = \sqrt{\frac{2\delta}{\gamma}} + \sqrt{\frac{\delta}{\lambda_6}} + {\frac{\|\zeta\|}{\gamma}},\
 \delta_1 =  \sqrt{r(\delta^2+\|X\|^2 + T)},\\
& & \delta_2 = 2 \gamma \delta_1 \sqrt{\frac{r\delta}{\min\{\lambda_2,\lambda_3,\lambda_4\}}},\
 \delta_3 = \sqrt{r}\|\xi\| + \gamma \sqrt{\frac{r\delta}{\lambda_6}},\\
& &\delta_4 =  \frac{2\sqrt{m}}{\sqrt{T}}
\left(
2 \sqrt{m(\delta_0^2+1)} \sqrt{\frac{\delta}{\min\{\lambda_1,\lambda_5\}}} + \max_{1\le t\le T}\|y_t\|
\right),
\
\delta_5 = \sqrt{
\frac{\delta(T-1)}{\lambda_2}} + \sqrt{T}.
\end{eqnarray*}
\end{lemma}

\subsection{ALM for the regularized RNNs}
\label{subsec:ALM}
{To solve the regularized constrained problem (\ref{prob:ror}), we propose the ALM in Algorithm  \ref{Alg:ALM}. 
The ALM first approximately
solves (\ref{subproblem_2cons}) that aims to  
minimize the AL function  with the fixed Lagrange multipliers $\xi^{k-1}$ and $\zeta^{k-1}$, and the fixed penalty parameter $\gamma_{k-1}$ for the quadratic terms, until
%an $\epsilon_{k-1}$-stationary point $\mathbf{s}^k$ is obtained.
%i.e., 
$\mathbf{s}^k$ satisfies the approximate first-order optimality necessary condition \cref{cond:subdiff} with tolerance $\epsilon_{k-1}$. 
Then the Lagrange multipliers are updated, and  the tolerance $\epsilon_k$ is reduced so that in the next iteration the subproblem is solved more accurately. Moreover, the penalty parameter $\gamma_k$ is unchanged if %\textcolor{red}{the ratio of feasibility of $\mathbf{s}^{k}$ to the feasibility of $\mathbf{s}^{k-1}$ is smaller than 1} 
{the feasibility of ${\bf{s}}^k$ is sufficiently improved compared to that of ${\bf{s}}^{k-1}$}, otherwise, $\gamma_k$ is increased.}

\begin{algorithm}
\caption{\textbf{The augmented Lagrangian method (ALM) for  \cref{prob:ror}}}
\label{Alg:ALM}
\begin{algorithmic}[1]
\STATE Set an initial penalty parameter $\gamma_{0} > 0$, parameters $\eta_1, \eta_2, \eta_4 \in(0,1)$ and $\eta_3 > 1$, an initial tolerance $\epsilon_{0}>0$, vectors of Lagrangian multipliers $\xi^{0}$, $\zeta^{0}$, and a feasible initial point $\mathbf{s}^{0} = (\mathbf{z}^{0}, \hat{\mathbf{h}}, \hat{\mathbf{u}})$ where $\hat{{h}}_0 = {0}$, $\hat{{u}}_t={W}\hat{{h}}_{t-1}+{V}{x}_{t}+{b}$ and $\hat{{h}}_{t}=(\hat{u}_t)_+$ for $t \in [T]$.

\STATE Set $k:=1$.
\STATE \textbf{Step 1:} Solve
\begin{align}
\label{subproblem_2cons}
     \min_{\mathbf{s}} \quad &\mathcal{L}(\mathbf{s}, \xi^{k-1}, \zeta^{k-1}, \gamma_{k-1})
\end{align}
to obtain $\mathbf{s}^{k}$ satisfying the following condition
\begin{align}
\text{dist}\big(0, \partial
%_{\bf{s}}
\mathcal{L}(\mathbf{s}^{k},\xi^{k-1}, \zeta^{k-1}, \gamma_{\sss k-1})\big) \leq \epsilon_{k-1}. \label{cond:subdiff}
\end{align}
\STATE \textbf{Step 2:} Update $\epsilon_{k}=\eta_4 \epsilon_{k-1}$,  $\xi^{k-1}$ and $\zeta^{k-1}$ as
\begin{align}
\label{lar_mul_update}
    \xi^{k} = \xi^{k-1} +\gamma_{k-1}\left( \mathbf{u}^{k} - \Psi(\mathbf{h}^{k})\mathbf{w}^{k}\right), \quad
    \zeta^{k} = \zeta^{k-1} + \gamma_{k-1}\left( \mathbf{h}^{k}-(\mathbf{u}^{k})_{+}\right).
\end{align}\\
\STATE \textbf{Step 3:} Set $\gamma_{\sss k}=\gamma_{\sss k-1}$, if the following condition is satisfied
\begin{eqnarray}
\label{cond:update_gamma}
     \max\left\{
    \|{\cal C}_1({\bf{s}}^k)\|, \|{\cal C}_2({\bf{s}}^k)\|
    \right\}
    \leq
    \eta_1 \max\left\{
    \|{\cal C}_1({\bf{s}}^{k-1})\|, \|{\cal C}_2({\bf{s}}^{k-1})\|\right\}.
\end{eqnarray}
\STATE Otherwise, set
\begin{align}
\label{update:gamma}
    \gamma_{\sss k} =  \max \left\{\gamma_{k-1} / \eta_2,\left\|\xi^{k}\right\|^{1+\eta_3}, \left\|\zeta^{k}\right\|^{1+\eta_3} \right\}.
\end{align}
\STATE Let $k-1:= k$ and go to \textbf{Step 1}.
\end{algorithmic}
\end{algorithm}

\begin{remark}
{The main operation of Algorithm \ref{Alg:ALM} is to  approximately solve the subproblem (\ref{subproblem_2cons}). Furthermore, to show that  Algorithm \ref{Alg:ALM} is well-defined  requires that the %\textcolor{red}{specific} 
algorithm for solving the subproblem (\ref{subproblem_2cons}) can be terminated within finite steps to meet the stopping condition in (\ref{cond:subdiff}).}

{In section \ref{subsec:BCD}, we will design a BCD method to solve the subproblem (\ref{subproblem_2cons}). The update of each block of the BCD method owns a closed-form formula, which makes the BCD method efficient. Moreover, the stopping condition (\ref{cond:subdiff}) can be replaced by a simpler condition (\ref{cond:cri_BCD}) as will be shown in Theorem \ref{lemma:stop_iter}.}  
\end{remark}

\subsection{BCD method for subproblem}
\label{subsec:BCD}
To solve the nonsmooth nonconvex problem \eqref{subproblem_2cons} in Step 1 of \cref{Alg:ALM},
we propose a BCD method in \cref{Alg:BCD} to solve the  {subproblem} at the $k$-th iteration in the ALM by alternatively updating the blocks in the order of $\mathbf{z}$, $\mathbf{h}$, and $\mathbf{u}$ in $\mathbf{s}$, respectively.
Let us choose a constant
 $\Gamma$ such that\begin{eqnarray}\label{def:gamma}
\Gamma \geq \mathcal{L} \big(\mathbf{s}^{0},\xi^{0}, \zeta^{0}, \gamma_{0} \big).
\end{eqnarray}

Because at the $k$-th iteration of the ALM, $\xi^{k-1}, \zeta^{k-1}, \gamma_{k-1}$ are fixed, we just write $\xi,\zeta,\gamma$ in the BCD method for brevity.
Furthermore, for the BCD solving the {subproblem} appeared at the $k$-th iteration of the ALM,  we define
\begin{eqnarray}
\label{s_z_kj}
\quad\quad \mathbf{s}^{k-1,j}_{\sss {\bf{z}}}
 := {({\mathbf{z}^{k-1,j}}; {\mathbf{h}^{k-1,j-1}}; {\mathbf{u}^{k-1,j-1}})},
\  {\mathbf{s}}^{k-1,j}_{{\bf{h}}} := ({\mathbf{z}^{k-1,j}}; {\mathbf{h}^{k-1,j}};{\mathbf{u}^{k-1,j-1}})
\end{eqnarray}
to denote the point obtained after updating the ${\bf z}$ block, and updating the ${\bf{h}}$ block  at the $j$-th iteration of the BCD method, and
 we use \begin{eqnarray}
 \label{sk-1j}
{\bf{s}}^{k-1,j} =({\bf z}^{k-1,j}; {\bf{h}}^{k-1,j}; {\bf{u}}^{k-1,j})
\end{eqnarray}
to represent the point obtained at the $j$-th iteration of the BCD method after updating the ${\bf u}$  block.

\begin{algorithm}
\caption{\textbf{Block Coordinate Descent (BCD) method for \cref{subproblem_2cons}}}
\label{Alg:BCD}
\begin{algorithmic}[1]
\STATE  Set the initial point of BCD algorithm as
\begin{align}
\label{ini_bcd}
   \mathbf{s}^{k-1,0} =
    \begin{cases}
    \mathbf{s}^{k-1},
    &\text {if} \  k > 1 \ \text{and} \ {\small \begin{matrix}\mathcal{L}\big(\mathbf{s}^{k-1},\xi, \zeta, {\gamma}\big) \leq \Gamma\end{matrix}},\\
    \mathbf{s}^{0},
    &\text{otherwise.}\end{cases}
\end{align}
Compute ${\hat r}_{k-1} = \mathcal{L}(\mathbf{s}^{k-1,0},\xi, \zeta,\gamma)$,  $L_{1,k-1}= L_1(\xi,\zeta,\gamma,\hat{r}_{k-1})$ and $L_{2,k-1} = L_2(\xi,\zeta,\gamma,\hat{r}_{k-1})$ by formula \eqref{L12}.\\
\STATE{Set $j := 1$.}
\WHILE{the stop criterion is not met}
\STATE \textbf{Step 1:} \label{step_BCD:update variables}Update blocks $\mathbf{z}^{k-1,j}$, $\mathbf{h}^{k-1,j}$ and  $\mathbf{u}^{k-1,j}$ separately as
\begin{align}
    \mathbf{z}^{k-1,j} \ &= \ \arg\min_{\mathbf{z}} \ \mathcal{L}\left(\mathbf{z}, \mathbf{h}^{k-1,j-1}, \mathbf{u}^{k-1,j-1}, \xi, \zeta, \gamma\right), \label{subpro_z_subproblem}\\
    \mathbf{h}^{k-1,j} \ &= \ \arg\min_\mathbf{h} \ \mathcal{L}\left(\mathbf{z}^{k-1,j}, \mathbf{h}, \mathbf{u}^{k-1,j-1}, \xi, \zeta, \gamma\right), \label{subpro_h_subproblem}\\
    \mathbf{u}^{k-1,j} \ &\in \ \arg\min_\mathbf{u} \ \mathcal{L} \left(\mathbf{z}^{k-1,j}, \mathbf{h}^{k-1,j}, \mathbf{u},\xi, \zeta, \gamma\right) + \tfrac{\mu}{2}\left\|\mathbf{u} - \mathbf{u}^{k-1,j-1}\right\|^2. \label{subpro_u_subproblem}
\end{align}
\\
Then set ${\mathbf{s}^{k-1,j}} = ({\mathbf{z}^{k-1,j}}; {\mathbf{h}^{k-1,j}}; {\mathbf{u}^{k-1,j}})$.
\STATE \textbf{Step 2:}  If the stop criterion
\begin{align}
\label{cond:cri_BCD}
     \left\|\mathbf{s}^{k-1,j} - \mathbf{s}^{k-1, j-1}\right\| \leq \frac{\epsilon_{k-1}}{\max\{L_{1,k-1}, L_{2,k-1}, \mu\}},
\end{align}
 is not satisfied, then set $j := j + 1$ and go to \textbf{Step 1}.
\ENDWHILE
\RETURN $\mathbf{s}^{k} = \mathbf{s}^{k-1,j}$.
\end{algorithmic}
\end{algorithm}
 { Condition \eqref{cond:subdiff} is satisfied when \eqref{cond:cri_BCD} holds, which will be proved in Theorem \ref{lemma:stop_iter}.}
The closed-form solutions of problems \cref{subpro_z_subproblem}, \cref{subpro_h_subproblem} and \cref{subpro_u_subproblem} are provided below.

\textbf{Update $\mathbf{z}^{k-1,j}$:} Problem \cref{subpro_z_subproblem} is an unconstrained optimization problem with smooth and strongly convex objective function. By employing \cref{gradient} \ref{lem_ele:grad_z} and  solving
\begin{eqnarray*}
\nabla_{\bf{z}}{\cal L}(\mathbf{s}^{k-1,j}_{\sss {\bf{z}}},\xi,\zeta,\gamma) =0,
\end{eqnarray*} the unique global minimizer $\mathbf{z}^{k-1,j} = ({{\mathbf{w}}^{k-1,j}}; {{\mathbf{a}}^{k-1,j}})$ can be computed as
\begin{eqnarray*}
    & & \mathbf{w}^{k-1,j} = -{\hat{Q}_1(\mathbf{s}^{k-1,j}_{\sss {\bf{z}}},\xi,\zeta,\gamma)}^{-1}
\hat{q}_1(\mathbf{s}^{k-1,j}_{\sss {\bf{z}}};\xi,\zeta,\gamma), \\
    & & \mathbf{a}^{k-1,j} = -{\hat{Q}_2(\mathbf{s}^{k-1,j}_{\sss {\bf{z}}},\xi,\zeta,\gamma)}^{-1}
\hat{q}_2(\mathbf{s}^{k-1,j}_{\sss {\bf{z}}},\xi,\zeta,\gamma).
\end{eqnarray*}

\textbf{Update $\mathbf{h}^{k-1,j}$:} The objective function of \cref{subpro_h_subproblem} is also strongly convex and smooth. By employing \cref{gradient} \ref{lem_ele:grad_h} and  solving
$\nabla_{\bf{h}}{\cal L}({\mathbf{s}}^{k-1,j}_{{\bf{h}}},\xi,\zeta,\gamma) =0,
$ we get
its unique global minimizer, given by
\begin{equation}
\begin{split}
    {h}^{k-1,j}_{t} =
    \begin{cases}
        {D_1({\bf{s}}_{\bf h}^{k-1,j},\xi,\zeta,\gamma)}^{-1} d_{1t}({\bf{s}}_{\bf h}^{k-1,j},\xi,\zeta,\gamma) , & \text{if} \ t\in [T-1],\\
        {D_2({\bf{s}}_{\bf h}^{k-1,j},\xi,\zeta,\gamma)}^{-1} d_{2T}({\bf{s}}_{\bf h}^{k-1,j},\xi,\zeta,\gamma),  & \text{if} \ t=T.
    \end{cases}
\end{split}
\end{equation}

\textbf{Update $\mathbf{u}^{k-1,j}$:}
Although problem \cref{subpro_u_subproblem} is  nonsmooth nonconvex, one of its global {solutions} is accessible, because
the objective function of problem \cref{subpro_u_subproblem} can be separated into $rT$ one-dimensional functions with the same structure.
Thus, we aim to solve the following one-dimensional problem:
\begin{equation}
\label{subpro_u_ele}
    \mathop{\min}_{u\in \mathbb{R}} \ \varphi(u) := \tfrac{\gamma}{2}(u - \theta_{1})^2 + \tfrac{\gamma}{2}(\theta_{2}-({u})_{+})^2 + \tfrac{\mu}{2} (u - \theta_{3})^2 + \lambda_6 u^2,
\end{equation}
where $ \theta_{1}, \theta_{2}, \theta_{3} \in \mathbb{R}$ are known real numbers.
Denote
\begin{align}
\label{sub_u_ele_u_all}
u^{+} := \mathop{\arg\min}_{u\in \mathbb{R}_+} \varphi(u) \quad \text{and} \quad   u^{-}:=\mathop{\arg\min}_{u\in \mathbb{R}_-} \varphi(u).
\end{align}
By direct computation,
\begin{eqnarray}
\label{subpro_u_ele_u_plus}
u^{+} =
    \left\{\begin{array}{ll}
\dfrac{\gamma \theta_{1} + \gamma \theta_{2} + \mu \theta_{3}}{2\gamma + 2\lambda_6 + \mu},  & \text{if} \ \gamma\theta_{1} + \gamma \theta_{2} + \mu \theta_{3} > 0,  \\
    0, & \text{otherwise},
    \end{array}
    \right.
\end{eqnarray}
and
\begin{eqnarray}
\label{subpro_u_ele_u_minus}
    u^{-} = \left\{
    \begin{array}{ll}
\dfrac{\gamma\theta_{1} + \mu \theta_{3}}{\gamma + 2\lambda_6 + \mu}, & \text{if} \ \gamma\theta_{1} + \mu \theta_{3}< 0,  \\
    0, & \text{otherwise}.
    \end{array}
\right.
\end{eqnarray}
Then a solution of (\ref{subpro_u_ele}) can be given as
\begin{eqnarray*}
u^* = \left\{
      \begin{array}{ll}
      u^+, & \text{if}\ \varphi(u^+) \le \varphi (u^{-}),\\
      u^{-}, & \text{otherwise}.
      \end{array}
      \right.
\end{eqnarray*}

By setting
\begin{eqnarray*}
\theta_{1} =(\Psi(\mathbf{h}^{k-1,j})\mathbf{w}^{k-1,j})_{i}- \frac{\xi_{i}}{\gamma},\quad \theta_{2} ={\mathbf{h}}^{k-1,j}_i + \frac{\zeta_{i}}{\gamma},\quad \theta_{3} ={{\mathbf{u}}_i^{k-1,{j-1}}},\\
{{\mathbf{u}}_{i}^{k-1,j}}=u^*, \quad {{{\mathbf{u}}_i}^+}=u^+  ,\quad {{\mathbf{u}}_i}^{-}=u^{-}, \quad\quad\quad\quad
\end{eqnarray*}
we obtain
 a closed-form solution of problem \cref{subpro_u_subproblem} as
\begin{eqnarray*}
{\mathbf{u}}^{k-1,j}_{i} = \left\{\begin{array}{ll}
\mathbf{u}_{i}^{+}, & \text{if} \ \varphi(\mathbf{u}_{i}^{+}) \leq \varphi(\mathbf{u}_{i}^{-}),\\
\mathbf{u}_{i}^{-}, & \text{otherwise}, \quad \quad  \quad\quad i=1,\ldots, rT.
    \end{array}
    \right.
\end{eqnarray*}

\begin{remark}
    It is important to mention that the solution set of problem \cref{subpro_u_subproblem} may not be a singleton. To ensure the selected solution is unique, we set ${\mathbf{u}}^{k-1,j}_{i} = \mathbf{u}_{i}^{+}$ when $\varphi(\mathbf{u}_{i}^{+}) = \varphi(\mathbf{u}_{i}^{-})$ for every $ i\in [rT]$.
\end{remark}

\section{Convergence analysis}
\label{sec:cv}
In this section, we show the convergence results of both the BCD method for the subproblem of the ALM, as well as the ALM for \cref{prob:ror}.

\subsection{Convergence analysis of \cref{Alg:BCD}}
\label{subsec:conv_bcd}

It is clear that
\begin{align}
\label{almfunc_subproblem_2cons}
\mathcal{L}(\mathbf{s},\xi, \zeta,\gamma)
    = g(\mathbf{s},\xi,\gamma) + q(\mathbf{s},\zeta,\gamma),
\end{align}
where
\begin{gather}
    g(\mathbf{s},\xi,\gamma) = \mathcal{R}(\mathbf{s})
    + \frac{\gamma}{2} {\left\|{\mathbf{u}}-\Psi(\mathbf{h})\mathbf{w} + \frac{\xi}{\gamma}\right\|}^2
    - \frac{\|\xi\|^2}{2 \gamma} , \label{func:subg} \\
    q(\mathbf{s},\zeta,\gamma)
    = \frac{\gamma}{2} {\left\|\mathbf{h}-(\mathbf{u})_{+}+ \frac{\zeta}{\gamma}\right\|}^2 - \frac{{\|\zeta\|}^2}{2\gamma}. \label{func:subq}
\end{gather}
The function $g$ is smooth but nonconvex, because it contains the bilinear structure ${\Psi(\mathbf{h})}\mathbf{w}$. The function $q$ is nonsmooth nonconvex.

For the convergence analysis below, we further use
${\bf{s}}_{ \bf{z}}^{(j)}$ and ${\bf{s}}_{\bf{h}}^{(j)}$ to represent ${\bf{s}}_{ \bf{z}}^{k-1,j}$ and ${\bf{s}}_{\bf{h}}^{k-1,j}$ in \eqref{s_z_kj}, and
use  ${\bf s}^{(j)}$ to represent $s^{k-1,j}$ in \eqref{sk-1j} for brevity.
We emphasize that the point ${\bf{s}}^k$ is generated by the ALM in \cref{Alg:ALM}, while the point ${\bf s}^{(j)}$ is generated by the BCD method in \cref{Alg:BCD} for solving the subproblem in the ALM at the $k$-th iteration. 

{The following two lemmas will be used in proving the convergence results of the BCD method.}
%the sequence generated by the BCD method is guaranteed to be in the level set  $\Omega_{\mathcal{L}}(\Gamma)$ where $\Gamma$ is an upper bound of ${\mathcal{L}}\big(\mathbf{s}^{0},\xi^{0}, \zeta^{0}, \gamma_{0} \big)$ that first appears in  (\ref{def:gamma}), and the AL function $\cal L$ owns nice properties such as locally Lipschitz continuous and directionally differentiable on $\Omega_{\mathcal{L}}(\Gamma)$.   
%The proofs of the two lemmas are in Appendix \ref{appendixa}.} 

\begin{lemma}
\label{lemma:sj_in_O}
    Let $\{\mathbf{s}^{(j)}\}$ represent the sequence generated by \cref{Alg:BCD}. Then $\{\mathbf{s}^{(j)}\}$ belongs to the level set $\Omega_{\mathcal{L}}(\Gamma)$, which is compact.
\end{lemma}

\begin{lemma}
\label{lemma: directional derivative}
    The AL function $\mathcal{L}$
    is locally Lipschitz continuous and directionally differentiable on
    $\Omega_{\mathcal{L}}(\Gamma).$
\end{lemma}

%\begin{proof}
%It is clear that $\Omega_{\mathcal{L}}(\Gamma)$ is compact by \cref{gradient} \ref{lem_ele:levelset_pro}.
%For the smooth part $g$ in $\mathcal{L}$, its gradient for those $\mathbf{s} \in \Omega_{\mathcal{L}}(\Gamma)$ is upper bounded. Now, let us turn to consider the nonsmooth part $q$ in $\mathcal{L}$. Let  $\mathbf{s} = (\bf{z}; \bf{h};\bf{u})$ and $\mathbf{s}' = (\bf{z'}; \bf{h'};\bf{u'})$ be any two points in $\Omega_{\mathcal{L}}(\Gamma)$. We have
%\begin{align*}
%    &\big|q(\mathbf{s}', \zeta, \gamma) - q(\mathbf{s}, \zeta, \gamma)\big|
 %   \\
 %   \leq\ &\tfrac{\gamma}{2}\Big| \big\|\mathbf{h}'-(\mathbf{u}')_{+}+\tfrac{\zeta}{\gamma}\big\|^2 - \big\|\mathbf{h}-(\mathbf{u})_{+}+\tfrac{\zeta}{\gamma}\big\|^2\Big|\\
 %   \leq\ &\tfrac{\gamma}{2}\big\|\mathbf{h}'-(\mathbf{u}')_{+}-(\mathbf{h}-(\mathbf{u})_{+})\big\| \big\|\mathbf{h}'-(\mathbf{u}')_{+} + \mathbf{h}-(\mathbf{u})_{+} + 2\tfrac{\zeta}{\gamma}\big\|\\
 %   \leq\ & \left(2 \gamma \max_{{\bf{s}}\in \Omega_{\mathcal{L}}(\Gamma)}\{\|\mathbf{h}\|_{\infty}+ \|\mathbf{u}\|_{\infty}\} + \|\zeta\|\right)(\|\mathbf{h}' - \mathbf{h}\|+\|\mathbf{u}' - \mathbf{u}\|).
%\end{align*}
%Up to now, we have proved the Lipschitz continuity of $g$ and $q$ on $\Omega_{\mathcal{L}}(\Gamma)$, which implies that $\mathcal{L}$ is Lipschitz continuous on $\Omega_{\mathcal{L}}(\Gamma)$.

%The above result, together with the piecewise smoothness of function $\mathcal{L}$, yields that $\mathcal{L}$ is directionally differentiable on $\Omega_{\mathcal{L}}(\Gamma)$ by \cite{mifflin1977semismooth}.
%\end{proof}

We can now show that the stop criterion \cref{cond:cri_BCD} in \cref{Alg:BCD} can be stopped in finite steps, {and condition \eqref{cond:subdiff} in Algorithm \ref{Alg:ALM} is satisfied when \eqref{cond:cri_BCD} holds.  
These results guarantee that the ALM in Algorithm \ref{Alg:ALM} is well-defined, when the subproblems are solved by the BCD method in  \cref{Alg:BCD}.
} 

\begin{theorem}
\label{lemma:stop_iter}
At the $k$-th iteration of ALM  in \cref{Alg:ALM}, the BCD method in \cref{Alg:BCD} for the subproblem \cref{subproblem_2cons} can be stopped within
finite steps to satisfy the stop criterion in \cref{cond:cri_BCD},
which is of order $O(1/(\epsilon_{k-1})^2)$. Moreover, condition \cref{cond:subdiff} of the ALM in \cref{Alg:ALM}  is satisfied at the output ${\bf{s}}^k$ of \cref{Alg:BCD}.
\end{theorem}
\begin{proof}
Since $\mathcal{L}$ is strongly convex with respect to the blocks $\bf z$ and $\bf h$, respectively, from \cref{subpro_z_subproblem} and \cref{subpro_h_subproblem},  we obtain
\begin{align}
   \mathcal{L}(\mathbf{s}^{(j-1)},\xi,\zeta,\gamma) - \mathcal{L}(\mathbf{s}^{(j)}_{\sss {\bf{z}}},\xi,\zeta,\gamma) \geq \tfrac{\alpha_1}{2} \|\mathbf{z}^{(j-1)} - \mathbf{z}^{(j)}\|^2, \label{ineq:rel_fv_z}\\
    \mathcal{L}(\mathbf{s}^{(j)}_{\sss {\bf{z}}},\xi,\zeta,\gamma) - \mathcal{L}(\mathbf{s}^{(j)}_{\sss {\bf{h}}},\xi,\zeta,\gamma) \geq \tfrac{\alpha_2}{2} \|\mathbf{h}^{(j-1)} - \mathbf{h}^{(j)}\|^2, \label{ineq:rel_fv_h}
\end{align}
where $\alpha_1$ and $\alpha_2$ are the minimum eigenvalues of the Hessian matrices $\nabla^2_{\bf{z}} \mathcal{L}({\bf{s}},\xi,\zeta,\gamma)$ and $\nabla^2_{\bf{h}} \mathcal{L}({\bf{s}},\xi,\zeta,\gamma)$ for all {\bf{s}} in the compact set $\Omega_{\mathcal{L}}(\Gamma)$, respectively.
Furthermore, by \cref{subpro_u_subproblem}, we have
\begin{align}
\label{ineq:rel_fv_u}
    \mathcal{L}(\mathbf{s}^{(j)}_{\sss {\bf{h}}},\xi,\zeta,\gamma) - \mathcal{L}(\mathbf{s}^{(j)},\xi,\zeta,\gamma) \geq  \tfrac{\mu}{2}\left\|\mathbf{u}^{(j)} - \mathbf{u}^{(j-1)}\right\|^2.
\end{align}
It follows that
\begin{align*}
    &\mathcal{L}(\mathbf{s}^{(j-1)},\xi,\zeta,\gamma) - \mathcal{L}(\mathbf{s}^{(j)},\xi,\zeta,\gamma)\\
    =\ &\big(\mathcal{L}(\mathbf{s}^{(j-1)},\xi,\zeta,\gamma) - \mathcal{L}(\mathbf{s}^{(j)}_{\sss {\bf{z}}},\xi,\zeta,\gamma)\big)+ \big(\mathcal{L}(\mathbf{s}^{(j)}_{\sss {\bf{z}}},\xi,\zeta,\gamma) - \mathcal{L}(\mathbf{s}^{(j)}_{\sss {\bf{h}}},\xi,\zeta,\gamma)\big)\\
    \quad\quad &
    + \big(\mathcal{L}(\mathbf{s}^{(j)}_{\sss {\bf{h}}},\xi,\zeta,\gamma) - \mathcal{L}(\mathbf{s}^{(j)},\xi,\zeta,\gamma)\big)\\
    \geq \ &\tfrac{\alpha_1}{2} \|\mathbf{z}^{(j)} - \mathbf{z}^{(j-1)}\|^2 + \tfrac{\alpha_2}{2}\|\mathbf{h}^{(j)} - \mathbf{h}^{(j-1)}\|^2 + \tfrac{\mu}{2}\|\mathbf{u}^{(j)} - \mathbf{u}^{(j-1)}\|^2 \\
    \geq \ &\max\{\tfrac{\alpha_1}{2}, \tfrac{\alpha_2}{2}, \tfrac{\mu}{2}\}\|\mathbf{s}^{(j)} - \mathbf{s}^{(j-1)}\|^2.
\end{align*}
Summing up $\mathcal{L}(\mathbf{s}^{(j-1)},\xi,\zeta,\gamma)$ $-$ $\mathcal{L}(\mathbf{s}^{(j)},\xi,\zeta,\gamma)$ from $j = 1$ to $J$, we have
\begin{eqnarray}\label{line409}
    \mathcal{L}(\mathbf{s}^{(0)},\xi,\zeta,\gamma)- \mathcal{L}(\mathbf{s}^{(J)},\xi,\zeta,\gamma)&\geq& \max\{\tfrac{\alpha_1}{2}, \tfrac{\alpha_2}{2}, \tfrac{\mu}{2}\}\sum_{j=1}^{J} \|\mathbf{s}^{(j)} - \mathbf{s}^{(j-1)}\|^2\\
    &\geq& J\max\{\tfrac{\alpha_1}{2}, \tfrac{\alpha_2}{2}, \tfrac{\mu}{2}\}  \min_{\sss j\in [J]} \{\|\mathbf{s}^{(j)} - \mathbf{s}^{(j-1)}\|^2\}.\nonumber
\end{eqnarray}
This, together with  \cref{gradient} \ref{lem_ele:func_lb}, yields that
\begin{eqnarray*}
 \min_{\sss j\in [J]} \{\|\mathbf{s}^{(j)} - \mathbf{s}^{(j-1)}\|^2\} \le
\frac{\mathcal{L}(\mathbf{s}^{(0)},\xi,\zeta,\gamma)   +  \frac{\|\xi\|^2}{2 \gamma} + \frac{{\|\zeta\|}^2}{2\gamma}}{J \max\{\tfrac{\alpha_1}{2}, \tfrac{\alpha_2}{2}, \tfrac{\mu}{2}\}}.
\end{eqnarray*}
It follows that the stop criterion \cref{cond:cri_BCD} holds, as long as
\begin{eqnarray}\label{hatJ}
J \ge \hat J := \left\lceil \frac
{\big(\mathcal{L}(\mathbf{s}^{(0)},\xi,\zeta,\gamma) +  \frac{\|\xi\|^2}{2 \gamma}+\frac{\|\zeta\|^2}{2 \gamma}\big)(\max\{L_{1,k-1}, L_{2,k-1},\mu\})^2}{\max\{\tfrac{\alpha_1}{2}, \tfrac{\alpha_2}{2}, \tfrac{\mu}{2}\}(\epsilon_{k-1})^2}\right\rceil.
\end{eqnarray}
Therefore, at the $k$-th iteration of the ALM  in \cref{Alg:ALM}, the BCD method in  \cref{Alg:BCD}  can be stopped in at most $\hat J$  iterations defined in \eqref{hatJ} and output ${\bf{s}}^k$,  which is of order
$O(1/(\epsilon_{k-1})^2)$.

Once condition \cref{cond:cri_BCD} is satisfied, condition \cref{cond:subdiff} in \cref{Alg:ALM} also holds, which will be proved in the following.
By Step 1 in \cref{Alg:BCD}, the first order optimality condition of the three blocked subproblems \cref{subpro_z_subproblem}, \cref{subpro_h_subproblem} and \cref{subpro_u_subproblem} are
\begin{gather*}
    0 = \nabla_{\bf z}
    \mathcal{L}(\mathbf{s}^{(j)}_{\bf{z}},\xi,\zeta,\gamma),
    \ 0 = \nabla_{\bf{h}} \mathcal{L}(\mathbf{s}^{(j)}_{{\bf h}},\xi,\zeta,\gamma),\\ \ 0 \in \nabla_{\bf{u}} g(\mathbf{s}^{(j)}, \xi, \gamma) + \partial_{\bf{u}} q(\mathbf{s}^{(j)}, \zeta, \gamma) + \mu(\mathbf{u}^{(j)} - \mathbf{u}^{(j-1)}).
\end{gather*}
Furthermore, the limiting subdifferential of the function $\mathcal{L}$ at $\mathbf{s}^{(j)}$ can be written as
\begin{eqnarray*}
    \partial \mathcal{L}(\mathbf{s}^{(j)},\xi,\zeta,\gamma) =
    \left(
    \nabla_{\bf{z}} {\cal L}(\mathbf{s}^{(j)},\xi,\zeta,\gamma);
    \nabla_{\bf{h}} \mathcal{L}(\mathbf{s}^{(j)},\xi,\zeta,\gamma);
    \nabla_{\bf{u}} g(\mathbf{s}^{(j)}, \xi) + \partial_{\bf{u}} q(\mathbf{s}^{(j)}, \zeta)
    \right).
\end{eqnarray*}
Hence
\begin{displaymath}
\left[\begin{array}{c}
    \nabla_{\bf{z}} {\cal L}(\mathbf{s}^{(j)},\xi,\zeta,\gamma) - \nabla_{\bf{z}} \mathcal{L}(\mathbf{s}^{(j)}_{\bf{z}},\xi,\zeta,\gamma)\\
    \nabla_{\bf{h}} \mathcal{L}(\mathbf{s}^{(j)},\xi,\zeta,\gamma) - \nabla_{\bf{h}} \mathcal{L}(\mathbf{s}^{(j)}_{\bf{h}},\xi,\zeta,\gamma)\\
    -\mu(\mathbf{u}^{(j)} - \mathbf{u}^{(j-1)})
\end{array}\right] \in \partial \mathcal{L}(\mathbf{s}^{(j)},\xi,\zeta,\gamma).
\end{displaymath}
By \cref{lem3.2}, we obtain
\begin{eqnarray*}
\label{ineq:sub_vb}
    {\rm dist}\big(0,\partial \mathcal{L}(\mathbf{s}^{(j)},\xi,\zeta,\gamma)\big)
    &\leq& \left\|\begin{array}{c}
    \nabla_{\bf{z}} \mathcal{L}(\mathbf{s}^{(j)},\xi,\zeta,\gamma) - \nabla_{\bf{z}} \mathcal{L}(\mathbf{s}^{(j)}_{\bf{z}},\xi,\zeta,\gamma) \\
    \nabla_{\bf{h}} \mathcal{L}(\mathbf{s}^{(j)},\xi,\zeta,\gamma) - \nabla_{\bf{h}} \mathcal{L}(\mathbf{s}^{(j)}_{\bf{h}},\xi,\zeta,\gamma)\\
    -\mu(\mathbf{u}^{(j)} - \mathbf{u}^{(j-1)})
    \end{array}
    \right\|\\
    &\leq & \max\{L_{1,k-1}, L_{2,k-1},\mu\}\|\mathbf{s}^{(j)} - \mathbf{s}^{(j-1)}\|.
\end{eqnarray*}
Thus condition \cref{cond:cri_BCD} that $\|\mathbf{s}^{(j)} - \mathbf{s}^{(j-1)}\| \leq \epsilon_{k-1}/\max\{L_{1,k-1}, L_{2,k-1},\mu\}$, together with
 ${\bf s}^k = {\bf s}^{(j)}$, implies ${\rm dist}(0, \partial \mathcal{L}(\mathbf{s}^{(k)},\xi,\zeta,\gamma))$ $\leq$ $\epsilon_{k-1}$ in condition \cref{cond:subdiff}.
\end{proof}

{Theorem \ref{lemma:stop_iter} above guarantees that the BCD method in Algorithm \ref{Alg:BCD} terminates within finite steps to meet the stop criterion (\ref{cond:cri_BCD}) for a fixed $\epsilon_{k-1}>0$.} 
 In the rest of this subsection, we discuss the convergence of \cref{Alg:BCD} for the case $\epsilon_{k-1}=0$, i.e.,   we replace the stop criterion
(\ref{cond:cri_BCD}) by
 \begin{equation}\label{cond:criBCD2}
     \left\|\mathbf{s}^{k-1,j} - \mathbf{s}^{k-1, j-1}\right\|=0.
     \end{equation}
   {We will show in Theorem \ref{d-sta} that the BCD method converges to a d-stationary point if $\epsilon_{k-1}=0$.}  
{For this purpose, we first show the following theorem that provides the convergence of the sequences of the function values $\cal L$ with respect to the three blocks, as well as the convergence of the subsequences of the iterative points with respect to the three blocks.}
     \vspace{-0.08in}
\begin{theorem}
\label{theo:BCD_blockSeq_conv}
    Suppose that (\ref{cond:cri_BCD}) is replaced  by
 (\ref{cond:criBCD2}) in \cref{Alg:BCD}. If there is $\bar{j}$ such that (\ref{cond:criBCD2}) holds, then
\begin{equation}\label{Thm4.4-1}
\mathcal{L}(\mathbf{s}^{(\bar{j})}_{\bf{ z}},\xi,\zeta,\gamma)= \mathcal{L}(\mathbf{s}^{(\bar{j})}_{\bf{h}},\xi,\zeta,\gamma)= \mathcal{L}(\mathbf{s}^{(\bar{j})},\xi,\zeta,\gamma) \quad {\rm  and}\quad
  \mathbf{s}^{(\bar{j})}_{\bf{z}}=\mathbf{s}^{(\bar{j})}_{\bf{h}}= \mathbf{s}^{(\bar{j})}.
  \end{equation}
   Otherwise, \cref{Alg:BCD} generates infinite sequences $\{\mathbf{s}^{(j)}_{\bf{z}}\}$, $\{\mathbf{s}^{(j)}_{\bf{h}}\}$ and $\{\mathbf{s}^{(j)}\}$, and the following statements hold.
    \begin{enumerate}[label=(\roman*)]
\item      \label{theo_ele:L(block seq) converge} The sequences $\{\mathcal{L}(\mathbf{s}^{(j)}_{\bf{ z}},\xi,\zeta,\gamma)\}$, $\{\mathcal{L}(\mathbf{s}^{(j)}_{\bf{h}},\xi,\zeta,\gamma)\}$ and $\{\mathcal{L}(\mathbf{s}^{(j)},\xi,\zeta,\gamma)\}$ all converge to a constant $\mathcal{L}^{\ast}$.
        \item \label{theo_ele:block seq converge} There {exists} a subsequence
        $\{j_i\} \subseteq \{j\}$ such that $\{\mathbf{s}^{(j_i)}_{\bf{z}}\}$, $\{\mathbf{s}^{(j_i)}_{\bf{h}}\}$ and $\{\mathbf{s}^{(j_i)}\}$ converging to the same point.
    \end{enumerate}
\end{theorem}
\begin{proof}
If there is $\bar{j}$ such that (\ref{cond:criBCD2}) holds, then  (\ref{Thm4.4-1}) is derived directly from
$\mathbf{s}^{k-1,\bar{j}}=\mathbf{s}^{k-1, \bar{j}-1}$ and
(\ref{subpro_z_subproblem})-(\ref{subpro_u_subproblem}).

If there is {no} $\bar{j}$ such that (\ref{cond:criBCD2}) holds,
    then \cref{Alg:BCD} generates infinite sequences $\{\mathbf{s}^{(j)}_{\bf{z}}\}$, $\{\mathbf{s}^{(j)}_{\bf{h}}\}$ and $\{\mathbf{s}^{(j)}\}$.

\ref{theo_ele:L(block seq) converge}
By \cref{lemma:sj_in_O}, there exists  an infinite subsequence $\{j_i\} \subseteq \{j\}$ such that ${\bf s}^{(j_i)} \to \bar{\bf{s}}$ as $j_i \to \infty$. Let $\mathcal{L}^* = \mathcal{L}(\bar{\bf{s}})$.
We can easily deduce that statement \ref{theo_ele:L(block seq) converge} holds,
by the descent inequality \cref{des_L_kj} and the lower boundedness of $\{\mathcal{L}(\mathbf{s}^{(j)},\xi,\zeta,\gamma)\}$ according to \cref{gradient} \ref{lem_ele:func_lb}.

\ref{theo_ele:block seq converge}
To further prove that $\{\mathbf{s}^{(j_i)}_{\bf{z}}\}$ and $\{\mathbf{s}^{(j_i)}_{\bf{h}}\}$ also converge to $\bar{{\bf{s}}}$, it is sufficient to prove
\begin{gather}
    \lim_{i \to \infty}\|\mathbf{s}^{(j_i)} - \mathbf{s}^{(j_i)}_{\bf{z}}\| = 0, \label{eq: s_zeq0} \quad
    \lim_{i \to \infty}\|\mathbf{s}^{(j_i)} - \mathbf{s}^{(j_i)}_{\bf{h}}\| = 0.
\end{gather}

Letting $J$ go to infinity and replacing $(j)$ in \eqref{line409} by $(j_i)$, it is easy to have that $\sum_{i=1}^{\infty} \|\mathbf{s}^{(j_i)} - \mathbf{s}^{(j_i-1)}\|^2 < \infty$. Hence,
\begin{align}
    \lim_{i \to \infty}\|\mathbf{s}^{(j_i)}-\mathbf{s}^{(j_i-1)}\| = 0,
\end{align}
which together with
\begin{gather*}
    \|\mathbf{s}^{(j_i)} - \mathbf{s}^{(j_i)}_{\bf{z}}\| \leq \|\mathbf{h}^{(j_i)} - \mathbf{h}^{(j_i-1)}\| + \|\mathbf{u}^{(j_i)} - \mathbf{u}^{(j_i-1)}\|,\\
    \|\mathbf{s}^{(j_i)} - \mathbf{s}^{(j_i)}_{\bf{h}}\| \leq \|\mathbf{u}^{(j_i)} - \mathbf{u}^{(j_i-1)}\|,
\end{gather*}
implies the validity of \eqref{eq: s_zeq0}.

\end{proof}

Now we turn to show that \cref{Alg:BCD} generates a d-stationary point of problem \cref{subproblem_2cons}. For convenience, when considering the directional derivative of a function with respect to a direction and we want to emphasize the blocks of the direction, we adopt a simple expression. For example, if  $d=({d}_{\bf{z}};{d}_{h};{d}_{\bf u})$, we also write ${\cal L}'({\bf s},\xi,\zeta,\gamma;d) = {\cal L}'({\bf s},\xi,\zeta,\gamma; ({d}_{\bf{z}},{d}_{\bf{h}},
{d}_{\bf{u}}))$ instead of ${\cal L}'({\bf s},\xi,\zeta,\gamma; ({d}_{\bf{z}};{d}_{\bf{h}};
{d}_{\bf{u}})).$
\begin{lemma}
\label{lemma:regular}
    If the directional derivatives of $\mathcal{L}$
    at $\bar{\mathbf{s}} \in \Omega_{\mathcal{L}}(\Gamma)$ satisfy
   \begin{displaymath}
        \mathcal{L}'\big(\bar{\mathbf{s}},\xi,\zeta,\gamma ;(d_{\bf{z}}, 0, 0)\big) \geq 0,  \
        \mathcal{L}'\big(\bar{\mathbf{s}},\xi,\zeta,\gamma ; (0, d_{\bf{h}}, 0)\big) \geq 0,  \
        \mathcal{L}'\big(\bar{\mathbf{s}},\xi,\zeta,\gamma ; (0, 0, d_{\bf{u}})\big) \geq 0,
    \end{displaymath}
    along any $d_{\bf{z}} \in \mathbb{R}^{N_\mathbf{w} + N_\mathbf{a}}$, $d_{\bf{h}} \in \mathbb{R}^{rT}$ and $d_{\bf{u}} \in \mathbb{R}^{rT}$, then
    \begin{displaymath}                     \mathcal{L}'(\bar{\mathbf{s}},\xi,\zeta,\gamma ; {d}) \geq 0, \quad \forall\ {d} \in \mathbb{R}^{N_\mathbf{w} + N_\mathbf{a} + 2rT}.
    \end{displaymath}
\end{lemma}

{As problem \cref{subproblem_2cons} is nonsmooth nonconvex, there are many kinds of stationary points for it, such as a Fr\'echet  stationary point, a limiting stationary point, and a d-stationary point. It is known that a limiting stationary point is a Fr\'echet stationary point, and a d-stationary point is a limiting stationary point,  but not vise versa \cite{liu2021linearly}.  
%Thus a d-stationary point is the sharpest kind among the three types of stationary points.   
The theorem below guarantees that either the BCD method terminates at a d-stationary point of problem \cref{subproblem_2cons} in finite steps, or any accumulation point of the sequence generated by the BCD method is a d-stationary point of problem \cref{subproblem_2cons}. }
\begin{theorem}
\label{d-sta}
Suppose that (\ref{cond:cri_BCD}) is replaced  by
 (\ref{cond:criBCD2}) in \cref{Alg:BCD}. If there is $\bar{j}$ such that (\ref{cond:criBCD2}) holds, then  $\mathbf{s}^{(\bar{j})}$ is a d-stationary point of problem \cref{subproblem_2cons}.
   Otherwise, \cref{Alg:BCD} generates an infinite sequence $\{\mathbf{s}^{(j)}\}$ and any accumulation point of $\{\mathbf{s}^{(j)}\}$  is a d-stationary point of problem \cref{subproblem_2cons}.
\end{theorem}
\begin{proof}
If there is $\bar j$ such that  (\ref{cond:criBCD2}) holds, then {${\bf{s}}^{k-1,\bar{j}} = {{\bf s}^{k-1,\bar{j}-1}}$, i.e., ${\bf s}^{(\bar{j})} = {\bf s}^{(\bar{j}-1)}$. This, combined with (\ref{Thm4.4-1}) of Theorem  \ref{theo:BCD_blockSeq_conv}, yields that $\mathbf{s}^{(\bar{j})}_{\bf{z}}=\mathbf{s}^{(\bar{j})}_{\bf{h}}= \mathbf{s}^{(\bar{j})} = {\bf s}^{(\bar{j}-1)}.$
Thus by \cref{subpro_z_subproblem}-\cref{subpro_u_subproblem} in  \cref{Alg:BCD}, we have for any $\lambda>0$ and any ${d}_{\bf{z}} \in \mathbb{R}^{N_{\bf{w}}+N_{\bf{a}}}$, ${d}_{\bf{h}} \in \mathbb{R}^{rT}$, ${d}_{\bf{u}} \in \mathbb{R}^{rT}$, 
\begin{gather*}
    \mathcal{L}(\mathbf{s}^{(\bar{j})},\xi,\zeta,\gamma) \leq \mathcal{L}\big(\mathbf{s}^{(\bar{j})} + \lambda ({d}_{\bf{z}}, {0},{0}),\xi,\zeta,\gamma\big), \\
    \mathcal{L}(\mathbf{s}^{(\bar{j})},\xi,\zeta,\gamma) \leq \mathcal{L}\big(\mathbf{s}^{(\bar{j})} + \lambda ({0}, {d}_{\bf{h}},{0}),\xi,\zeta,\gamma\big),\\
    \mathcal{L}(\mathbf{s}^{(\bar{j})},\xi,\zeta,\gamma) \leq \mathcal{L}\big(\mathbf{s}^{(\bar{j})} + \lambda ({0}, {0},{d}_{\bf{u}}),\xi,\zeta,\gamma\big).
\end{gather*}}
{By Lemma \ref{lemma: directional derivative} and  the definition of the directional derivative, we get for any ${d}_{\bf{z}}$, ${d}_{\bf{h}}$, ${d}_{\bf{u}}$,
%\begin{align*}
\begin{eqnarray*}    & \mathcal{L}'({\mathbf{s}}^{(\bar{j})},\xi,\zeta,\gamma ; ({d}_{\bf{z}}, {0},{0})) \geq 0,\  
\mathcal{L}'({\mathbf{s}}^{(\bar{j})},\xi,\zeta,\gamma ; ( {0}, {d}_{\bf{h}}, {0})) \geq 0, \\
& \mathcal{L}'({\mathbf{s}}^{(\bar{j})},\xi,\zeta,\gamma ; ({0}, {0}, {d}_{\bf{u}})) \geq 0.
\end{eqnarray*}
%\end{align*}
 The above inequalities, along with \cref{lemma:regular}, yields that $                 \mathcal{L}'({\mathbf{s}}^{(\bar{j})},\xi,\zeta,\gamma ; {d}) \geq 0$ for any  ${d} \in \mathbb{R}^{N_\mathbf{w} + N_\mathbf{a} + 2rT}.
   $ 
    Hence, ${\bf {s}}^{(\bar{j})}$ is a d-stationary point of problem \cref{subproblem_2cons}. } 

If there is {no}  $\bar{j}$ such that (\ref{cond:criBCD2}) holds,
    then \cref{Alg:BCD} generates an infinite sequence  $\{\mathbf{s}^{(j)}\}$.
By \cref{subpro_u_subproblem}, we have
\begin{eqnarray*}
  {\cal L}({\bf s}^{(j)},\xi,\zeta,\gamma) \le {\cal L}({\bf s}^{(j)},\xi,\zeta,\gamma) + \frac{\mu}{2}\|{\bf u}^{(j)} - {\bf u}^{(j-1)}\|^2 \le {\cal L}({\bf s}_{\bf h}^{(j)}, \xi,\zeta,\gamma).
\end{eqnarray*}
Letting $j\to \infty$ in the above inequalities and using \cref{theo:BCD_blockSeq_conv} (i), we have
\begin{eqnarray*}
\lim_{j\to \infty} \|{\bf u}^{(j)} - {\bf{u}}^{(j-1)}\| = 0.
\end{eqnarray*}
By \cref{theo:BCD_blockSeq_conv} (ii), let
  $\{\mathbf{s}^{(j_i)}_{\sss z}\}$, $\{\mathbf{s}^{(j_i)}_{\sss h}\}$ and $\{\mathbf{s}^{(j_i)}\}$ be any convergent subsequences with limit $\bar{
  \bf{s}}$.
  Furthermore, {by} \cref{subpro_z_subproblem}-\cref{subpro_u_subproblem} in  \cref{Alg:BCD}, we have {for any $\lambda>0$ and any ${d}_{\bf{z}} \in \mathbb{R}^{N_{\bf{w}}+N_{\bf{a}}}$, ${d}_{\bf{h}} \in \mathbb{R}^{rT}$, ${d}_{\bf{u}} \in \mathbb{R}^{rT}$,}
\begin{gather*}
    \mathcal{L}(\mathbf{s}^{(j_i)}_{\bf{z}},\xi,\zeta,\gamma) \leq \mathcal{L}\big(\mathbf{s}^{(j_i)}_{\bf{z}} + \lambda ({d}_{\bf{z}}, {0},{0}),\xi,\zeta,\gamma\big), \\ \mathcal{L}\big(\mathbf{s}^{(j_i)}_{\bf{h}}, \xi,\zeta,\gamma\big) \leq \mathcal{L}\big(\mathbf{s}^{(j_i)}_{\bf{h}} + \lambda ({0}, {d}_{\bf{h}},{0}),\xi,\zeta,\gamma\big),\\
    \mathcal{L}(\mathbf{s}^{(j_i)},\xi,\zeta,\gamma) \leq \mathcal{L}\big(\mathbf{s}^{(j_i)} + \lambda ({0}, {0},{d}_{\bf{u}}),\xi,\zeta,\gamma\big) + \tfrac{\mu}{2}\|\mathbf{u}^{(j_i)} + \lambda d_{\bf{u}} - \mathbf{u}^{(j_i-1)}\|^2.
\end{gather*}
As $i \to \infty$, the above equality and inequalities imply that {for any $\lambda>0$} and  any ${d}_{\bf{z}}$, ${d}_{\bf{h}}$, ${d}_{\bf{u}}$,
\begin{equation*}
\begin{split}
\label{ieq:L_bars_d}
    \mathcal{L}(\bar{\mathbf{s}},\xi,\zeta,\gamma) \leq \mathcal{L}\big(\bar{\mathbf{s}} + \lambda ({d}_{\bf{z}}, {0},{0}),\xi,\zeta,\gamma\big), & \ \mathcal{L}(\bar{\mathbf{s}},\xi,\zeta,\gamma) \leq \mathcal{L}\big(\bar{\mathbf{s}} + \lambda ({0}, {d}_{\bf{h}}, {0}),\xi,\zeta,\gamma\big),\\
    \mathcal{L}(\bar{\mathbf{s}},\xi,\zeta,\gamma) \leq \mathcal{L}\big(\bar{\mathbf{s}} + \lambda ({0}, {0},& {d}_{\bf{u}}),\xi,\zeta,\gamma\big)  + \tfrac{\mu}{2} \lambda^2 \|d_{\bf{u}}\|^2.
\end{split}
\end{equation*}
%
%{
%If there is $\bar{j}$ such that (\ref{cond:criBCD2}) holds, then \eqref{ieq:L_bars_d} is also satisfied as $\bar{\mathbf{s}}=\mathbf{s}^{(\bar{j})}$ based on \cref{theo:BCD_blockSeq_conv}.
%}
%
By \cref{lemma: directional derivative} and the definition of directional derivative, it follows that
\begin{align*}
    \mathcal{L}'(\bar{\mathbf{s}},\xi,\zeta,\gamma ; ({d}_{\bf{z}}, {0},{0})) \geq 0, \
    \mathcal{L}'(\bar{\mathbf{s}},\xi,\zeta,\gamma ; ( {0}, {d}_{\bf{h}}, {0})) \geq 0, \
    \mathcal{L}'(\bar{\mathbf{s}},\xi,\zeta,\gamma ; ({0}, {0}, {d}_{\bf{u}})) \geq 0,
\end{align*}
for any ${d}_{\bf{z}}$, ${d}_{\bf{h}}$ and ${d}_{\bf{u}}$.
The above inequalities, along with \cref{lemma:regular}, yield that $\bar{\bf {s}}$ is a d-stationary point of problem \cref{subproblem_2cons}.
%{regardless of the existence of $\bar{j}$}.
\end{proof}

\subsection{Convergence analysis of \cref{Alg:ALM}}
\label{subsec:conv_alm}

By \cref{lemma:stop_iter}, the ALM in \cref{Alg:ALM} is well-defined, since Step 1 can always be fulfilled in finite steps by the BCD method in \cref{Alg:BCD}.

{It is well known that the classical ALM may converge to an infeasible point. In contrast, the following theorem guarantees that any accumulation point of the ALM in Algorithm \ref{Alg:ALM} is a feasible point. The delicate strategy for updating the penalty parameter $\gamma_k$ in Step 3 of Algorithm \ref{Alg:ALM} plays an important role in the proof of the theorem.}

\begin{theorem}
\label{feasiblilty}
Let $\left\{\mathbf{s}^{k}\right\}$ be the sequence generated by  \cref{Alg:ALM}. Then the following statements hold.
    \begin{enumerate}[label=(\roman*)]
    \item \label{theo_item:ALM_feas} $\lim_{k \rightarrow \infty}\left\|\mathbf{u}^{k} - \Psi(\mathbf{h}^{k})\mathbf{w}^{k}\right\|=0$ and  $\lim _{k \rightarrow \infty}\left\|\mathbf{h}^{k} - (\mathbf{u}^{k})_+\right\|=0$.

     \item \label{theo_item:ALM_ELP}
     There exists at least one accumulation point of $\{{\bf{s}}^k\}$, and any accumulation point is a feasible point of \eqref{prob:ror}.
    \end{enumerate}
\end{theorem}
\begin{proof}
\ref{theo_item:ALM_feas}
Let the index set
\begin{align}
    \mathcal{K} := \left\{k: \gamma_{\sss k}= \max \{\gamma_{\sss k-1} / \eta_2,\|\xi^{k}\|^{1+\eta_3}, \|\zeta^{k}\|^{1+\eta_3} \} \right\}.
\end{align}
If $\cal{K}$ is a finite set, then there exists $K\in \mathbb{N}_+$, such that for all $k>K$,
\begin{eqnarray}
\label{inq:ForProof_feasibility}
    \max\left\{
    \|{\cal C}_1({\bf{s}}^k)\|, \|{\cal C}_2({\bf{s}}^k)\|\right\}
    &\leq&
    \eta_1  \max\left\{
    \|{\cal C}_1({\bf{s}}^{k-1})\|, \|{\cal C}_2({\bf{s}}^{k-1})\|\right\} \nonumber\\
    &\leq&   \eta_1^{k-K} \max\left\{\|{\cal C}_1({\bf{s}}^K)\|, \|{\cal C}_2({\bf{s}}^K)\|\right\}.
\end{eqnarray}
Since $\eta_1 \in (0, 1)$, we get  $\lim_{k \to \infty}$ $\max\big\{\|\mathbf{u}^{k} - \Psi(\mathbf{h}^{k})\mathbf{w}^{k}\|, \|\mathbf{h}^{k} - (\mathbf{u}^{k})_{+}\|\big\} = 0.$
The statement (i) can thus be proved for this case.

Otherwise, $\cal{K}$ is an infinite set. Then for those $k-1\in {\cal{K}}$,
$$\max\left\{\tfrac{\left\|\xi^{k-1}\right\|}{\gamma_{\sss k-1}},\tfrac{\left\|\zeta^{k-1}\right\|}{\gamma_{\sss k-1}}\right\}  \leq \left(\gamma_{\sss k-1}\right)^{\tfrac{-\eta_3}{1+\eta_3}},\
\max\left\{\tfrac{\left\|\xi^{k-1}\right\|^2}{\gamma_{\sss k-1}}, \tfrac{\left\|\zeta^{k-1}\right\|^2}{\gamma_{\sss k-1}}\right\} \leq \left(\gamma_{\sss k-1}\right)^{\tfrac{1-\eta_3}{1+\eta_3}}.$$
The above inequalities, together with $\eta_3 > 1$
yields that
{\small{\begin{gather}
\label{lim_K:xixi2/gamma=0}
    \lim_{k \to \infty, k-1 \in \mathcal{K}} \  \max\left\{\frac{\left\|\xi^{k-1}\right\|}{\gamma_{\sss k-1}},  \frac{\left\|\zeta^{k-1}\right\|}{\gamma_{\sss k-1}},\frac{\left\|\xi^{k-1}\right\|^2}{\gamma_{\sss k-1}}, \frac{\left\|\zeta^{k-1}\right\|^2}{\gamma_{\sss k-1}}  \right\} = 0.
\end{gather}}}Recalling  \eqref{alm function for 2constraints}, and
employing condition \cref{cond:func_value bounded} and Step 1 of \cref{Alg:BCD}, we have
\begin{equation}
\begin{split}
\label{leq:ubound_feas}
    0 \leq &\big\|\mathbf{u}^{k}-\Psi(\mathbf{h}^{k})\mathbf{w}^{k}+\tfrac{\xi^{k-1}}{\gamma_{\sss k-1}}\big\|^2 + \big\|\mathbf{h}^{k}-(\mathbf{u}^{k})_{+}+\tfrac{\zeta^{k-1}}{\gamma_{\sss k-1}}\big\|^2 \\
    \leq &\tfrac{2}{\gamma_{\sss k-1}}\big(\Gamma - \mathcal{R} (\mathbf{s}^{k})\big) + \big(\tfrac{\|\xi^{k-1}\|}{\gamma_{\sss k-1}}\big)^2 + \big(\tfrac{\|\zeta^{k-1}\|}{\gamma_{\sss k-1}}\big)^2.
\end{split}
\end{equation}
Then by \cref{lim_K:xixi2/gamma=0} and the lower boundedness of $\left\{\mathcal{R} (\mathbf{s}^{k}) \right\}$, we have
\begin{align}
\label{eq:lim con in K}
    \lim_{k \to \infty, k-1 \in \mathcal{K}} \ \|\mathbf{u}^{k}-\Psi(\mathbf{h}^{k})\mathbf{w}^{k}\| = 0 \quad \text{and} \quad \lim_{k \to \infty, k-1 \in \mathcal{K}} \ \|\mathbf{h}^{k}-(\mathbf{u}^{k})_{+}\| = 0.
\end{align}

To extend the results in \cref{eq:lim con in K} to any $k > K$, let $l_k$ denote the largest element in $\mathcal{K}$ satisfying $l_k < k$. If $l_k = k-1$, the limitations are the same as \cref{eq:lim con in K}. If $l_k < k-1$, let us define an index set ${\mathcal{I}}_{k} := \left\{i : l_k<i < k \right\}$. The updating rule for the penalty parameter, as stated in \cref{update:gamma}, implies that $\gamma_{i} = \gamma_{l_k}$. This, combined with the updating rules for the Lagrangian multipliers, yields that for all $i \in {\mathcal{I}}_k$, the following holds:
\begin{align}
\frac{\|\xi^{i}\|}{\gamma_{i}} &= \frac{\|\xi^{i}\|}{\gamma_{i-1}} \leq \frac{\|\xi^{i-1}\|}{\gamma_{i-1}} + \left\|\mathbf{u}^{i} - \Psi(\mathbf{h}^{i})\mathbf{w}^{i}\right\|, \label{inq:xi1/gammaLEQpre}\\
\frac{\|\zeta^{i}\|}{\gamma_{i}} &= \frac{\|\zeta^{i}\|}{\gamma_{i-1}} \leq \frac{\|\zeta^{i-1}\|}{\gamma_{i-1}} + \left\|\mathbf{h}^{i} - (\mathbf{u}^{i})_+\right\| \label{inq:xi2/gammaLEQpre}.
\end{align}
Summing up inequalities \cref{inq:xi1/gammaLEQpre} and \cref{inq:xi2/gammaLEQpre} for every $i \in {\mathcal{I}}_k$, we have
\begin{align}
 \frac{\|\xi^{k-1}\|}{\gamma_{\sss k-1}} &\leq \frac{\|\xi^{l_k}\|}{\gamma_{l_k}} + \sum_{i=1}^{k-l_k-1}\left\|\mathbf{u}^{k-i} - \Psi(\mathbf{h}^{k-i})\mathbf{w}^{k-i}\right\|, \label{inq:xi/gamma_lk}\\
 \frac{\|\zeta^{k-1}\|}{\gamma_{\sss k-1}} &\leq \frac{\|\zeta^{l_k}\|}{\gamma_{l_k}} + \sum_{i=1}^{k-l_k-1}\left\|\mathbf{h}^{k-i} - (\mathbf{u}^{k-i})_+\right\|.\label{inq:zeta/gamma_lk}
\end{align}
By the updating rule of $\gamma_{\sss k}$ in \eqref{cond:update_gamma}, \cref{inq:xi/gamma_lk} and \cref{inq:zeta/gamma_lk}, we obtain
{\small
\begin{align*}
 \frac{\|\xi^{k-1}\|}{\gamma_{\sss k-1}}
 \leq \frac{\|\xi^{l_k}\|}{\gamma_{l_k}}
 + \frac{\eta_1}
 {1-\eta_1}\max\left\{\left\|\mathbf{u}^{l_k+1} - \Psi(\mathbf{h}^{l_k+1})\mathbf{w}^{l_k+1}\right\|, \left\|\mathbf{h}^{l_k+1} - (\mathbf{u}^{l_k+1})_+\right\|\right\}, \notag \\
 \frac{\|\zeta^{k-1}\|}{\gamma_{\sss k-1}} \leq \frac{\|\zeta^{l_k}\|}{\gamma_{l_k}} + \frac{\eta_1
 }{1-\eta_1}\max\left\{\left\|\mathbf{u}^{l_k+1} - \Psi(\mathbf{h}^{l_k+1})\mathbf{w}^{l_k+1}\right\|, \left\|\mathbf{h}^{l_k+1} - (\mathbf{u}^{l_k+1})_+\right\|\right\}.
\end{align*}}This, together with \cref{lim_K:xixi2/gamma=0}, \cref{eq:lim con in K} and $\eta_1 \in (0, 1)$, yields that
\begin{align}
    \label{lim:xi/gamma=0}
    \lim_{k \to \infty} \  \frac{\left\|\xi^{k-1}\right\|}{\gamma_{\sss k-1}} = 0, \quad \lim_{k \to \infty} \  \frac{\left\|\zeta^{k-1}\right\|}{\gamma_{\sss k-1}} = 0.
\end{align}
By the inequality \cref{leq:ubound_feas} and nondecreasing sequence $\{\gamma_{\sss k}\}$, we conclude that
\begin{align}
    \lim_{k \to \infty} \ \|\mathbf{u}^{k}-\Psi(\mathbf{h}^{k})\mathbf{w}^{k}\| = 0, \quad
    \lim_{k \to \infty} \ \|\mathbf{h}^{k}-(\mathbf{u}^{k})_{+}\| = 0,
\end{align}
using the same manner for showing \cref{eq:lim con in K}.

\ref{theo_item:ALM_ELP}
When $\mathcal{K}$ is finite, there exists a constant $K$ such that $\gamma_{\sss k-1} = \gamma_{\sss K}$ for those $k > K$. Then, we turn to consider the boundedness of $\{\xi^{k-1}\}$ and $\{\zeta^{k-1}\}$.
Summing up \cref{lar_mul_update} for those $k > K$,  and using  \eqref{cond:update_gamma}, we find
\begin{eqnarray*}
  & & \max\{\{\|{\xi^{k-1}}\|, \|\zeta^{k-1}\|\} \\
   & & \leq \max\{\|\xi^{K}\|,\|\zeta^{K}\|\}
   + \frac{\eta_1 {\gamma_{\sss K}}}
   {1-\eta_1}
   \max\left\{\left\|\mathbf{u}^{K} - \Psi(\mathbf{h}^{K})\mathbf{w}^{K}\right\|, \left\|\mathbf{h}^{K} - (\mathbf{u}^{K})_+\right\|\right\}.
\end{eqnarray*}
From the above, the boundedness of $\{\xi^{k-1}\}$ and $\{\zeta^{k-1}\}$ are thus proved. Together with $\gamma_{\sss k-1} = \gamma_{\sss K}$ for those $k > K$, we can further deduce that $\|\xi^{k-1}\|^2/\gamma_{\sss k-1}$ and $\|\zeta^{k-1}\|^2/\gamma_{\sss k-1}$ are bounded for those $k \in \mathbb{N}_{+}$. 

When the set $\mathcal{K}$ is infinite, by \cref{lim_K:xixi2/gamma=0} we know that $\|\xi^{k-1}\|^2/\gamma_{\sss k-1}$ and $\|\zeta^{k-1}\|^2/\gamma_{\sss k-1}$ are bounded for $k-1 \in {\cal K}$. {Therefore, no matter $\mathcal{K}$ is finite or infinite, $\|\xi^{k-1}\|^2/\gamma_{\sss k-1}$ and $\|\zeta^{k-1}\|^2/\gamma_{\sss k-1}$ are bounded for $k-1 \in {\cal K}$.}

Moreover, we can deduce the following inequality according to the expression of ${\cal{L}}_{k-1}$,
condition \cref{cond:func_value bounded}, and ${\bf s}^k = {\bf s}^{k-1,j}$:
\begin{equation}
\begin{split}
\label{ieq:bounded R(s)}
    &\mathcal{R} (\mathbf{s}^{k}) + \frac{\gamma_{\sss k-1}}{2}\left\|\mathbf{u}^{k}-\Psi(\mathbf{h}^{k})\mathbf{w}^{k}+\frac{\xi^{k-1}}{\gamma_{\sss k-1}}\right\|^2  + \frac{\gamma_{\sss k-1}}{2}\left\|\mathbf{h}^{k}-(\mathbf{u}^{k})_{+}+\frac{\zeta^{k-1}}{\gamma_{\sss k-1}}\right\|^2\\
   \leq &\ \Gamma + \frac{\|\xi^{k-1}\|^2}{2\gamma_{\sss k-1}} + \frac{\|\zeta^{k-1}\|^2}{2\gamma_{\sss k-1}}.
\end{split}
\end{equation}
The above inequality, along with the boundedness of {$\left\{{\|\xi^{k-1}\|^2}/{\gamma_{k-1}}\right\}_{k-1 \in {\cal K}}$} {and} {$\left\{{\|\zeta^{k-1}\|^2}/{\gamma_{k-1}}\right\}_{k-1 \in {\cal K}}$}, yields the boundedness of $\{\mathbf{s}^{k}\}_{k-1\in {\cal K}}$ by the same manner in \cref{gradient} \ref{lem_ele:levelset_pro}.
Hence there exists at least one accumulation point of $\{{\bf s}^k\}$.

Any accumulation point is a feasible point of \cref{prob:ror}, which can be derived immediately by (i), because of the continuity of the functions in the constraints of \cref{prob:ror}.
\end{proof}

Below we show the main convergence result of the ALM.
\begin{theorem}
\label{ALM-convergence}
    Every accumulation point of $\{\mathbf{s}^{k}\}$ generated by \cref{Alg:ALM} is a KKT point of problem \cref{prob:ror}.
\end{theorem}
\begin{proof}
Let $\{\mathbf{s}^{k_{i}}\}$ be a subsequence of $\{\mathbf{s}^k\}$ converging to $\bar{\mathbf{s}}$.
Then $\bar{s} \in \mathcal{F}$
by  %\ref{theo_item:ALM_feas} (ii) in 
\cref{feasiblilty}.
We claim that
\begin{equation}
\begin{split}
\label{eq:AL_update}
    &\partial
    %_{\mathbf{s}} 
\mathcal{L}\big(\mathbf{s}^{k_i}, \xi^{k_i-1}, \zeta^{k_i-1},\gamma_{\sss k_i-1}\big) \\
    =\ & \nabla\mathcal{R}(\mathbf{s}^{k_i}) + \nabla_\mathbf{s}\Big(\langle  \xi^{k_i-1},  {\mathbf{u}^{k_i}}-{\Psi(\mathbf{h}^{k_i})}\mathbf{w}^{k_i}\rangle + \tfrac{\gamma_{\sss k_i-1}}{2}\big\|{\mathbf{u}^{k_i}}-{\Psi(\mathbf{h}^{k_i})}\mathbf{w}^{k_i}\big\|^2\Big) \\
    \ & \quad + \partial_\mathbf{s} \Big(\langle  \zeta^{k_i-1},  {\mathbf{h}^{k_i}}-(\mathbf{u}^{k_i})_{+}\rangle + \tfrac{\gamma_{\sss k_i-1}}{2}\big\|{\mathbf{h}^{k_i}}-(\mathbf{u}^{k_i})_{+}\big\|^2\Big)\\
    =\ & \nabla\mathcal{R}(\mathbf{s}^{k_i}) + J{\mathcal{C}_1}({\mathbf{s}^{k_i}})^{\top} \xi^{k_i} + \partial\Big((\zeta^{k_i})^{\top} \mathcal{C}_2(\mathbf{s}^{k_i})\Big),
\end{split}
\end{equation}
where $\mathcal{C}_{1}$ and $\mathcal{C}_{2}$ are defined in \cref{func:cons}.

First, by employing \cref{lar_mul_update} and by direct computation, we have
\begin{equation}
\begin{split}
\label{KKT_grad_smooth}
    &\nabla_\mathbf{s} \big(\langle  \xi^{k_i-1},  {\mathbf{u}^{k_i}}-{\Psi(\mathbf{h}^{k_i})}\mathbf{w}^{k_i}\rangle + \tfrac{\gamma_{\sss k_i-1}}{2}\big\|{\mathbf{u}^{k_i}}-{\Psi(\mathbf{h}^{k_i})}\mathbf{w}^{k_i}\big\|^2\big)\\
    =\ &J{\mathcal{C}_1}({\mathbf{s}^{k_i}})^{\top} \big(\xi^{k_i-1}+\gamma_{\sss k_i-1}({\mathbf{u}^{k_i}}-{\Psi(\mathbf{h}^{k_i})}\mathbf{w}^{k_i})\big)
    =J{\mathcal{C}_1}({\mathbf{s}^{k_i}})^{\top} \xi^{k_i}.
\end{split}
\end{equation}
Then, it remains to verify that
\begin{align}
\label{sub_Q_s}
    \partial_\mathbf{s} (\langle  \zeta^{k_i-1},  {\mathbf{h}^{k_i}}-(\mathbf{u}^{k_i})_{+}\rangle + \tfrac{\gamma_{\sss k_i-1}}{2}\|{\mathbf{h}^{k_i}}-(\mathbf{u}^{k_i})_{+}\|^2) = \partial\Big((\zeta^{k_i})^{\top} \mathcal{C}_2(\mathbf{s}^{k_i})\Big).
\end{align}
To verify \cref{sub_Q_s}, it can be divided into the subdifferential associated with $\mathbf{h}$ and $\mathbf{u}$. We first prove that \cref{sub_Q_s} is satisfied associated with $\mathbf{h}$. By simple computation,
\begin{equation}
\begin{split}
\label{subdiff_q_h}
    &\nabla_{\mathbf{h}}\big(\langle  \zeta^{k_i-1},  {\mathbf{h}^{k_i}}-(\mathbf{u}^{k_i})_{+}\rangle +\tfrac{\gamma_{\sss k_i-1}}{2}\big\|{\mathbf{h}^{k_i}}-(\mathbf{u}^{k_i})_{+}\big\|^2\big)\\
    =\ &J_{\mathbf{h}}\mathcal{C}_2(\mathbf{z}^{k_i}, \mathbf{h}^{k_i}, \mathbf{u}^{k_i})^{\top} \big(\zeta^{k_i-1} + \gamma_{\sss k_i-1}({\mathbf{h}^{k_i}}-(\mathbf{u}^{k_i})_{+})\big)\\
    =\ &J_{\mathbf{h}}\mathcal{C}_2(\mathbf{z}^{k_i}, \mathbf{h}^{k_i}, \mathbf{u}^{k_i})^{\top}\zeta^{k_i} = \nabla_{\mathbf{h}} \big(\langle\zeta^{k_i}, {\mathbf{h}^{k_i}}-(\mathbf{u}^{k_i})_{+}\rangle\big).
\end{split}
\end{equation}

Then we prove that \cref{sub_Q_s} is satisfied associated with $\mathbf{u}$, which can be replaced by proving $rT$ one dimensional equations with the similar structure as follows:
\begin{align}
\label{subdiff_q_u_new}
    &\partial_{\mathbf{u}_{j}}\Big(\zeta^{k_i-1}_{j}(\mathbf{h}_{j}^{k_i} - (\mathbf{u}_{j}^{k_i})_{+}) + \tfrac{\gamma_{\sss k_i-1}}{2}({\mathbf{h}_{j}^{k_i}}-(\mathbf{u}_{j}^{k_i})_{+})^2\Big) = \partial_{\mathbf{u}_{j}} \Big(\zeta^{k_i}_{j}(\mathbf{h}_{j}^{k_i} - (\mathbf{u}_{j}^{k_i})_{+})\Big),
\end{align}
where $j = 1, 2, ..., rT$.
When $\mathbf{u}_{j}^{k_i} \neq 0$, equation \cref{subdiff_q_u_new} can be easily deduced by the same proof method as in \cref{subdiff_q_h}. When $\mathbf{u}_{j}^{k_i} = 0$, the validity of (\ref{subdiff_q_u_new}) can be proved as follows:
\begin{equation}
\begin{split}
\label{subdiff_q_u_left}
    &\partial_{\mathbf{u}_{j}}\Big(\zeta^{k_i-1}_{j}(\mathbf{h}_{j}^{k_i} - (\mathbf{u}_{j}^{k_i})_{+}) + \tfrac{\gamma_{\sss k_i-1}}{2}({\mathbf{h}_{j}^{k_i}}-(\mathbf{u}_{j}^{k_i})_{+})^2\Big)\\
    =\ &\begin{cases}
    \{0, -\zeta^{k_i-1}_{j} - \gamma_{\sss k_i-1}(\mathbf{h}_j^{k_i} - \mathbf{u}_j^{k_i})\}, &\text{if} \ \gamma_{\sss k_i-1}\mathbf{h}_j^{k_i}+\zeta^{k_i-1}_{j} \geq 0, \\
    \big[0, -\zeta^{k_i-1}_{j} - \gamma_{\sss k_i-1}(\mathbf{h}_j^{k_i} - \mathbf{u}_j^{k_i})\big], &\text{if} \ \gamma_{\sss k_i-1}\mathbf{h}_j^{k_i}+\zeta^{k_i-1}_{j} < 0,
    \end{cases}\\
    =\ &\begin{cases}
    \{0, -\zeta^{k_i}_{j}\}, &\text{if} \ \zeta^{k_i}_{j} \geq 0, \\
    \big[0, -\zeta^{k_i}_{j}\big], &\text{if} \ \zeta^{k_i}_{j} < 0,
  \end{cases}\\
  =\ &\partial_{\mathbf{u}_{j}} \Big(\zeta^{k_i}_{j}(\mathbf{h}_{j}^{k_i} - (\mathbf{u}_{j}^{k_i})_{+})\Big).
\end{split}
\end{equation}
Combining \cref{KKT_grad_smooth} and \cref{sub_Q_s} yields the validity of \cref{eq:AL_update}.

Up to now, we have verified that equation \cref{eq:AL_update} holds. Thus, there exists a sequence $\{\varsigma^{k_i}\}$ satisfying $\|\varsigma^{k_i}\| \leq \epsilon^{k_i}$ such that
\begin{align}
\label{L_function_ki}
    \varsigma^{k_i} \in \nabla\mathcal{R}(\mathbf{s}^{k_i}) + J{\mathcal{C}_1}({\mathbf{s}^{k_i}})^{\top} \xi^{k_i} + \partial\Big((\zeta^{k_i})^{\top} \mathcal{C}_2(\mathbf{s}^{k_i})\Big).
\end{align}
However, the boundedness of $\{\xi^{k_i}\}$ and $\{\zeta^{k_i}\}$ in \cref{L_function_ki} are still not sure. Define $\varrho^{i}$ $=\max\{\|\xi^{k_i}\|_{\infty}, \|\zeta^{k_i}\|_{\infty}\}$ and assume that $\{\varrho^{i}\}$ is unbounded. It is trivial to have bounded sequences $\{\xi^{k_i}/\varrho^{i}\}$ and $\{\zeta^{k_i}/\varrho^{i}\}$ according to the definition of $\varrho^{i}$. Without loss of generality, we assume $\{\xi^{k_i}/\varrho^{i}\}$ $\to$ $\bar{\xi}$ and $\{\zeta^{k_i}/\varrho^{i}\} \to \bar{\zeta}$ as $k \to \infty$ and thus have
\begin{align}
\label{max_norm_1}
    \max\{\|\bar{\xi}\|_{\infty}, \|\bar{\zeta}\|_{\infty}\} = 1.
\end{align}
Dividing by $\varrho^{i}$ on both sides of \cref{L_function_ki} and taking $i \to \infty$, and using the facts that the limiting subdifferential is outer semicontinuous \cite[Proposition 8.7]{rockafellar2009variational}, and $\varsigma^{k_i} \to 0$ as $i \to \infty$,  we derive that
\begin{align}
\label{zero in sub L}
    0 \in J{\mathcal{C}_1}(\bar{\mathbf{s}})^{\top} \bar \xi + \partial\Big({\bar{\zeta}}^{\top} \mathcal{C}_2(\bar{\mathbf{s}})\Big).
\end{align}
 Combining \cref{zero in sub L} and \cref{lemma: NNAMCQ} yields that $\bar{\xi} = 0$ and $\bar{\zeta} = 0$, which contradicts \cref{max_norm_1}.
Therefore, $\{\xi^{k_i}\}$ and $\{\zeta^{k_i}\}$ are bounded. Without loss of generality, we assume $\{\xi^{k_i}\} \to \bar{\xi}$ and $\{\zeta^{k_i}\} \to \bar{\zeta}$ as $i \to \infty$. Letting $i \to \infty$ in \cref{L_function_ki}, we obtain
\begin{displaymath}
    0 \in \ \nabla\mathcal{R}(\bar{\mathbf{s}}) + J{\mathcal{C}_1}(\bar{\mathbf{s}})^{\top} \bar{\xi} + \partial\Big({\bar{\zeta}}^{\top} \mathcal{C}_2(\bar{\mathbf{s}})\Big).
\end{displaymath}
Therefore, $\bar{\mathbf{s}}$ is a KKT point of problem \cref{prob:ror}.
\end{proof}

\subsection{{Extensions to other activation functions}}
\label{subsec:exten_act_model}
{Now we discuss the possible extensions of our methods, algorithms and theoretical analysis, using other activation functions rather than the ReLU.}

{First, we claim that the activation functions are required to be {locally Lipschitz continuous}, because the {locally Lipschitz} continuity of the ReLU function is used in $L_2(\xi, \zeta, \gamma, {\hat r})$ of Lemma \ref{lem3.2} that depends on the Lipschitz constant of the ReLU function on a compact set.}
{Then we find that in the analysis above only the following two places make use of the special piecewise linear structure of the ReLU  function: 
\begin{itemize}
\item[P1.] Explicit formula for ${\bf{u}}^{k-1,j}$ in \eqref{subpro_u_subproblem} of the BCD method in Algorithm \ref{Alg:BCD}.
\item[P2.] Equations (\ref{subdiff_q_u_left}) for proving (\ref{subdiff_q_u_new})
in the proof of Theorem \ref{ALM-convergence}.
\end{itemize}}
%Below we give several examples of the activation functions that can be extended. 
{For P1, even if the activation function in \eqref{ReLU} %\eqref{subpro_u_subproblem} 
is replaced by others, the objective function  in problem \eqref{subpro_u_subproblem} can still be separated into $rT$ one-dimensional functions, which is obtained by substituting the ReLU function $(u)_+$ in \eqref{subpro_u_ele} by a more general activation function.}
%{Although our theory in this paper is developed for nonsmooth activation functions, it can also allow smooth activation functions at the expense of losing the explicit formula of a global solution of (\ref{gen_subpro_u_ele}).  
{For P2, if an arbitrary smooth activation function is considered, then (4.29) holds obviously because the limiting subdifferential reduces to the gradient. Below we illustrate in detail the leaky ReLU and the ELU activation functions as examples for extensions. 
%We also provide the extension to the sigmoid activation function  in Appendix \ref{extension-sigmoid},  since the analysis for the sigmoid function is similar to that of the ELU function.  
It is clear that the expression of $L_{2}(\xi, \zeta, \gamma, \hat{r})$ in \cref{lem3.2} remains unchanged for the two activation functions because they all have Lipschitz constant 1, the same as that of the ReLU.  }

\subsubsection*{Extension to the leaky ReLU}
{

Let us replace the ReLU activation function $\sigma(u)=(u)_+$
%in \eqref{prob:cor} and \eqref{prob:ror} 
with the leaky ReLU activation function defined by 
\begin{align*}
    \sigma_{\rm{lRe}}(u): = \max\{u, \varpi u\}, 
\end{align*}
where $\varpi\in (0,1)$ is a fixed parameter. The leaky ReLU activation function has been widely used in recent years. %\cite{}.
With regard to P1, by direct computation, a closed-form global solution of 
\begin{equation}
\label{gen_subpro_u_ele}
    \mathop{\min}_{u\in \mathbb{R}} \ \varphi_{\rm{lRe}}(u) := \tfrac{\gamma}{2}(u - \theta_{1})^2 + \tfrac{\gamma}{2}(\theta_{2}-\sigma_{\rm{lRe}}(u))^2 + \tfrac{\mu}{2} (u - \theta_{3})^2 + \lambda_6 u^2,
\end{equation}
%\eqref{gen_subpro_u_ele} 
can be obtained similarly using the procedures  for ReLU in \eqref{sub_u_ele_u_all}-\eqref{subpro_u_ele_u_minus}, except that the expression $u^{-}$ of \eqref{subpro_u_ele_u_minus} changes to 
\begin{eqnarray}
\label{subpro_relu_ele_u_min}
    u^{-} = \left\{
    \begin{array}{ll}
\dfrac{\gamma\theta_{1} + \gamma \varpi \theta_2+ \mu \theta_{3}}{\gamma + \gamma \varpi^2 + 2\lambda_6 + \mu}, & \text{if} \ \gamma\theta_{1} + \mu \theta_{3}< 0,  \\
    0, & \text{otherwise}.
    \end{array}
\right.
\end{eqnarray}
%As for the adjustment of \cref{ALM-convergence}, only 
For P2, \eqref{subdiff_q_u_left} is modified as follows: when $\mathbf{u}_{j}^{k_i}=0$,
\begin{equation}
\begin{split}
\label{prelu_subdiff_q_u_left}
    &\partial_{\mathbf{u}_{j}}\Big(\zeta^{k_i-1}_{j}(\mathbf{h}_{j}^{k_i} -\sigma_{\rm{lRe}}(\mathbf{u}_{j}^{k_i})) + \tfrac{\gamma_{\sss k_i-1}}{2}({\mathbf{h}_{j}^{k_i}}-\sigma_{\rm{lRe}}(\mathbf{u}_{j}^{k_i}))^2\Big)\\
    =\ &\begin{cases}
    \{-\varpi \zeta^{k_i}_{j}, -\zeta^{k_i-1}_{j} - \gamma_{\sss k_i-1}(\mathbf{h}_j^{k_i} - \mathbf{u}_j^{k_i})\}, &\text{if} \ \gamma_{\sss k_i-1}\mathbf{h}_j^{k_i}+\zeta^{k_i-1}_{j} \geq 0, \\
    \big[-\varpi \zeta^{k_i}_{j}, -\zeta^{k_i-1}_{j}-\gamma_{\sss k_i-1}(\mathbf{h}_j^{k_i} - \mathbf{u}_j^{k_i})\big], &\text{if} \ \gamma_{\sss k_i-1}\mathbf{h}_j^{k_i}+\zeta^{k_i-1}_{j} < 0,
    \end{cases}\\
    =\ &\begin{cases}
    \{-\varpi \zeta^{k_i}_{j}, -\zeta^{k_i}_{j}\}, &\text{if} \ \zeta^{k_i}_{j} \geq 0, \\
    \big[-\varpi \zeta^{k_i}_{j}, -\zeta^{k_i}_{j}\big], &\text{if} \ \zeta^{k_i}_{j} < 0,
  \end{cases}\\
  =\ &\partial_{\mathbf{u}_{j}} \Big(\zeta^{k_i}_{j}(\mathbf{h}_{j}^{k_i} - \sigma_{\rm{lRe}}(\mathbf{u}_{j}^{k_i}))\Big).
\end{split}
\end{equation}
}

\subsubsection*{Extension to the  ELU}
{ Let us replace the ReLU activation function with the convex and smooth activation function ELU defined by 
\begin{align*}
\sigma_{\rm{ELU}}(u) := 
\begin{cases} 
u & \text{if } u \ge 0, \\
e^{u} - 1 & \text{if } u < 0.
\end{cases}
\end{align*}
When $u\ge 0$, the ELU activation function is the same as the ReLU function. Thus for P1, the solution of \eqref{gen_subpro_u_ele} can be obtained similarly  as the ReLU case, except that we do not have the  explicit formula of $u^{-}$, which is a global solution of 
\begin{align}
\label{sub_ele_elu}
    \min_{u \in (-\infty, 0]} \ \varphi_{ \text{\tiny ELU}}(u) = \tfrac{\gamma}{2}(u - \theta_{1})^2 + \tfrac{\gamma}{2}(\theta_{2}-(e^u - 1))^2 + \tfrac{\mu}{2} (u - \theta_{3})^2 + \lambda_6 u^2,
\end{align}
 due to the presence of the exponential function in the ELU activation function. 

Now we illustrate that $u^{-}$ 
%({\ref{gen_subpro_u_ele}}) 
can be obtained numerically through solving several one-dimensional minimization problems.
First, using the formula of $\varphi_{\rm{ELU}}(u)$ and $\varphi_{\rm{ELU}}(u) \to +\infty$ as $u\to -\infty$, 
we can easily find a lower bound $\underline{u}<0$ such that   (\ref{sub_ele_elu}) is equivalent to 
\begin{equation}
\label{bounded_ELU}
\min_{u \in [\underline{u},0]} \varphi_{\rm{ELU}}(u).
\end{equation}
The objective function $\varphi_{ \text{\tiny ELU}}(u)$ is smooth on $(-\infty, 0]$. We thus calculate the second-order derivative of $\varphi_{ \text{\tiny ELU}}(u)$ as
\begin{align}
\label{sub_ele_elu_so}
    \varphi_{ \text{\tiny ELU}}''(u) = 2\gamma e^{2u} - \gamma(\theta_2 + 1)e^{u} + \mu + \gamma + 2\lambda_{6}.
\end{align}
Letting ${z}=e^{u}$, \eqref{sub_ele_elu_so} can thus be represented as 
\begin{align}
    \psi_{\text{\tiny ELU}}(z) := 2\gamma z^2 - \gamma(\theta_2 + 1)z + \mu + \gamma + 2\lambda_{6}, 
    %\ z \in (0, 1]
\end{align}
which is a quadratic function. {Hence there are at most two distinct roots of $$\psi_{\rm{ELU}}(z) = 0,$$ and consequently at most two distinct roots for $\varphi''(u)=0$ on  $[\underline{u},0]$. Hence the convexity and concavity can only be changed at most three times in $[\underline{u},0]$. That is, we can divide $[\underline{u},0]$ into at most three closed intervals, and in each interval $\varphi_{\rm{ELU}}$ is either convex or concave.} {We  minimize the objective function $\varphi_{\rm{ELU}}$ in each of those intervals that $\varphi_{\rm{ELU}}$ is 
convex, and obtain a global solution in each interval numerically. Then, we  select a point among those solutions, $0$, and $\underline{u}$   
 that has the minimal objective value.  This point is a global solution of  (\ref{sub_ele_elu}).}

%as
%illustrated in detail  in \cref{appendix:u_minus_for_ELU}.

%For the ELU activation function, the constant $L_{2}(\xi, \zeta, \gamma, \hat{r})$ in \cref{lem3.2} does not need to adjust, because its Lipschitz constant is still 1, the same as that of the ReLU function. 
%Furthermore, \eqref{subdiff_q_u} and \eqref{subdiff_q_u_left} in the proof of \cref{ALM-convergence} are adjusted to
%\begin{align*}
%\label{subdiff_q_u}
%    &\nabla_{\mathbf{u}}\Big(\big\langle  \zeta^{k_i-1},  {\mathbf{h}^{k_i}}-\sigma(\mathbf{u}^{k_i})\big\rangle +\tfrac{\gamma_{\sss k_i-1}}{2}\big\|{\mathbf{h}^{k_i}}-\sigma(\mathbf{u}^{k_i})\big\|^2\Big)\\
%    =\ &J_{\mathbf{u}}\mathcal{C}_2(\mathbf{z}^{k_i}, \mathbf{h}^{k_i}, \mathbf{u}^{k_i})^{\top} \Big(\zeta^{k_i-1} + \gamma_{\sss k_i-1}\big({\mathbf{h}^{k_i}}-\sigma(\mathbf{u}^{k_i})\big)\Big)\\
%    =\ &J_{\mathbf{u}}\mathcal{C}_2(\mathbf{z}^{k_i}, \mathbf{h}^{k_i}, \mathbf{u}^{k_i})^{\top}\zeta^{k_i} = \nabla_{\mathbf{u}} \Big(\big\langle\zeta^{k_i}, {\mathbf{h}^{k_i}}-\sigma(\mathbf{u}^{k_i})\big\rangle\Big).
%\end{align*}
}

%Moreover, the constant $L_{2}(\xi, \zeta, \gamma, \hat{r})$ for the sigmoid activation function is  the same as that of the  ReLU activation function in \cref{lem3.2}, because the Lipschitz constant of the sigmoid activation function is also  1, the same as that of the ReLU function.
%}

%\subsection{Extension to other RNN architectures}
%先LSTM能不能过去，需要调整哪里；如果过不去，修改题目成Discussions on possible extensions.

\section{Numerical experiments}
\label{sec:ne}
We employ a real world dataset, \textbf{Volatility of S$\mathit{\&}$P index}, and synthetic datasets to evaluate the effectiveness of our reformulation (\ref{prob:ror}) and \cref{Alg:ALM} with \cref{Alg:BCD}.  To be specific, we first use RNNs with unknown weighted matrices to model these sequential datasets, and then utilize the ALM with the BCD method to train RNNs. After the training process, we can predict future values of these sequential datasets using the trained RNNs.

The numerical experiments consist of two components. The first part involves assessing whether the outputs generated by the ALM adhere to the constraints in \cref{prob:ror}. The second part is to compare the training and forecasting performance of the ALM with state-of-the-art gradient descent-based algorithms (GDs).
{All the numerical experiments were conducted using Python 3.9.8. 
For the datasets, \textbf{Synthetic dataset ($T = 10$)} and \textbf{Volatility of S$\&$P index}, experiments were carried out on a desktop (Windows 10 with 2.90 GHz Inter Core i7-10700 CPU and 32GB RAM). Additionally, experiments for \textbf{Synthetic dataset ($T = 500$)} were implemented on a server} 
{(2 Intel Xeon Gold 6248R CPUs and 768GB RAM) at the high-performance servers of the Department of Applied Mathematics, the Hong Kong Polytechnic University.
}

\subsection{Datasets}
\label{subsec:datasets}
The process of generating synthetic datasets is as follows. We randomly generate the weighted matrices $\hat{A}$, $\hat{W}$, $\hat{V}$, $\hat{b}$ and $\hat{c}$, noise $\tilde e_{t}$, $t = 1, 2, ..., T$ and the input data $X = (x_1,..., x_T)$ with some distributions. Then we calculate the output data $Y = (y_1, ..., y_T)$ by $y_t=(\hat{A} (\hat{W}(...(\hat{V}{x}_{1}+\hat{b})_{+}...)+\hat{V}{x}_{t}+\hat{b})_{+}+\hat{c})+ \tilde{e}_{t}$
for $t \in [T]$. In the numerical experiments, we generate two synthetic datasets with $T = 10$ and $T = 500$. The detailed information of the two synthetic datasets is listed in \cref{tab:syn_inf}. Moreover, the ratio of splitting for the training and test sets is about $9:1$.
\begin{table}[tbhp]
\footnotesize
\caption{Synthetic datasets}
\renewcommand{\arraystretch}{1.3}
\begin{center}
\label{tab:syn_inf}
\begin{tabular}{|c|c|c|c|ccc|}
\hline
\multirow{2}{*}{$T$} & \multirow{2}{*}{$n$} & \multirow{2}{*}{$m$} & \multirow{2}{*}{$r$} & \multicolumn{3}{c|}{Distributions}\\
\cline{5-7}
~ & ~  & ~ & ~ & \multicolumn{1}{c|}{weight matrices}& \multicolumn{1}{c|}{the noise}& the input data\\
\hline
10 & 5 & 3 & 4 & \multicolumn{1}{c|}{$\mathcal{N}(0, 0.8)$} & \multicolumn{1}{c|}{$\mathcal{N}(0, 10^{-3})$} & $\mathcal{U}(-1, 1)$\\
\hline
500 & 80 & 30 & 100 & \multicolumn{1}{c|}{$\mathcal{N}(0, 0.05)$} & \multicolumn{1}{c|}{$\mathcal{N}(0, 10^{-5})$} & $\mathcal{U}(-1, 1)$\\
\hline
\end{tabular}
\end{center}
\end{table}

The dataset, \textbf{Volatility of S$\&$P index}, consists of the monthly realized volatility of the S$\&$P index and 11 corresponding exogenous variables from February 1973 to June 2009, totaling 437 time steps, i.e., $T = 437$, $n = 11$ and $m = 1$. The dataset was collected in strict adherence to the guidelines in \cite{bucci2020realized} and contains no missing values.
In the dataset, the monthly realized volatility of S$\&$P index is appointed as the output variable, while 11 exogenous variables are input variables. For training the RNNs, we first standardize the dataset as zero mean and unit variance, and then allocate 90$\%$ of the dataset, consisting of 393 time steps, as the training set, while the remaining 44 time steps are the test set. Moreover, we have $r = 20$ for the real dataset.

\subsection{Evaluations}
We define
$\textbf{FeasVio} := \text{max}\{\|\mathbf{u} - \Psi(\mathbf{h})\mathbf{w}\|, \|\mathbf{h} - (\mathbf{u})_{+}\|\}$ to evaluate the feasibility violation for constraints $\mathbf{u}=\Psi(\mathbf{h})\mathbf{w}$ and $\mathbf{h}=(\mathbf{u})_+$.
Moreover, the training and test errors are used to evaluate the forecasting accuracy of RNNs in training and test sets denoted as \begin{eqnarray*}
\textbf{TrainErr} := \frac{1}{T_{1}} \sum_{t=1}^{T_{1}}\|y_{t}-({A} ({W}(...({V}{x}_{1}+{b})_{+}...)+{V}{x}_{t}+{b})_{+} + {c}\|^2\\
\textbf{TestErr} := \frac{1}{T_{2}} \sum_{t=T_{1}+1}^{T_{1}+T_{2}}\|y_{t}-({A} ({W}(...({V}{x}_{1}+{b})_{+}...)+{V}{x}_{t}+{b})_{+} + {c})\|^2,
\end{eqnarray*}
where $T_1$ and $T_2$ are the time length of the training set and test set, and ${A}$, ${W}$, ${V}$, ${b}$ and ${c}$ are the output solutions from ALM.

\subsection{Investigating the feasibility}
\label{subsec:infea}
In this subsection, we aim to verify the outputs from the  ALM satisfying the constraints of (\ref{prob:cor}) through numerical experiments, while we have already proved the feasibility of any accumulation point of a sequence generated by the ALM in section 4. Initial values of weight matrices ${A}^{0},$ ${W}^{0},$ ${V}^{0}$ are randomly generated from the standard Gaussian distribution $\mathcal{N}\left(0, 0.1 \right)$. Moreover, the bias ${b}^{0}$ and ${c}^{0}$ are set as $0$. For all three datasets, we stop the outer loop (ALM) when it reaches 100 iterations, and the inner loop  (BCD method) terminates at 500 iterations. Other parameters are listed in \cref{tab:paras}.

\begin{table}[tbhp]
\footnotesize
\caption{Parameters of the ALM: the parameters for the given datasets are set as $\gamma^{0}=1$, $\xi^{0}=\boldsymbol{0}$, $\zeta^{0}=\boldsymbol{0}$, $\epsilon_{0} = 0.1$, $\Gamma = 10^2$, $\mu = 10^{-5}$, $\lambda_1 = \tau/rm$, $\lambda_2 = \tau/r^2$, $\lambda_3 = \tau/rn$, $\lambda_4 = \tau/r$, $\lambda_5 = \tau/m$, $\lambda_6 = 10^{-8}$.}
\label{tab:paras}
\begin{center}
\renewcommand\arraystretch{1.5}
\begin{tabular}{|c|c|c|}
\hline
Datasets & Regularization parameters & Algorithm parameters \\ \hline
\makecell{\textbf{Synthetic dataset ($T = 10$)}} & {$\tau = 1.2$} & \multirow{2}{*}{\makecell{$\eta_1=0.99$, $\eta_2=5/6$,\\ $\eta_3=0.01$, $\eta_4=5/6$.} } \\ \cline{1-2}
\makecell{\textbf{Volatility of S$\mathit{\&}$P index}} & $\tau = 1$ & \\ \hline
\makecell{\textbf{Synthetic dataset ($T = 500$)}} &  $\tau = 500$ &  \makecell{$\eta_1=0.90$, $\eta_2=0.90$,\\ $\eta_3=0.015$, $\eta_4=0.8$.} \\ \hline
\end{tabular}
\end{center}
\end{table}

From \cref{fig:fv}, we observe that the feasibility violation in each dataset is very small at the beginning, which implies that the selected initial point is feasible. As it turns to the first iteration, the feasibility violation goes to a large value. After that, the value goes to exhibit an oscillatory decrease and tends to zero.
This indicates that the points generated by the ALM gradually satisfy the constraint conditions as the number of iterations increases.
\begin{figure}[tbhp]
    \centering
    \subfloat[Synthetic dataset ($T=10$)]{\includegraphics[width=4.3cm]{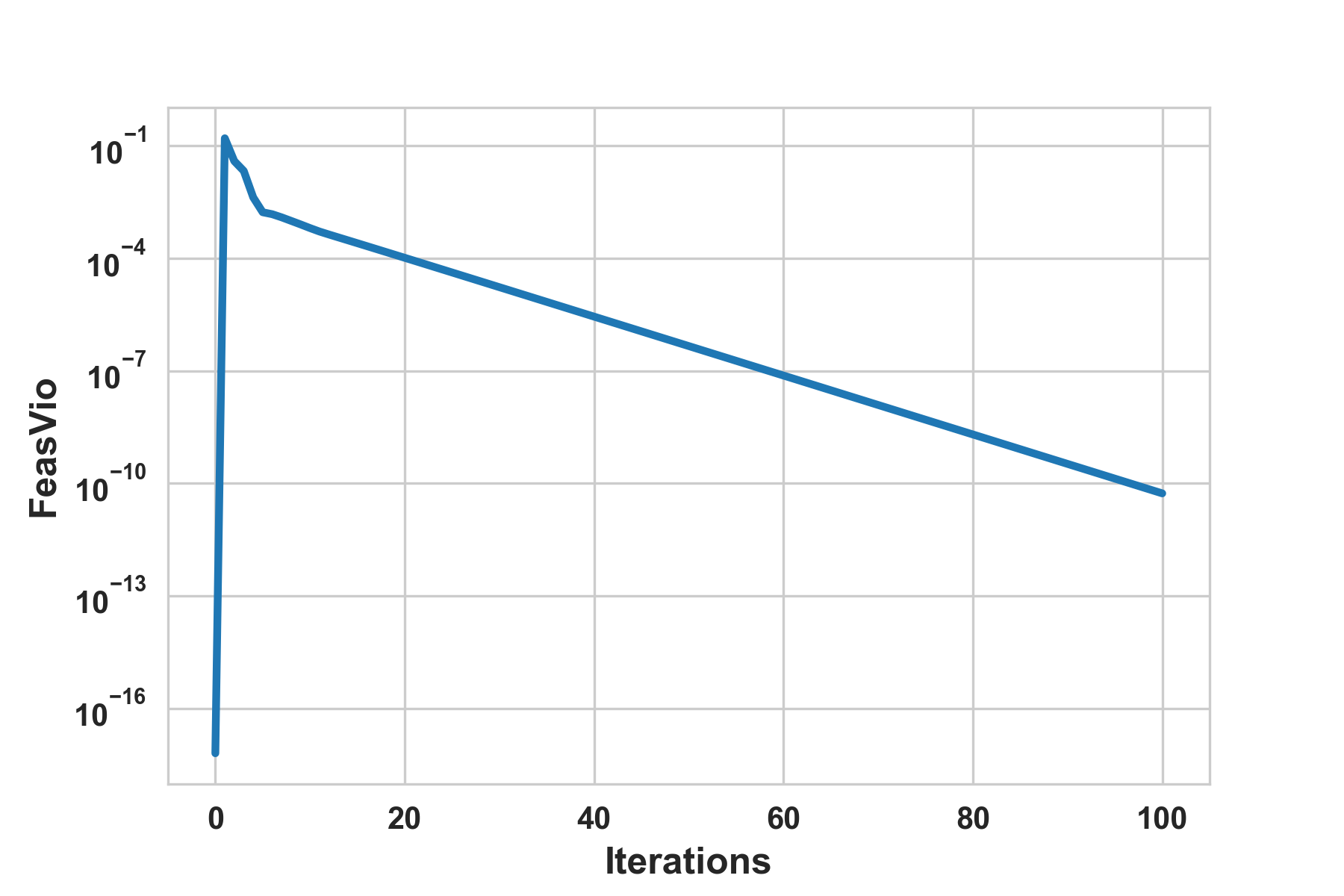}}
    \subfloat[Synthetic dataset ($T=500$)]{\includegraphics[width=4.3cm]{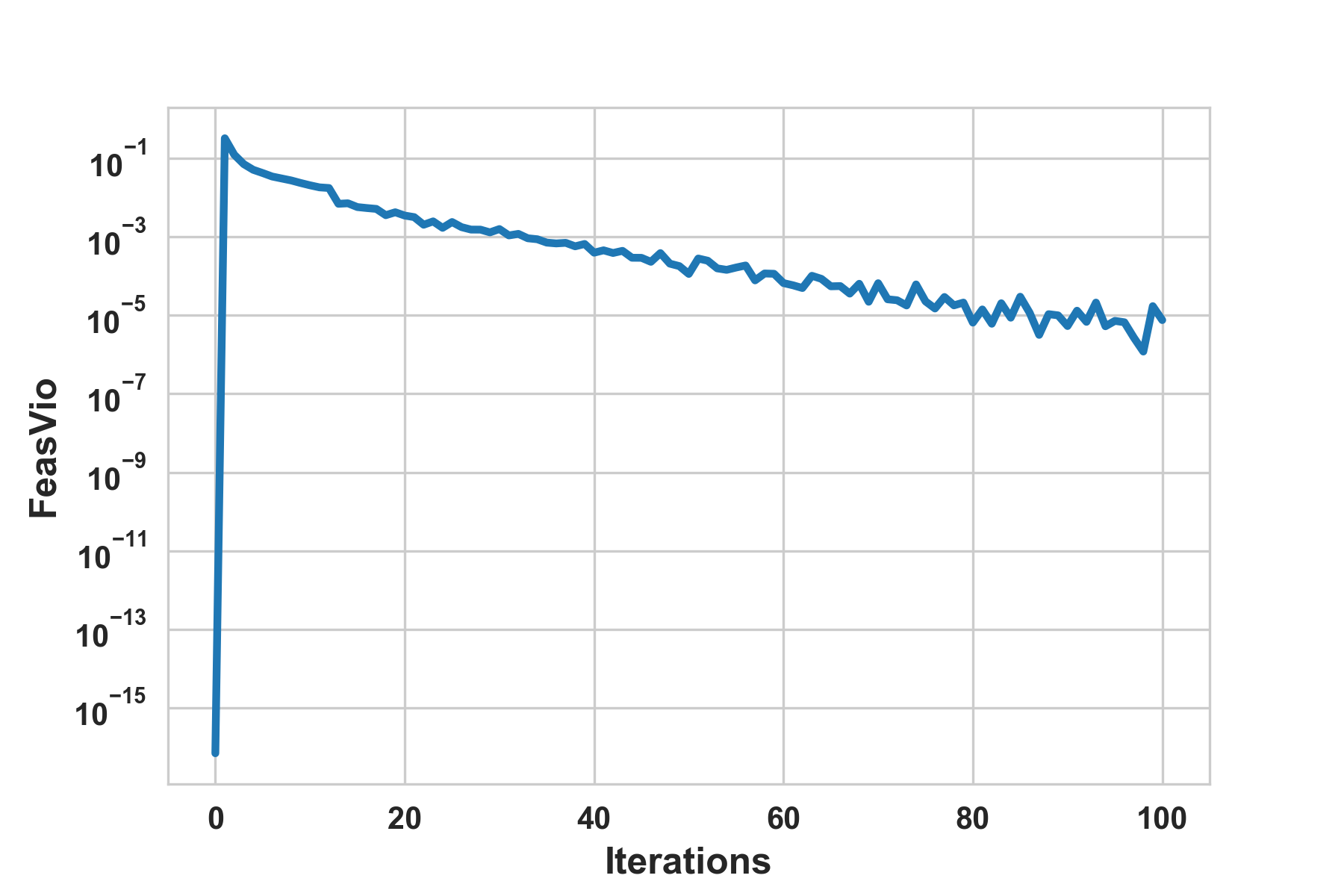}}
    \subfloat[Volatility of S$\mathit{\&}$P index]{\includegraphics[width=4.3cm]{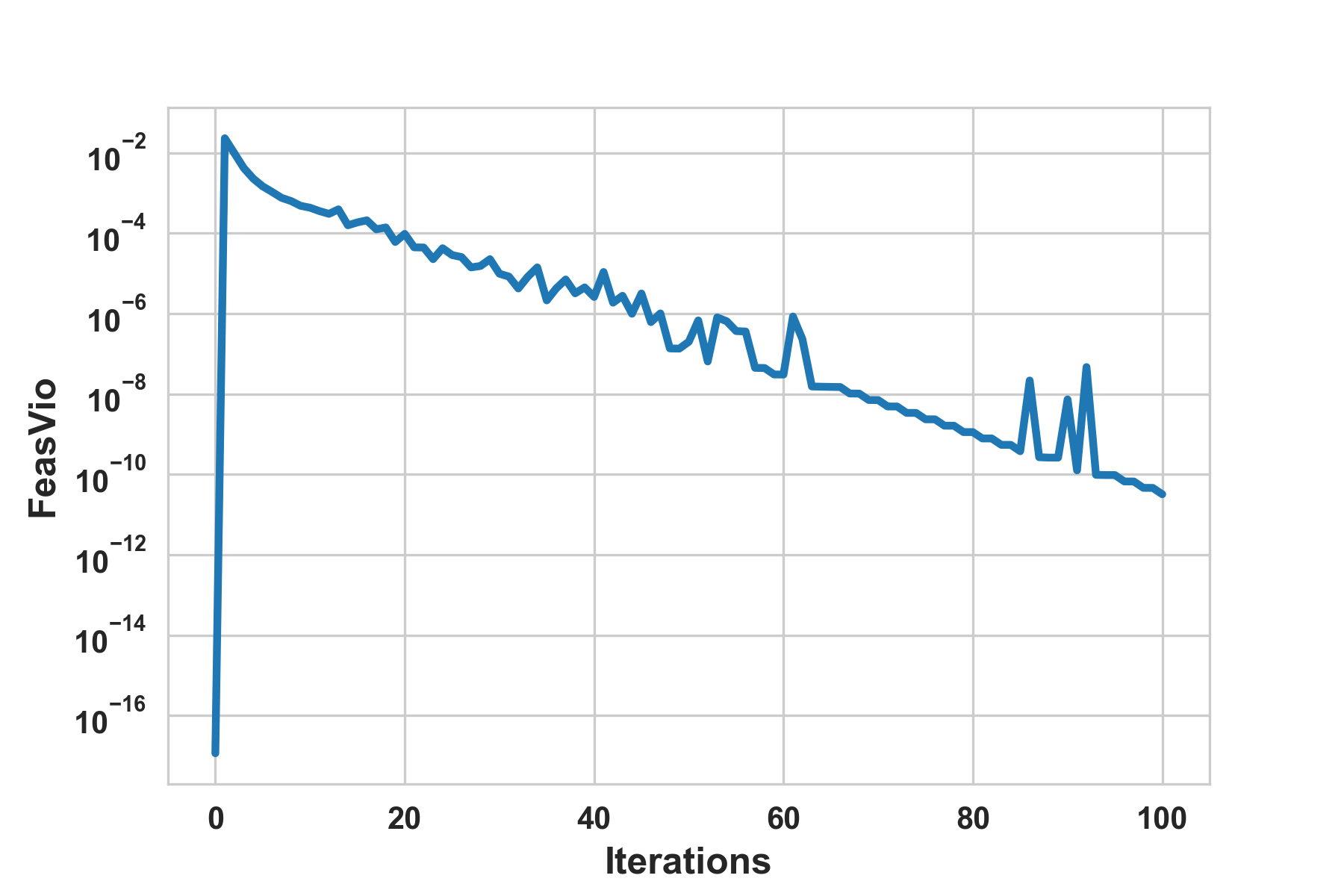}}
    \caption{The feasibility violation of the ALM in different datasets}
    \label{fig:fv}
\end{figure}

\subsection{Comparisons with state-of-the-art GDs}
\label{subsec:comp_GDs}
In this subsection, we compare the training and forecasting accuracy of RNNs using different methods. Specifically, we compare our ALM with the state-of-the-art GDs and SGDs with special techniques, i.e., gradient descent (GD), gradient descent with gradient clipping (GDC), gradient descent with Nesterov momentum {(GDNM)}, Mini-batch SGD and Adam.  

{For the initial values of ${A}^0$, ${W}^0$, ${V}^0$, we use the   following initialization strategies: random normal initialization \cite{bengio2009learning} with zero mean and standard deviations of \(10^{-3}\) and \(10^{-1}\), He initialization \cite{he2015delving}, Glorot initialization \cite{glorot2010understanding}, and LeCun initialization \cite{klambauer2017self}.  Notably, the initial values of bias, $b^0$ and $c^0$, were both set to $0$ according to \cite[pp. 305]{goodfellow2016deep}.}
% {\color{red}
% For appropriate selection the clipping norm for GDC, we applied the grid search method over the clipping norm of $\{0.5, 1, 1.5, 2, 3, 4, 5, 6\}$
% }
{
%For appropriate selection   learning rates for GDs and SGDs,

We search the  learning rates for GDs and SGDs over $\{10^{-4}$, $10^{-3}$, $10^{-2}$, $10^{-1}$, $1\}$, as well as the clipping norm of GDC  
%under different initialization strategies 
over $\{0.5, 1, 1.5, 2, 3, 4, 5, 6\}$.} 
%
%{The grid search is a widely used method to test multiple hyperparameters \cite{bergstra2012random}, which works through all combinations of multiple hyperparameters and performs}
{
We employ the leave-P-out cross validation and repeated each method 30 trials with $P=1$ in \textbf{Synthetic dataset ($T=10$)}, and $P=10$ in \textbf{Volatility of  S$\mathit{\&}$P index} and \textbf{Synthetic dataset ($T=500$)}. We then select the learning rates and clipping norm with the best test error averaged over 30 trials, which are recorded in Table \ref{tab:lr} of  \cref{subsec:hyperpare}.}
% %The appropriate parameters are finally determined by  those achieving the smallest \textcolor{red}{test error} through a leave-P-out cross-validation 
% $P=1$ in \textbf{Synthetic dataset ($T=10$)} and $P=10$ in \textbf{Volatility of the S$\mathit{\&}$P index} and \textbf{Synthetic dataset ($T=500$)}. 
% The final selections of the learning rates and the clipping norm are recorded 
% %in Table \ref{tab:lr} 
% in the supplementary material.} %{to determine the combination that yields the best model performance.
%The detailed results are in the supplementary material. 
%In the following experiments, the selection of learning rates for GDs and SGDs is based on these results. 
{The batch size for SGDs is set to 2 for \textbf{Synthetic dataset ($T=10$)}, 50 for \textbf{Volatility of S$\mathit{\&}$P index}, and 100 for \textbf{Synthetic dataset ($T=500$)}. We employ the Keras API \cite{chollet2015keras} running on TensorFlow 2 to implement the GDs and SGDs. Additionally, the parameters for the ALM are listed in \cref{tab:paras}.  }

{To evaluate the performance of different methods under various initialization strategies, we conducted the following experiments: each method was repeated 10 times under each initialization strategy. In each repetition, we recorded the final test error 
%\textbf{TrainErr} 
and %\textbf{TestErr}
the training error. We then calculated their means (\textbf{TrainErr} and \textbf{TestErr}) and the corresponding standard deviations, and listed them  
 in \cref{tab:comp_methods}.
% the averaged \textbf{TrainErr} and \textbf{TestErr} on the real {\bf{Volatility of S\&P dataset}}.  
%indexdifferent optimization methods and initialization strategies using the .
}
{Each row  records the results for a certain optimization method from different initialization strategies, with the best  \textbf{TrainErr} or \textbf{TestErr}
 highlighted in bold. 
Each column provides the results of all the optimization methods with the same initial values, where the best \textbf{TrainErr} and \textbf{TestErr} are highlighted underline.} 

%\textcolor{magenta}{With regard to the \textbf{TrainErr}, we can easily find that our ALM performs the best using any of the initialization strategies for  the two synthetic datasets, and performs well for the real dataset, \textbf{Volatility of the S$\mathit{\&}$P index}. Moreover, the ALM obtains the best mean \textbf{TarinErr} when using the initialization strategy ${\cal{N}}(0, 10^{-1})$ among all combinations of optimization methods and initialization strategies.
%
%We are more concerned about the \textbf{TestErr}, and highlight  the best  {\bf{TestErr}} in each row/column in bold/underline, respectively.}

{
\cref{tab:synT10_comp_methods} and 
Table 3c
%\eqref{tab:synT500_comp_methods} 

demonstrate that for \textbf{Synthetic dataset ($T=10$)} and \textbf{Synthetic dataset ($T=500$)},  no matter which initialization strategy is employed, our ALM method achieves the best {\bf{TrainErr}} and {\bf{TestErr}} among all the methods.} 
{\cref{tab:sp500_comp_methods} illustrates that our ALM achieves  
the best {\bf{TrainErr}}
under two types of initialization strategies, and 
obtains 
the best  %{\bf{TrainErr}} and 
{\bf{TestErr}} under three types of initialization strategies for  \textbf{Volatility of  S$\mathit{\&}$P index}. For any of the three datasets, our ALM  %with normal initialization $\mathcal{N}(0, 0.1)$ 
achieves the best {\bf{TrainErr}} and {\bf{TestErr}} among all combinations of optimization methods and initialization strategies,
which we highlight in blue. 
%From \cref{tab:comp_methods}, we can see that our ALMRNN perform well and is stable using different initialization strategies, compared to other methods. 
 }
  
\setlength{\tabcolsep}{1.5pt}
\renewcommand\arraystretch{1.5}
\begin{table}[H]
    \centering
    \caption{Results of training Elman RNNs using different optimization methods and initialization strategies across multiple trials.}
    \scriptsize
    \label{tab:comp_methods}
    \subfloat[\textbf{Synthetic dataset ($T=10$)}: For the ALM method, the maximum iteration for the outer loop is 50 and 10 for the inner loop. For GDs and SGDs, the number of epochs is set to 500.]
    {
    \label{tab:synT10_comp_methods}
    % \resizebox{0.95\textwidth}{!}{%
    \begin{tabular}{|c|c|c|c|c|c|c|}
    \hline
    &   & He  & $\mathcal{N}(0, 10^{-3})$ & $\mathcal{N}(0, 10^{-1})$ & Glorot  & LeCun \\ \hline
    \multirow{2}{*}{ALM} 
    & \textbf{TrainErr} & $\underline{0.345 \pm 0.24}$ & $\textcolor{blue}{{\underline{\bf{0.113 \pm 0.03}}}}$ & $\underline{0.143 \pm 0.04}$ & $\underline{0.206 \pm 0.10}$ & $\underline{0.279 \pm 0.22}$ \\ \cline{2-7} 
    & \textbf{TestErr}  & {$\underline{4.770 \pm 1.25}$} & {$\textcolor{blue}{\underline{\bf{4.437 \pm 0.28}}}$} & {$\underline{4.660\pm 0.35}$} & {$\underline{4.628 \pm 1.17}$} & {$\underline{4.650 \pm 0.62}$} \\ 
    \hline
    \multirow{2}{*}{GD}  
    & \textbf{TrainErr} & $4.459 \pm 0.77$ & $2.747 \pm 1.5\text{e-}6$ & $2.768 \pm 0.01$ & $1.814 \pm 0.27$ & $\bf{1.604 \pm 0.17}$ \\ 
    \cline{2-7} 
    & \textbf{TestErr}  & $6.432 \pm 2.15$ & $5.311 \pm 9.3{\text{e-}}6$  & $5.057 \pm 0.07$ & {$\bf{4.696 \pm 0.90}$} & $5.056 \pm 1.10$\\ 
    \hline
    \multirow{2}{*}{GDC}    
    & \textbf{TrainErr} & $\bf{1.479 \pm 0.32}$ & $2.769 \pm 1.4{\text{e-}}6$ & $2.768 \pm 0.01$ & $1.684 \pm 0.23$ & $1.502 \pm 0.26$ \\ \cline{2-7} 
    & \textbf{TestErr}  & $5.376 \pm 0.88$ & $5.079 \pm 1.0{\text{e-}}6$ & $5.057 \pm 0.07$ & {$\bf{4.922 \pm 1.20}$} & $5.266 \pm 0.96$\\ 
    \hline
    \multirow{2}{*}{GDNM} 
    & \textbf{TrainErr} & $2.689 \pm 0.40$ & $2.769 \pm 1.4{\text{e-}}6$ & $2.768 \pm 0.01$ & $3.340 \pm 0.54$ & $\bf{0.801 \pm 0.60}$\\ \cline{2-7} 
    & \textbf{TestErr}  & $6.169 \pm 2.06$ & $5.079 \pm 1.0{\text{e-}}6$ & $5.057 \pm 0.07$  & $7.469 \pm 2.30$ & {$\bf{4.844 \pm 0.64}$}\\ 
    \hline
    \multirow{2}{*}{SGD}    
    & \textbf{TrainErr} & $\bf{2.224 \pm 0.02}$ & $2.247 \pm 0.02$ & $2.232 \pm 0.02$ & $2.238 \pm 0.02$ & $2.225 \pm 0.02$\\ \cline{2-7} 
    & \textbf{TestErr}  & $6.455 \pm 0.23$ &  {$\bf{6.230 \pm 0.23}$} &$6.373 \pm 0.18$ & $6.543 \pm 0.23$ & $6.446 \pm 0.18$ \\
    \hline
    \multirow{2}{*}{Adam}  
    & \textbf{TrainErr} & $2.283 \pm 0.07$ & $2.244 \pm 0.02$ & $2.237 \pm 0.02$ & $\bf{2.231 \pm 0.01}$ & $2.239 \pm 0.03$\\ 
    \cline{2-7} 
    & \textbf{TestErr}  & {$\bf{6.335 \pm 0.61}$} & $6.432 \pm 0.27$ & $6.411 \pm 0.25$ & $6.508 \pm 0.14$ & $6.406 \pm 0.20$\\ 
    \hline
    \end{tabular}
    }
    % }
    \quad
    \quad
    \subfloat[\textbf{Volatility of S$\mathit{\&}$P index}: For the ALM method, the maximum iteration for the outer loop is 200 and 500 for the inner loop. For GDs and SGDs, the number of epochs is set to 5000.]{
    \label{tab:sp500_comp_methods}
    % \resizebox{0.95\textwidth}{!}{%
    \begin{tabular}{|c|c|c|c|c|c|c|}
    \hline
    &   & He  & $\mathcal{N}(0, 10^{-3})$ & $\mathcal{N}(0, 10^{-1})$ & Glorot  & LeCun \\ \hline
    \multirow{2}{*}{ALM} 
    & \textbf{TrainErr} & $0.058 \pm 0.02$ & $\underline{0.004 \pm 3.6{\text{e-}}{5}}$ & $\textcolor{blue}{\underline{\bf{0.003 \pm 1.4{\text{e-}}4}}}$ & $0.009 \pm 0.002$ & $0.013 \pm 0.002$ \\ \cline{2-7} 
    & \textbf{TestErr}  & $0.229 \pm 0.13$ & {$\underline{0.041 \pm 4.7{\text{e-}}4}$} & {$\textcolor{blue}{\underline{\bf{0.032 \pm 0.005}}}$} & {$\underline{0.064 \pm 0.04}$} & $0.053 \pm 0.03$ \\ 
    \hline
    \multirow{2}{*}{GD}     
    & \textbf{TrainErr} & $\bf{0.005 \pm 0.001}$ & $0.015 \pm 1.8{\text{e-}}4$ & $0.012 \pm 9.2{\text{e-}}4$ & $0.020 \pm 0.003$ & $0.025 \pm 0.006$ \\ 
    \cline{2-7} 
    & \textbf{TestErr}  & $0.124 \pm 0.10$ & $0.077 \pm 0.03$  & {$\bf{0.0429 \pm 0.01}$} & $0.206 \pm 0.20$ & $0.307 \pm 0.20$\\ 
    \hline
    \multirow{2}{*}{GDC}    
    & \textbf{TrainErr} & $0.567 \pm 0.47$ & $0.015 \pm 1.8{\text{e-}}4$ & $0.016 \pm 0.009$ & $\underline{\bf{0.003 \pm 5.6{\text{e-4}}}}$ & $0.011 \pm 0.003$ \\ \cline{2-7} 
    & \textbf{TestErr}  & $1.135 \pm 0.55$ & $0.077 \pm 0.03$ & $0.047 \pm 0.02$ & $0.107 \pm 0.03$ & {$\underline{\bf{0.041 \pm 0.01}}$}\\ 
    \hline
    \multirow{2}{*}{GDNM}  
    & \textbf{TrainErr} & $0.005 \pm 0.001$ & $0.015 \pm 1.8{\text{e-}}4$ & $0.012 \pm 9.2{\text{e-}}4$ & $\bf{0.003 \pm 5.8{\text{e-}}4}$ & $\underline{0.004 \pm 6.6{\text{e-}}4}$\\ \cline{2-7} 
    & \textbf{TestErr}  & $0.124 \pm 0.10$ & $0.077 \pm 0.03$ & {$\bf{0.043} \pm 0.01$}  & $0.097 \pm 0.03$ & $0.102 \pm 0.02$\\ 
    \hline
    \multirow{2}{*}{SGD}    
    & \textbf{TrainErr} & $\underline{\bf{0.005 \pm 1.8\text{e-}4}}$ & $0.006 \pm 0.002$ & $0.006 \pm 0.002$ & $0.006 \pm 0.002$ & $0.006 \pm 0.002$\\ \cline{2-7} 
    & \textbf{TestErr}  & {$\underline{\bf{0.072 \pm 0.01}}$} & $0.095 \pm 0.02$ & $0.086 \pm 0.02$ & $0.085 \pm 0.01$ & $0.096 \pm 0.01$ \\
    \hline
    \multirow{2}{*}{Adam}   
    & \textbf{TrainErr} & $0.006 \pm 0.001$ & $\bf{0.005 \pm 7.6{\text{e-}}4}$ & $0.006 \pm 0.002$ & $0.006 \pm 0.001$ & $\bf{0.005 \pm 7.6{\text{e-4}}}$\\ 
    \cline{2-7} 
    & \textbf{TestErr}  & $0.079 \pm 0.01$ & {$\bf{0.074 \pm 0.01}$} & $0.084 \pm 0.01$ & $0.080 \pm 0.02$ & $0.080 \pm 0.02$\\ 
    \hline
    \end{tabular}
    }
    % }
\end{table}

\begin{table}[H]
\ContinuedFloat
\centering
\footnotesize
\captionsetup{position=top} 
%\caption{Results of training Elman RNNs using different optimization methods and initialization strategies across multiple trials.}
%\label{tab:comp_methods}
    \subfloat[\textbf{Synthetic dataset ($T=500$)}: For the ALM method, the maximum iteration for the outer loop is 100 and 500 for the inner loop. For GDs and SGDs, the number of epochs is set to 1000.]{
    \label{tab:synT500_comp_methods}
    \resizebox{0.97\textwidth}{!}{%
    \begin{tabular}{|c|c|c|c|c|c|c|}
    \hline
    &   & He  & $\mathcal{N}(0, 10^{-3})$ & $\mathcal{N}(0, 10^{-1})$ & Glorot  & LeCun \\ \hline
    \multirow{2}{*}{ALM} 
    & \textbf{TrainErr} & $\underline{4.639 \pm 0.78}$ & $\textcolor{blue}{\underline{\bf{3.461 \pm 0.06}}}$ & $\underline{3.472 \pm 0.05}$ & $\underline{3.472 \pm 0.06}$ & $\underline{3.475 \pm 0.06}$ \\ \cline{2-7} 
    & \textbf{TestErr}  & {$\underline{14.77 \pm 0.93}$} & {$\underline{12.418 \pm 0.16}$} & {$\underline{12.407 \pm 0.27}$} & {$\textcolor{blue}{\underline{\bf{12.394 \pm 0.22}}}$} & {$\underline{12.517 \pm 0.16}$} \\ 
    \hline
    \multirow{2}{*}{GD}     
    & \textbf{TrainErr} & $58.137 \pm 2.42$ & $30.010 \pm 0.003$ & $30.013 \pm 0.008$ & $30.000 \pm 0.008$ & $\bf{29.985 \pm 0.007}$ \\ 
    \cline{2-7} 
    & \textbf{TestErr}  & $58.314 \pm 2.76$ & $28.644 \pm 0.006$  & $28.641 \pm 0.009$ & $28.630 \pm 0.006$ & {$\bf{28.626 \pm 0.009}$}\\ 
    \hline
    \multirow{2}{*}{GDC}    
    & \textbf{TrainErr} & $250.471 \pm 399.70$ & $\bf{30.004 \pm 0.003}$ & $30.144 \pm 0.001$ & $30.143 \pm 8.8{\text{e-}}4$ & $30.144 \pm 0.001$ \\ \cline{2-7} 
    & \textbf{TestErr}  & $119.007 \pm 66.71$ & {$\bf{28.640 \pm 0.007}$} & $28.723 \pm 0.007$ & $28.730 \pm 0.006$ & $28.725 \pm 0.01$\\ 
    \hline
    \multirow{2}{*}{GDNM}  
    & \textbf{TrainErr} & $58.137 \pm 2.42$ & $30.010 \pm 0.003$ & $30.013 \pm 0.008$ & $30.000 \pm 0.008$ & $\bf{29.985 \pm 0.007}$\\ \cline{2-7} 
    & \textbf{TestErr}  & $58.314 \pm 2.76$ & $28.644 \pm 0.006$ & $28.641 \pm 0.009$  & $28.730 \pm 0.006$ & {$\bf{28.626 \pm 0.009}$}\\ 
    \hline
    \multirow{2}{*}{SGD}    
    & \textbf{TrainErr} & $30.142 \pm 3.5{\text{e-}}6$ & $30.142 \pm 4.7{\text{e-}}6$ & $\bf{30.142 \pm 5.2{\text{e-}}6}$ & $30.142 \pm 4.4\text{e-}6$ & $30.142 \pm 4.8{\text{e-}}6$\\ \cline{2-7} 
    & \textbf{TestErr}  &{$\bf{28.725 \pm 3.2{\text{e-}}5}$} & $28.725 \pm 4.4{\text{e-}}5$ & $28.725 \pm 4.7{\text{e-}}5$ & $28.725 \pm 3.9{\text{e-}}5$ & $28.725 \pm 4.1{\text{e-}}5$ \\
    \hline
    \multirow{2}{*}{Adam}   & \textbf{TrainErr} & $\bf{30.142 \pm 7.1{\text{e-}}5}$ & $30.142 \pm 6.5{\text{e-}}5$ & $30.142 \pm 7.3{\text{e-}}5$ & $30.142 \pm 5.1{\text{e-}}5$ & $30.142 \pm 5.7{\text{e-}}5$\\ 
    \cline{2-7} 
    & \textbf{TestErr}  & $28.726 \pm 6.1{\text{e-}}4$ & $28.725 \pm 5.0{\text{e-}}4$ & $28.726 \pm 5.9{\text{e-}}4$ & $28.726 \pm 5.0{\text{e-}}4$ & {$\bf{28.725 \pm 4.8{\text{e-}}4}$}\\ 
    \hline
    \end{tabular}
    }
    }
\end{table}

\begin{figure}[H]
    \centering
    \subfloat[\textbf{Volatility of S$\mathit{\&}$P index}]{
    \label{fig_i:comalg_sp500}
    \includegraphics[width=0.482\textwidth]{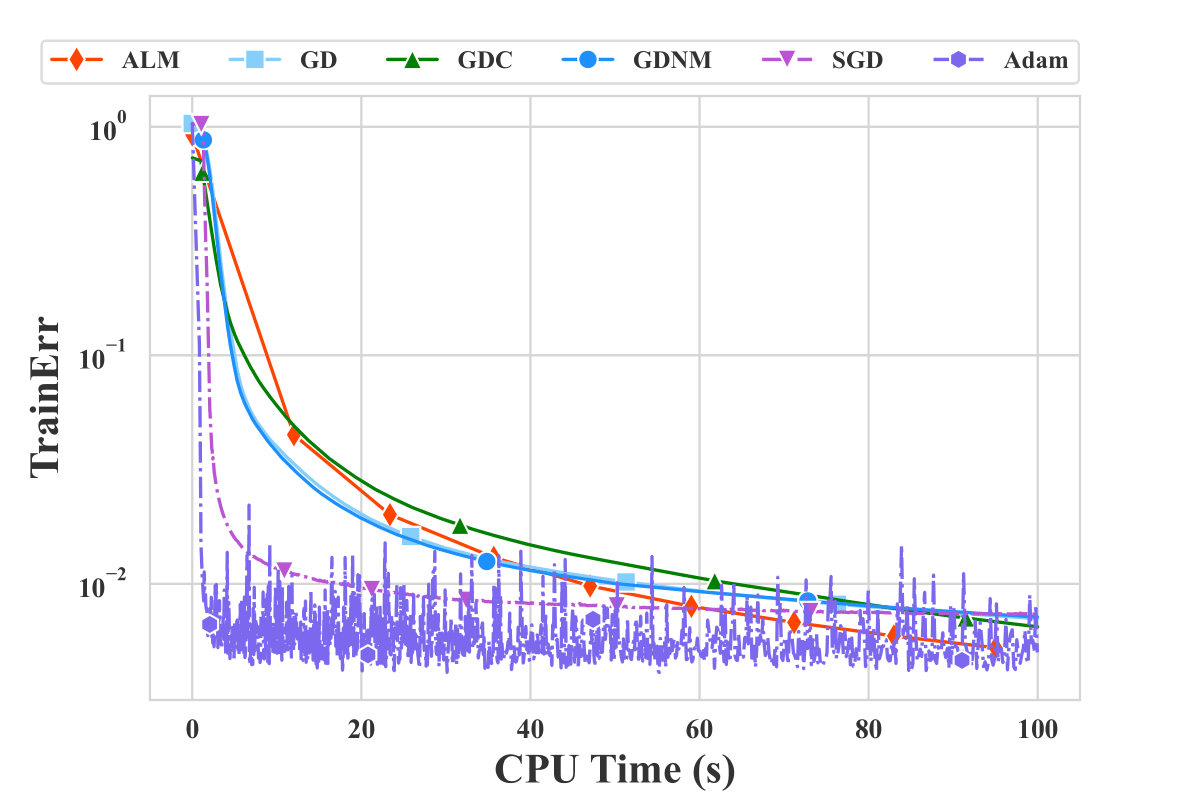}
    \includegraphics[width=0.482\textwidth]{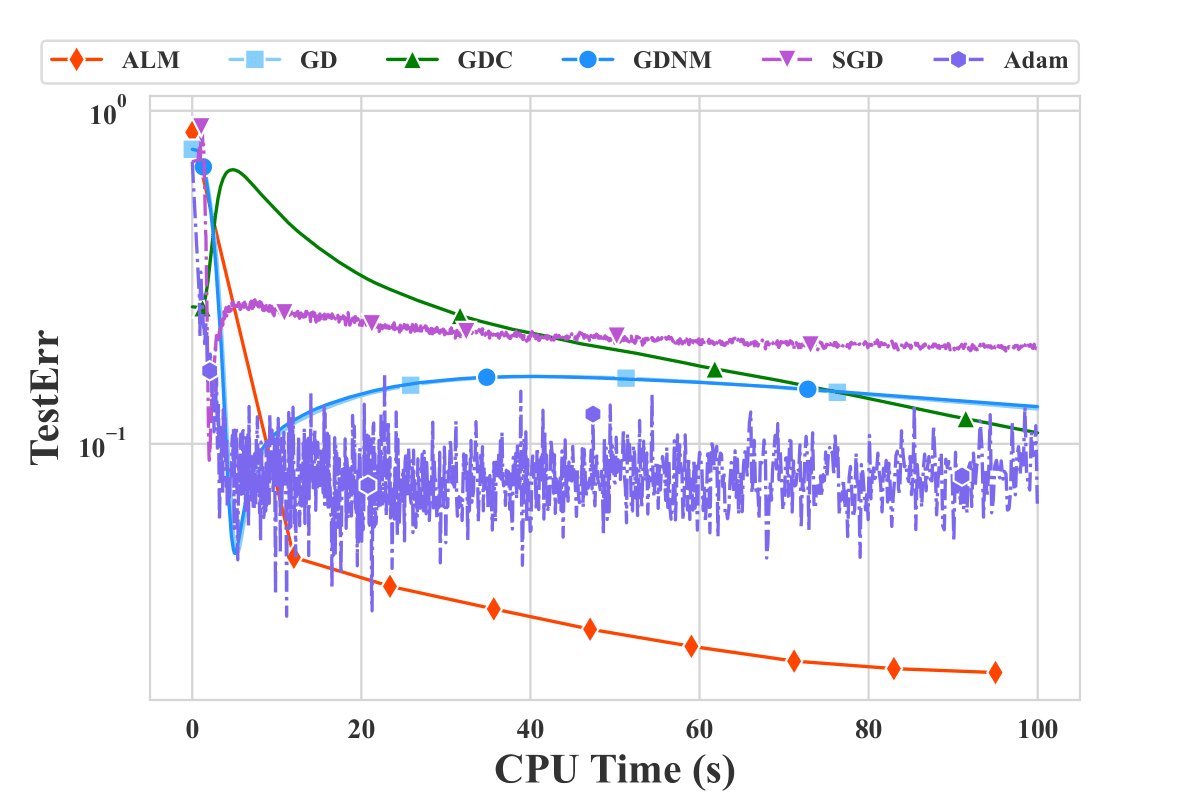}
    }
    \quad
    \subfloat[\bf{Synthetic dataset ($T = 500$)}]{
    \label{fig_i:comalg_synT500}
    \includegraphics[width=0.481\textwidth]{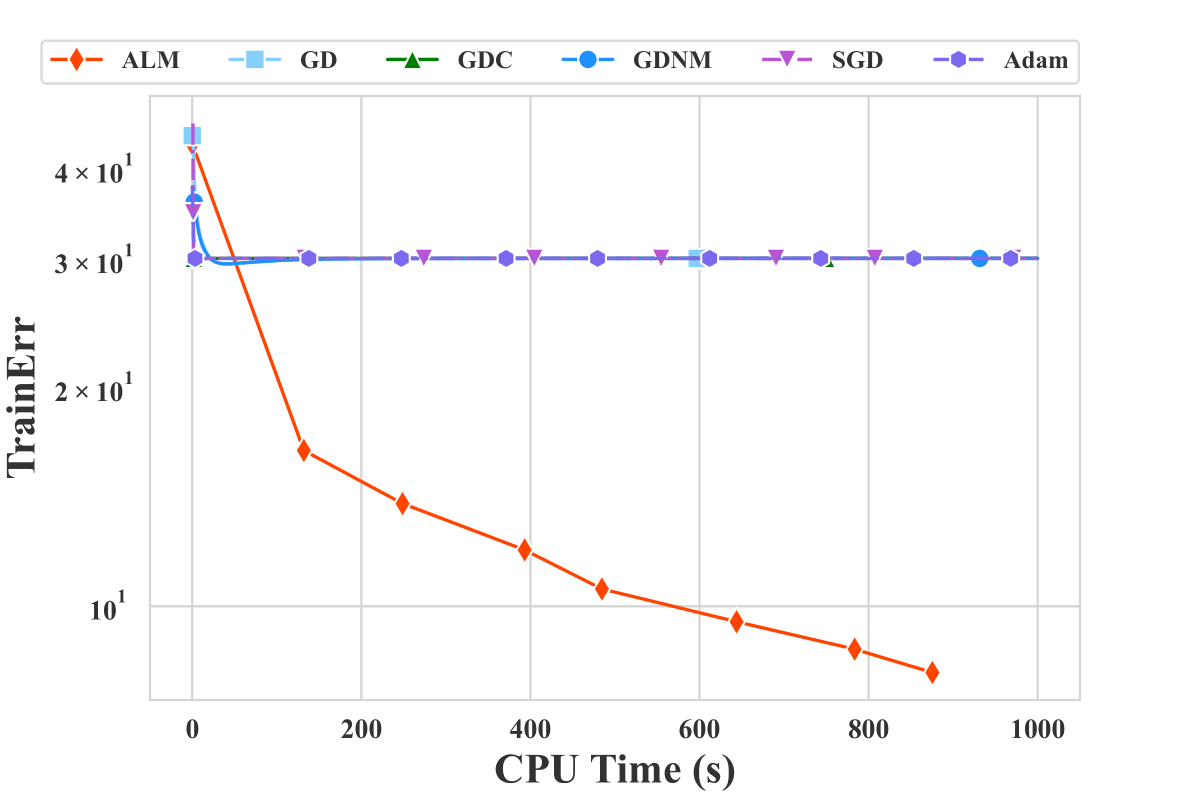}
    \includegraphics[width=0.481\textwidth]{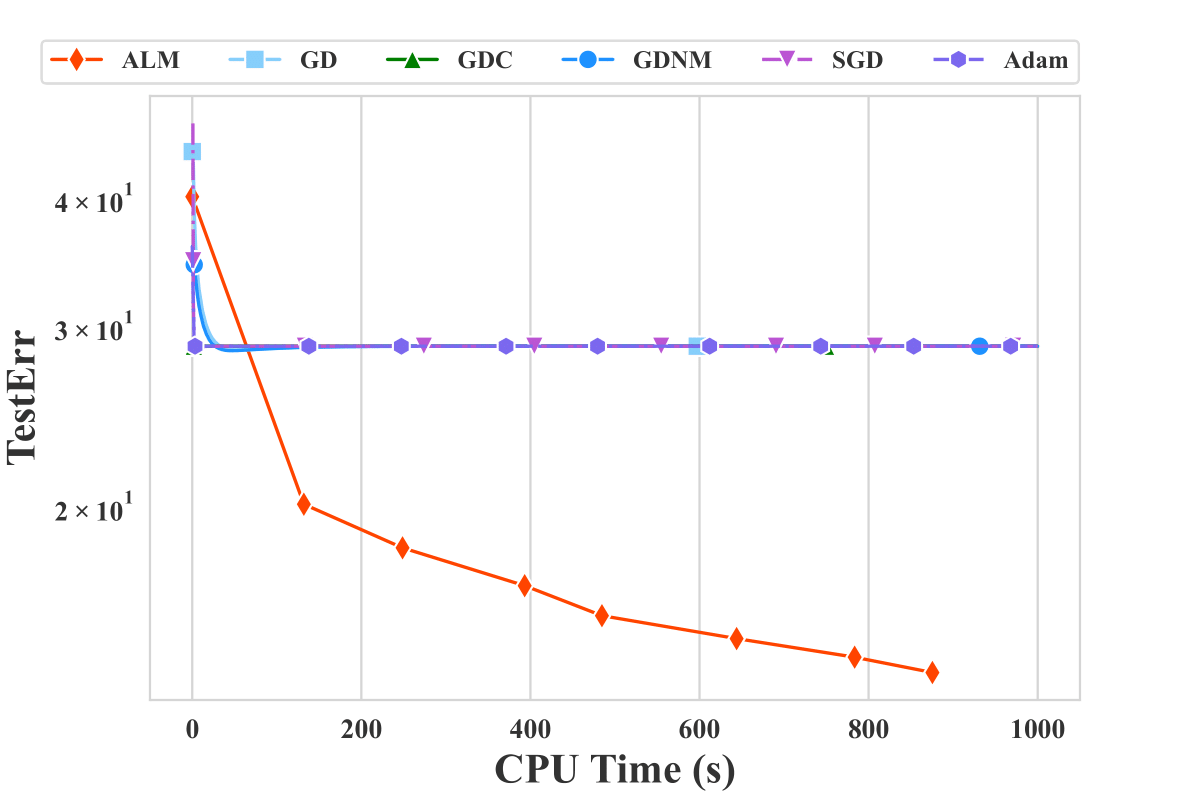}
    }
    \caption{Comparisons of the performance of the ALM, GDs and SGDs across different datasets.}
    \label{fig:comalg}
\end{figure} 
{We plot in \cref{fig:comalg}   the \textbf{TrainErr} and \textbf{TestErr} versus CPU time} {measured in seconds using  \textbf{Volatility of S$\mathit{\&}$P index} and  \textbf{Synthetic dataset ($T=500$)}. Each line corresponds to a certain optimization method as indicated in the legend, with  its most appropriate initialization strategy
%and learning rate 
that leads to the final {\textbf{TestErr}} in bold as outlined in \cref{tab:comp_methods}. For the real world dataset, \textbf{Volatility of S$\mathit{\&}$P index}, the ALM achieves the smallest test error among all the methods. For the larger-scale \textbf{Synthetic dataset ($T=500$)} with $N_{\mathbf{w}} = 1.81 \times 10^{4}$, $N_{\mathbf{a}} = 3.03 \times 10^{3}$ and $r=500$, the ALM exhibits superior performance in terms of both training and test errors.}

\section{Conclusion}
In this paper, the minimization model (\ref{pb:SAA_rnn}) for training  RNNs is equivalently reformulated as problem (\ref{prob:cor}) by using auxiliary variables.   We propose the ALM in \cref{Alg:ALM} with \cref{Alg:BCD} to solve the regularized problem (\ref{prob:ror}). The BCD method in  \cref{Alg:BCD} is efficient for solving the subproblems of the ALM, which has a closed-form solution for each block problem. We establish the solid convergence results of the ALM to a KKT point of problem (\ref{prob:ror}), as well as the finite termination of the BCD method for the subproblem of the ALM at each iteration.
The efficiency and effectiveness of the ALM for training RNNs are demonstrated by
numerical results with real world datasets and synthetic data, and comparison with state-of-art algorithms.
An interesting further study is to extend our algorithm to a stochastic algorithm that is potential to deal with problems of huge samples efficiently. {We believe that it is possible to extend our method and its corresponding analysis to other more complex RNN architectures, such as LSTMs, and will give rigorous analysis in the near future.}

\section*{Acknowledgments}
%We would like to thank xxx for their support during this work. 
We are grateful to Prof. M. Mahoney and the anonymous referees for valuable comments.

\appendix
\section{Proofs of the lemmas}\label{appendixa}
\subsection{Proof of Lemma \ref{lemma: NNAMCQ}}
\begin{proof}
By direct computation,
\begin{equation}
\begin{split}
\label{subdiff_2cons_details}
    J\mathcal{C}_1(\mathbf{s})^{\top} \xi + \partial \big(\zeta^{\sss \top} {\cal C}_2(\mathbf{s})\big)
    = \begin{bmatrix}
    J_{\bf{z}}\mathcal{C}_1(\mathbf{s})^{\top} \xi\\
    {J_{\mathbf{h}}\mathcal{C}_1(\bf{s})^{\top} \xi} + J_{\bf{h}}\mathcal{C}_2(\mathbf{s})^{\top} \zeta\\
    J_{\mathbf{u}}\mathcal{C}_1(\mathbf{s})^{\top} \xi + \partial_{\mathbf{u}}\left(\zeta^{\top} \mathcal{C}_2(\mathbf{s})\right)
    \end{bmatrix},
\end{split}
 \end{equation}
where
\begin{gather}
    J_{\mathbf{h}}\mathcal{C}_1(\mathbf{s})^{\top} \xi + J_{\mathbf{h}}\mathcal{C}_2(\mathbf{s})^{\top} \zeta = \left[-{W}^{\top}\xi_2 + \zeta_1; ...; -{W}^{\top}\xi_{T} + \zeta_{T-1}; \zeta_{T}\right], \label{CQ_partial_h}\\
    J_{\mathbf{u}}\mathcal{C}_1(\mathbf{s})^{\top} \xi + \partial_{\mathbf{u}}\left(\zeta^{\top} \mathcal{C}_2(\mathbf{s})\right)
    = \xi + \partial_{\mathbf{u}}(-\zeta ^{\top}(\mathbf{u})_+). \label{CQ_partial_u}
\end{gather}
In order to achieve ${0} \in J\mathcal{C}_1(\mathbf{s})^{\top} \xi + \partial \big(\zeta^{\sss \top} {\cal C}_2(\mathbf{s})\big)$, it is necessary to require $\zeta_{T} = 0$, which is located in the last row of $J_{\mathbf{h}}\mathcal{C}_1(\mathbf{s})^{\top} \xi + J_{\mathbf{h}}\mathcal{C}_2(\mathbf{s})^{\top} \zeta$. Using $\zeta_{T}=0$ and \cref{CQ_partial_u}, we find  $\xi_{T} = 0$.
Substituting the results into \cref{CQ_partial_h} and \cref{CQ_partial_u} recursively and using \cref{CQ_partial_h} and \cref{CQ_partial_u} equal $0$, we can derive that there exist no nonzero vectors $\xi$ and $\zeta$ such that ${0} \in J\mathcal{C}_1(\mathbf{s})^{\sss \top} \xi + \partial \big(\zeta^{\sss \top} {\cal C}_2(\mathbf{s})\big)$.
\end{proof}

\subsection{Proof of Lemma \ref{lemma:solutionset_nonempty}}
\begin{proof}
It is clear that ${0} \in \mathcal{D}_{\cal R}(\rho)$ and consequently $\mathcal{D}_{\cal R}(\rho)$ is nonempty.
Moreover,
\begin{eqnarray}
\label{AWVb}\|{A}\|_F^2 \leq \rho/\lambda_1, \|{W}\|_F^2 \leq \rho/\lambda_2, \|{V}\|_F^2 \leq \rho/\lambda_3,\\
\|{b}\|^2 \leq \rho/\lambda_4, \|{c}\|^2 \leq \rho/\lambda_5, {\|\bf{u}\|^2} \leq \rho/\lambda_6,
\notag
\end{eqnarray}
from  $\mathcal{R}(\mathbf{s}) \leq \rho$, $\ell(\mathbf{s})  \geq 0$ and $P(\mathbf{s}) \geq 0$. Hence for ${\bf{s}}=({\bf{z}};{\bf{h}};{\bf{u}})\in {\cal D}_{\cal R}(\rho)$, $\bf{z}$ and  $\bf{u}$ are bounded, and
consequently $\mathbf{h}$ is also bounded because ${\mathbf{h}} = ({\mathbf{u}})_+$.

Up to now, we have obtained the boundedness of $\mathcal{D}_{\cal R}(\rho)$. By the continuity of $\mathcal{R}(\mathbf{s})$, we can assert that $\mathcal{D}_{\cal R}(\rho)$ is closed according to \cite[ Theorem 1.6]{rockafellar2009variational}.
Thus we can claim that the level set $\mathcal{D}_{\cal R}(\rho)$ is nonempty and compact for any $\rho > {\cal R}(0)$. Then the solution set $\mathcal{S}_1$ is nonempty and compact according to \cite[Proposition A.8]{bertsekas1997nonlinear}.
\end{proof}

\subsection{ Proof of Lemma \ref{gradient}}
\begin{proof}
Statement  \ref{lem_ele:func_lb}   can be easily obtained by the expression of ${\cal{L}}({\bf s},\xi,\zeta,\gamma)$ in  \eqref{alm function for 2constraints} and the nonnegativity of ${\cal R}(\bf{s})$ in \eqref{prob:ror}.

For statement \ref{lem_ele:levelset_pro}, the nonemptyness and closedness of  the level set $\Omega_{{\cal L}}(\hat \Gamma)$ are obvious. Moreover,  we have ${\cal R}({\bf{s}})$ and $\|{\bf h} - ({\bf u})_+ + \frac{\zeta}{\gamma}\|$ are upper bounded for all ${\bf s}$ in $\Omega_{{\cal L}}(\hat \Gamma)$. The function ${\cal R}({\bf s})$ is upper bounded implies that ${\bf w}, \mathbf{a}, {\bf u}$ are bounded. Then the boundedness of $\|{\bf h} - ({\bf u})_+ + \frac{\zeta}{\gamma}\|$ indicates that $\bf h$ is also bounded. Thus, ${\bf s}$ is bounded and statement \ref{lem_ele:levelset_pro} holds.

Statements \ref{lem_ele:grad_z} and \ref{lem_ele:grad_h} can be obtained by direct computation.
\end{proof}

% \subsection{Statements of smoothing lemmas}
\subsection{Proof of Lemma \ref{lem3.2}}
\begin{proof}
Using \cref{gradient} \ref{lem_ele:grad_z}, we have
\begin{align}
\label{cond1}
 &\nabla_{\bf{z}} {\cal L}({\bf z},{\bf h}', {\bf u}',\xi,\zeta,\gamma)  - \nabla_{\bf{z}} {\cal L}({\bf z},{\bf h}, {\bf u},\xi,\zeta,\gamma)\\
=\ &\left[
\begin{array}{c}
\gamma \Delta_1 {\bf{w}} -\left(\Psi(\mathbf{h}')-\Psi(\mathbf{h}) \right)^{\sss \top} \xi - \gamma \Delta_3 \\
\frac{2}{T} \sum_{t=1}^T  \Delta_{2,t}  {\bf a} - \frac{2}{T} \sum_{t=1}^T \left(\Phi(h'_t) - \Phi(h_t)\right)^{\sss \top} y_t
\end{array}
\right],
\notag
\end{align}
where  $\Delta_1 = \Psi(\mathbf{h}')^{\sss \top} \Psi(\mathbf{h}') - \Psi(\mathbf{h})^{\sss \top} \Psi(\mathbf{h})$ and $\Delta_{2,t} = \Phi(h'_t)^{\sss \top}\Phi(h'_t) - \Phi(h_t)^{\sss \top} \Phi(h_t) $ and
$\Delta_3 = \Psi(\mathbf{h}'){\bf u}' - \Psi({\mathbf{h}}){\bf u}$.
It is easy to see that
\begin{eqnarray}
\label{cond2} \|\Delta_1\|
&=& \nonumber\|\Psi(\mathbf{h}')^{\sss \top} \Psi(\mathbf{h}') - \Psi(\mathbf{h}')^{\sss \top} \Psi(\mathbf{h})
+ \Psi(\mathbf{h}')^{\sss \top} \Psi(\mathbf{h})\
-\Psi(\mathbf{h})^{\sss \top} \Psi(\mathbf{h})
 \|\\
 &\le &  \left( \|\Psi(\mathbf{h}')\| + \|\Psi(\mathbf{h})\| \right) \|\Psi(\mathbf{h}') - \Psi(\mathbf{h}) \|.
\end{eqnarray}
Similarly, we have
\begin{eqnarray}
  \|\Delta_{2,t}\| &\le& \left( \|\Phi(h'_t)\| + \|\Phi(h_t)\| \right) \|\Phi(h'_t) - \Phi(h_t)\|,\ \forall t\in [T],\\
  \|\Delta_3\| &\le& \|\Psi({\bf{h}'})\|\|{\bf u}' - {\bf u}\| + \|{\bf u}\|\|\Psi({\bf h}') - \Psi({\bf h})\|.
\end{eqnarray}

Since ${\bf s}, {\bf s}' \in \Omega_{{\cal L}}(\hat \Gamma)$, we know that
\begin{eqnarray*}
{\ell}({\bf s}) + P({\bf s}) +\frac{\gamma}{2} {\left\|{\mathbf{u}}-\Psi(\mathbf{h})\mathbf{w} + \frac{\xi}{\gamma}\right\|}^2
    + \frac{\gamma}{2} {\left\|\mathbf{h}-(\mathbf{u})_{+}+ \frac{\zeta}{\gamma}\right\|}^2\le \delta.
\end{eqnarray*}
This, together with the expressions of $\ell(\bf{s})$ in \eqref{prob:ror} and $P(\mathbf{s})$ in \eqref{P}, yields
{\small
\begin{align}
\label{cond5}
    \quad\quad \|W\|_F \le  \sqrt{\frac{\delta}{\lambda_2}}, \ \|{\bf a}\| \le \sqrt{\frac{\delta}{\min\{\lambda_1,\lambda_5\}}},\ \|{\bf{w}}\|\le \sqrt{ \frac{\delta}{\min\{\lambda_2,\lambda_3,\lambda_4\}}},\ \|{\bf u}\| \le \sqrt{\frac{\delta}{\lambda_6}}.
\end{align}}Moreover, since
$\|\mathbf{h}\| - \|(\mathbf{u})_{+} - \tfrac{\zeta}{\gamma}\| \leq \|\mathbf{h}-(\mathbf{u})_{+} + \tfrac{\zeta}{\gamma}\| \leq \sqrt{\tfrac{2 \delta}{\gamma}}$,
we find
\begin{eqnarray}
\label{cond6}
\|\mathbf{h}\| \le \delta_0.
\end{eqnarray}

Using \eqref{PhiPsi}, we can easily obtain that \begin{eqnarray}
\label{cond7}
\|\Psi({\bf{h}})- \Psi({\bf{h}}')\| \leq \sqrt{r}\|{\bf{h}}' - {\bf{h}}\|, \quad \|\Phi(h_t') - \Phi(h_t)\| \leq \sqrt{m}\|{h}_{t}' - {h}_{t}\|,\\
\label{cond8}
\|\Psi({\bf{h}})\| = \sqrt{r(\|\mathbf{h}\|^2 + \|X\|^2 + T)}, \quad \|\Phi(h_t)\| = \sqrt{m(\|{h_t}\|^2 +1)}.
\end{eqnarray}

Using the facts that for any  $\iota_1,\iota_1, \ldots,\iota_j \in \mathbb{R}$, any  $g_1,g_2,\ldots,g_j \in \mathbb{R}^{ n_r}$, and any matrices $B_1,B_2,\ldots,B_j \in \mathbb{R}^{n_c \times n_r}$, $\|B_1\|\le \|B_1\|_F$, and
\begin{eqnarray}
\label{cond}\quad\quad
\|\sum_{i=1}^{(j)}\iota_j B_j g_j \| \le \sum_{i=1}^j |\iota_j| \|B_j\|\|g_j\| ,\quad
\sum_{i=1}^j\|\iota_i g_i\| \le \max_{1\le i\le j}\{|\iota_i|\}\sqrt{j}\|(g_1;\ldots;g_j)\|,
\end{eqnarray}
taking the norm of both sides of \eqref{cond1}, and employing \eqref{cond2}-\eqref{cond8},
we can get \eqref{grd_lip_z} with the expression of $L_1(\xi,\zeta, \gamma,\hat{r})$ in \eqref{L12} as desired.

Using \cref{gradient} \ref{lem_ele:grad_h}, we have by direct computation
\begin{eqnarray*}
& & \nabla_{\bf{h}} {\cal L}({\bf z},{\bf h}, {\bf u}',\xi,\zeta,\gamma)  - \nabla_{\bf{h}} {\cal L}({\bf z},{\bf h}, {\bf u},\xi,\zeta,\gamma)\\
&=& \gamma W^T \sum_{t=1}^{T-1} (u_{t+1} - u'_{t+1}) + \gamma \sum_{t=1}^T ((u_t)_+ -(u'_t)_+).
\end{eqnarray*}
Taking the norm of both sides of the above system of equations, employing \eqref{cond5},
\eqref{cond}, and the facts
$\|(u_t)_+ - (u'_t)_+\| \le \|u'_t-u_t\|$ for each $t$,
we can get \eqref{grd_lip_h} with
$L_2(\xi,\zeta, \gamma, \hat r)$ in the form of \eqref{L12} as desired.
\end{proof}

\subsection{Proof of Lemma \ref{lemma:sj_in_O}}
\begin{proof}
By \cref{subpro_z_subproblem}, \cref{subpro_h_subproblem} and \cref{subpro_u_subproblem}, we know that for any $j \in \mathbb{N}$:
\begin{align}
\label{des_L_kj}
   \mathcal{L}\big(\mathbf{s}^{(j)},\xi,\zeta,\gamma\big) \leq \mathcal{L}\big(\mathbf{s}^{(j)}_{\bf{h}},\xi,\zeta,\gamma\big)
   {\leq} \mathcal{L}\big(\mathbf{s}^{(j)}_{\bf{z}},\xi,\zeta,\gamma\big)
    {\leq} \mathcal{L}\big(\mathbf{s}^{(j-1)},\xi,\zeta,\gamma\big).
\end{align}
By the definition of $\Gamma$ in \cref{Alg:BCD}
 and  \cref{des_L_kj}, we can deduce that
\begin{align}
\label{cond:func_value bounded}
    \mathcal{L}\big(\mathbf{s}^{(j)},\xi,\zeta,\gamma\big) \leq \Gamma, \quad \forall j \in \mathbb{N}.
\end{align}
By the definition of $\Omega_{\mathcal{L}}(\Gamma)$ and \cref{gradient} \ref{lem_ele:levelset_pro}, the proof is completed.
\end{proof}

\subsection{Proof of Lemma \ref{lemma: directional derivative}}
\begin{proof}
It is clear that $\Omega_{\mathcal{L}}(\Gamma)$ is compact by \cref{gradient} \ref{lem_ele:levelset_pro}.
For the smooth part $g$ in $\mathcal{L}$, its gradient for those $\mathbf{s} \in \Omega_{\mathcal{L}}(\Gamma)$ is upper bounded. Now, let us turn to consider the nonsmooth part $q$ in $\mathcal{L}$. Let  $\mathbf{s} = (\bf{z}; \bf{h};\bf{u})$ and $\mathbf{s}' = (\bf{z'}; \bf{h'};\bf{u'})$ be any two points in $\Omega_{\mathcal{L}}(\Gamma)$. We have
\begin{align*}
    &\big|q(\mathbf{s}', \zeta, \gamma) - q(\mathbf{s}, \zeta, \gamma)\big|
    \\
    \leq\ &\tfrac{\gamma}{2}\Big| \big\|\mathbf{h}'-(\mathbf{u}')_{+}+\tfrac{\zeta}{\gamma}\big\|^2 - \big\|\mathbf{h}-(\mathbf{u})_{+}+\tfrac{\zeta}{\gamma}\big\|^2\Big|\\
    \leq\ &\tfrac{\gamma}{2}\big\|\mathbf{h}'-(\mathbf{u}')_{+}-(\mathbf{h}-(\mathbf{u})_{+})\big\| \big\|\mathbf{h}'-(\mathbf{u}')_{+} + \mathbf{h}-(\mathbf{u})_{+} + 2\tfrac{\zeta}{\gamma}\big\|\\
    \leq\ & \left(2 \gamma \max_{{\bf{s}}\in \Omega_{\mathcal{L}}(\Gamma)}\{\|\mathbf{h}\|_{\infty}+ \|\mathbf{u}\|_{\infty}\} + \|\zeta\|\right)(\|\mathbf{h}' - \mathbf{h}\|+\|\mathbf{u}' - \mathbf{u}\|).
\end{align*}
Up to now, we have proved the Lipschitz continuity of $g$ and $q$ on $\Omega_{\mathcal{L}}(\Gamma)$, which implies that $\mathcal{L}$ is Lipschitz continuous on $\Omega_{\mathcal{L}}(\Gamma)$.

The above result, together with the piecewise smoothness of function $\mathcal{L}$, yields that $\mathcal{L}$ is directionally differentiable on $\Omega_{\mathcal{L}}(\Gamma)$ by \cite{mifflin1977semismooth}.
\end{proof}

\subsection{Proof of Lemma \ref{lemma:regular}}
\begin{proof}
By  \cref{almfunc_subproblem_2cons}, the directional derivative of $\mathcal{L}$ at $\bar{\mathbf{s}}$ along ${d} \in \mathbb{R}^{N_\mathbf{w} + N_\mathbf{a} + 2rT}$ refers to $\mathcal{L}'(\bar{\mathbf{s}},\xi,\zeta,\gamma ; {d}) = g'(\bar{\mathbf{s}}, \xi, \gamma ; {d}) + q'(\bar{\mathbf{s}}, \zeta, \gamma ; {d})$.
It is clear that
\begin{align}
\label{dir_g}
    g'(\bar{\mathbf{s}},\xi,\gamma ; {d}) = \langle \nabla_{\mathbf{z}}g(\bar{\mathbf{s}}, \xi, \gamma), {d}_{\bf{z}} \rangle + \langle \nabla_{\mathbf{h}}g(\bar{\mathbf{s}}, \xi, \gamma), {d}_{\bf{h}} \rangle + \langle \nabla_{\mathbf{u}}g(\bar{\mathbf{s}}, \xi, \gamma), {d}_{{\bf{u}}} \rangle.
\end{align}
It remains to consider the directional derivative of nonsmooth part $q$. The function $q$ can be separated into $rT$ one dimensional functions with the same structure, i.e.,
\begin{displaymath}
    \phi(\bar{{h}}, \bar{{u}}) =(\bar{{h}} - (\bar{{u}})_{+} + \nu_{1})^2 - \nu_1^2,
\end{displaymath}
where $\bar{{h}}, \bar{{u}} \in \mathbb{R}$ are variables and $\nu_{1} \in \mathbb{R}$ is a constant.
The directional derivative of $\phi$ along the direction $ {({\bar{d}}_{1}; {\bar{d}}_{2})} \in \mathbb{R}^{2}$ can be represented as the sum of the directional derivatives of $\phi$ along ${({\bar{d}}_{1}; 0)}$ and ${(0; {\bar{d}}_{2})}$ by the definition of directional derivative, i.e.,
{\small{
\begin{align*}
    \phi'\big(\bar{{h}}, \bar{{u}};({\bar{d}}_{1}, {\bar{d}}_{2})\big) &= \lim_{\lambda \downarrow 0}\frac{\Big(\bar{{h}} + \lambda {\bar{d}}_{1} - \big(\bar{{u}} + \lambda {\bar{d}}_{2}\big)_{+} + \nu_{1}\Big)^2 - \Big(\bar{{h}} - \big(\bar{{u}}\big)_{+} + \nu_{1} \Big)^2}{\lambda}\\
    &= \phi'\big(\bar{{h}}, \bar{{u}};({\bar{d}}_{1}, 0)\big)
    + \phi'\big(\bar{{h}}, \bar{{u}};(0, {\bar{d}}_{2})\big)
    - \lim_{\lambda \downarrow 0} \
    \frac{2\lambda {\bar{d}}_{1}\big(({u}+ \lambda {\bar{d}}_{2})_+ - ({u})_+\big)}{\lambda}
\end{align*}}}where
\begin{align*}
     \phi'\big(\bar{{h}}, \bar{{u}};({\bar{d}}_{1}, 0)\big)
     &= \lim_{\lambda \downarrow 0} \frac{\left(\bar{{h}} + \lambda {\bar{d}}_{1} - (\bar{{u}})_{+} + \nu_{1}\right)^2 - \left(\bar{{h}} - (\bar{{u}})_{+} + \nu_{1} \right)^2}{\lambda}\\
    &=\lim_{\lambda \downarrow 0} \frac{(\bar{{h}} + \lambda {\bar{d}}_{1} + \nu_{1})^2-(\bar{{h}} + {\nu_{1}})^2-2(\lambda {\bar{d}}_{1})(\bar{{u}})_{+}}{\lambda},
\end{align*}
\begin{align*}
    \phi'\big(\bar{{h}}, \bar{{u}};(0, d_{2})\big)
     &= \lim_{\lambda \downarrow 0}\frac{\left(\bar{{h}} + \nu_{1} - (\bar{{u}} + \lambda {\bar{d}}_{2})_{+} \right)^2 - \left(\bar{{h}} + \nu_{1} - (\bar{{u}})_{+} \right)^2}{\lambda}\\
    &= \lim_{\lambda \downarrow 0} \frac{(\bar{{u}} + \lambda {\bar{d}}_{2})^2_+ -(\bar{{u}})^2_{+} -2(\bar{{h}} + \nu_{1})\big((\bar{{u}} + \lambda {\bar{d}}_{2})_{+} - (\bar{{u}})_{+}\big)}{\lambda},
\end{align*}
and $\lim_{\lambda \downarrow 0}\tfrac{2\lambda {\bar{d}}_{1}(({u} + \lambda {\bar{d}}_{2})_{+} - ({u})_{+})}{\lambda} = 0$. By setting $\bar{h} = \bar{\mathbf{h}}_{i}$, $\bar{u} = \bar{\mathbf{u}}_{i}$, ${\bar{d}}_{1} = (d_{\bf{h}})_{i}$, ${\bar{d}}_{2} = (d_{\bf{u}})_{i}$, $\nu_{1} = \tfrac{\zeta_{i}}{\gamma}$,
we have
\begin{eqnarray*}
q'(\bar{\mathbf{s}}, \zeta, \gamma; {\bar{d}}) &=& \frac{\gamma}{2} \sum_{i=1}^{rT} \phi'\big(\bar{\mathbf{h}}_{i}, \bar{\mathbf{u}}_{i};((d_{\bf{h}})_{i},(d_{\bf{u}})_{i})\big)\\
    &=& \frac{\gamma}{2}\sum_{i=1}^{rT}\phi'\big(\bar{\mathbf{h}}_{i}, \bar{\mathbf{u}}_{i};((d_{\bf{h}})_i, 0)\big) + \phi'_{i}\big(\bar{\mathbf{h}}_{i}, \bar{\mathbf{u}}_{i};(0,(d_{\bf{u}})_i)\big)\\
    &=& q'\big(\bar{\mathbf{s}}, \zeta, \gamma ; (0,{d}_{\bf{h}}, {0})\big)
    + q'\big(\bar{\mathbf{s}}, \zeta, \gamma ; (0, 0, {d}_{\bf{u}})\big).
\end{eqnarray*}
This, along with \cref{dir_g}, yields that
\begin{eqnarray*}
& & \mathcal{L}'(\bar{\mathbf{s}},\xi, \zeta, \gamma ; {d})\\
    &=&  \langle \nabla_{\mathbf{z}}g(\bar{\mathbf{s}}, \xi, \gamma), {d}_{\bf{z}} \rangle + \langle \nabla_{\mathbf{h}}g(\bar{\mathbf{s}}, \xi, \gamma), {d}_{\bf{h}} \rangle + \langle \nabla_{\mathbf{u}}g(\bar{\mathbf{s}}, \xi, \gamma), {d}_{\bf{u}} \rangle\\
    & & \quad\quad  + q'\big(\bar{\mathbf{s}}, \zeta, \gamma ; (0, {d}_{\bf{h}}, {0})\big)
    + q'\big(\bar{\mathbf{s}}, \zeta, \gamma ; (0, 0, {d}_{\bf{u}})\big)\\
   &=&  \mathcal{L}'(\bar{\mathbf{s}},\xi, \zeta, \gamma ; ({d}_{\bf{z}}, {0}, {0}))
    + \mathcal{L}'(\bar{\mathbf{s}},\xi, \zeta, \gamma ; ({0},{d}_{\bf{h}}, {0})) + \mathcal{L}'(\bar{\mathbf{s}},\xi, \zeta, \gamma ; ({0}, {0}, {d}_{\bf{u}})).
\end{eqnarray*}
Hence \cref{lemma:regular} holds.
\end{proof}

\section{Parameters for numerical experiments in section 5.4\label{subsec:hyperpare}}

%In subsection 5.4, we first use a grid search method to select suitable learning rates for GDs and SGDs under various initialization strategies. The grid search method is over learning rates $\{10^{-4}$, $10^{-3}$, $10^{-2}$, $10^{-1}$, $1\}$ and initialization strategies: random normal initialization with zero mean and standard deviations of \(10^{-3}\) and \(10^{-1}\), He initialization, Glorot initialization, and LeCun initialization. The details of the experimental procedure are as follows: For all the combinations of learning rates and initialization strategies, we repeated the GDs and SGDs methods 30 times, recording the learning rate with achieved the best averaged cross-validation performance. 

The final selected learning rates for GDs and SGDs, as well as the clipping norm for GDC, are listed in \cref{tab:lr}.

\addtocounter{table}{1}
\begin{table}[H]
    \centering
    \scriptsize
\begin{tabular}{|c|c|c|c|c|c|c|}
\hline
      &                                                                         & He   & $\mathcal{N}(0, 10^{-3})$ & $\mathcal{N}(0, 10^{-1})$ & Glorot & LeCun \\ \hline
GD    & \bf{Synthetic dataset} ($T = 10$) & 1e-4 & 1e-3                    & 1e-4                  & 1      & 1     \\ \hline
      & \bf{Volatility of  S$\mathit{\&}$P index} & 1e-4 & 0.01                    & 0.01                  & 0.01   & 0.01  \\ \hline
      & \bf{Synthetic dataset} ($T = 500$)    & 0.01 & 0.01                     & 0.01                   & 1e-3   & 1e-3  \\ \hline
GDC   & \bf{Synthetic dataset} ($T = 10$)  & 1 (6)& 1e-4   (1)                 & 1e-4  (1)                & 1  (6)   & 1   (6)  \\ \hline
      &  \bf{Volatility of  S$\mathit{\&}$P index} & 1e-4 (3) & 0.01  (1)                  & 0.1   (1)                & 0.1 (4)   & 0.1 (1)   \\ \hline
      & \bf{Synthetic dataset} ($T = 500$)   & 1e-4 (1)& 0.01  (1)                   & 0.01  (4)                 & 0.01  (1.5) & 0.1  (0.5)  \\ \hline
GDNM & \bf{Synthetic dataset} ($T = 10$)  & 1e-3 & 1e-4                    & 1e-4                  & 1e-4   & 0.1   \\ \hline
      & \bf{Volatility of  S$\mathit{\&}$P index} & 1e-4 & 0.01                    & 0.01                  & 0.01   & 0.01  \\ \hline
      & \bf{Synthetic dataset} ($T = 500$)    & 0.01 & 0.01                     & 0.01                   & 0.01   & 0.01  \\ \hline
SGD   & \bf{Synthetic dataset} ($T = 10$)  & 0.1  & 0.1                     & 0.1                   & 0.1    & 0.1   \\ \hline
      & \bf{Volatility of  S$\mathit{\&}$P index} & 0.01 & 0.01                    & 0.01                  & 0.01   & 0.01  \\ \hline
      & \bf{Synthetic dataset} ($T = 500$)    & 0.01 & 1e-3                    & 0.01                   & 0.01    & 0.01   \\ \hline
Adam  & \bf{Synthetic dataset} ($T = 10$)  & 0.1  & 0.01                    & 0.01                  & 0.01   & 0.01  \\ \hline
      & \bf{Volatility of  S$\mathit{\&}$P index} & 0.01 & 0.01                    & 0.01                  & 0.01   & 0.01  \\ \hline
      & \bf{Synthetic dataset} ($T = 500$)    & 0.01 & 0.01                    & 0.01                  & 0.01   & 0.01  \\ \hline
\end{tabular}

    \caption{The learning rates for GDs and SGDs, and the clipping norm value for GDC (the second number in each cell for parameters)  under different initialization strategies.}
    \label{tab:lr}
\end{table}
% \vspace{-1cm}

\bibliographystyle{siamplain}
% \bibliography{references}

\begin{thebibliography}{10}

\bibitem{arneric2014garch}
{\sc P.~T. Arneri{\'c}~Josip and A.~Zdravka}, {\em Garch based artificial
  neural networks in forecasting conditional variance of stock returns}, Croat.
  Oper. Res. Rev., 5 (2014), pp.~329--343,
  \url{https://doi.org/10.17535/crorr.2014.0017}.

\bibitem{bengio2009learning}
{\sc Y.~Bengio}, {\em Learning deep architectures for {AI}}, Found. Trends
  Mach. Learn.,  (2009), pp.~136, \url{https://doi.org/10.1561/2200000006}.

\bibitem{bengio2013advances}
{\sc Y.~Bengio, N.~Boulanger-Lewandowski, and R.~Pascanu}, {\em Advances in
  optimizing recurrent networks}, in I{EEE} International Conference on
  Acoustics, Speech and Signal Processing, I{EEE}, 2013, pp.~8624--8628.

\bibitem{bengio1994learning}
{\sc Y.~Bengio, P.~Simard, and P.~Frasconi}, {\em Learning long-term
  dependencies with gradient descent is difficult}, IEEE Trans. Neural Netw.
  Learn. Syst., 5 (1994), pp.~157--166,
  \url{https://doi.org/10.1109/72.279181}.

\bibitem{bertsekas1997nonlinear}
{\sc D.~P. Bertsekas}, {\em Nonlinear Programming}, Athena Scientific, Nashua,
  NH, 2nd~ed., 1999.

\bibitem{bucci2020realized}
{\sc A.~Bucci}, {\em Realized volatility forecasting with neural networks}, J.
  Financ. Econ., 18 (2020), pp.~502--531,
  \url{https://doi.org/10.1093/jjfinec/nbaa008}.

\bibitem{carreira2014distributed}
{\sc M.~Carreira-Perpinan and W.~Wang}, {\em Distributed optimization of deeply
  nested systems}, in the 17th International Conference on Artificial
  Intelligence and Statistics, Reykjavic, Iceland, 2014, pp.~10--19.

\bibitem{chandriah2021rnn}
{\sc K.~K. Chandriah and R.~V. Naraganahalli}, {\em {RNN}/{LSTM} with modified
  {A}dam optimizer in deep learning approach for automobile spare parts demand
  forecasting}, Multimed. Tools. Appl., 80 (2021), pp.~26145--26159,
  \url{https://doi.org/10.1007/s11042-021-10913-0}.

\bibitem{chen2017augmented}
{\sc X.~Chen, L.~Guo, Z.~Lu, and J.~J. Ye}, {\em An augmented {L}agrangian
  method for non-{L}ipschitz nonconvex programming}, SIAM J. Numer. Anal., 55
  (2017), pp.~168--193, \url{https://doi.org/10.1137/15M1052834}.

\bibitem{chollet2015keras}
{\sc F.~Chollet et~al.}, {\em Keras}.
\newblock \url{https://keras.io}, 2015.

\bibitem{clarke1990optimization}
{\sc F.~H. Clarke}, {\em Optimization and Nonsmooth Analysis}, SIAM,
  Philadelphia, PA, 1990.

\bibitem{cui2020multicomposite}
{\sc Y.~Cui, Z.~He, and J.-S. Pang}, {\em Multicomposite nonconvex optimization
  for training deep neural networks}, SIAM J. Optim., 30 (2020),
  pp.~1693--1723, \url{https://doi.org/10.1137/18M1231559}.

\bibitem{elman1990finding}
{\sc J.~L. Elman}, {\em Finding structure in time}, Cogn. Sci., 14 (1990),
  pp.~179--211, \url{https://doi.org/10.1016/0364-0213(90)90002-E}.

\bibitem{goodfellow2016deep}
{\sc I.~Goodfellow, Y.~Bengio, and A.~Courville}, {\em Deep learning}, MIT
 Press, Cambridge, MA, 2016.

\bibitem{graves2013speech}
{\sc A.~Graves, A.-r. Mohamed, and G.~Hinton}, {\em Speech recognition with
  deep recurrent neural networks}, in the 38th IEEE International Conference on
  Acoustics, Speech and Signal Processing, Vancouver, BC, 2013, I{EEE},
  pp.~6645--6649.

\bibitem{hallak2023adaptive}
{\sc N.~Hallak and M.~Teboulle}, {\em An adaptive {L}agrangian-based scheme for
  nonconvex composite optimization}, Math. Oper. Res.,  (2023), p.~136,
  \url{https://doi.org/10.1287/moor.2022.1342}.

\bibitem{kruger2003frechet}
{\sc A.~Y. Kruger}, {\em On {F}r{\'e}chet subdifferentials}, J. Math. Sci., 116
  (2003), pp.~3325--3358, \url{https://doi.org/10.1023/A:1023673105317}.

\bibitem{le2015simple}
{\sc Q.~V. Le, N.~Jaitly, and G.~E. Hinton}, {\em A simple way to initialize
  recurrent networks of rectified linear units}, preprint,
  https://arxiv.org/abs/1504.00941, 2015.

\bibitem{liu2021linearly}
{\sc W.~Liu, X.~Liu, and X.~Chen}, {\em Linearly constrained nonsmooth
  optimization for training autoencoders}, SIAM J. Optim., 32 (2022),
  pp.~1931--1957, \url{https://doi.org/10.1137/21M1408713}.

\bibitem{liu2022inexact}
{\sc W.~Liu, X.~Liu, and X.~Chen}, {\em An inexact augmented {L}agrangian
  algorithm for training leaky {R}e{LU} neural network with group sparsity}, J.
  Mach. Learn. Res., 24 (2023), pp.~1--43,
  \url{http://jmlr.org/papers/v24/22-0491.html}.

\bibitem{mifflin1977semismooth}
{\sc R.~Mifflin}, {\em Semismooth and semiconvex functions in constrained
  optimization}, SIAM J. Control Optim., 15 (1977), pp.~959--972,
  \url{https://doi.org/10.1137/0315061}.

\bibitem{mikolov2010recurrent}
{\sc T.~Mikolov, M.~Karafi{\'a}t, L.~Burget, J.~Cernock{\`y}, and
  S.~Khudanpur}, {\em Recurrent neural network based language model}, in the
  11th Annual Conference of the International Speech Communication Association,
  Chiba, 2010, I{SCA}, pp.~1045--1048.

\bibitem{mirmirani2004comparison}
{\sc S.~Mirmirani and H.~C. Li}, {\em A comparison of {VAR} and neural networks
  with genetic algorithm in forecasting price of oil}, in Applications of
  Artificial Intelligence in Finance and Economics, Emerald Publishing Limited,
  Leeds, 2004, pp.~203--223.

\bibitem{pascanu2013difficulty}
{\sc R.~Pascanu, T.~Mikolov, and Y.~Bengio}, {\em On the difficulty of training
  recurrent neural networks}, in the 30th International Conference on Machine
  Learning, Atlanta GA, 2013, I{MLS}.

\bibitem{peng2020computation}
{\sc D.~Peng and X.~Chen}, {\em Computation of second-order directional
  stationary points for group sparse optimization}, Optim. Methods Softw., 35
  (2020), pp.~348--376, \url{https://doi.org/10.1080/10556788.2019.1684492}.

\bibitem{rockafellar2009variational}
{\sc R.~T. Rockafellar and R.~J.-B. Wets}, {\em Variational Analysis},
  Springer, Berlin, 2009.

\bibitem{lstm2014lsa}
{\sc H.~Sak, A.~Senior, and F.~Beaufays}, {\em Long short-term memory recurrent
  neural network architectures for large scale acoustic modeling}, in the 15th
  Annual Conference of the International Speech Communication Association,
  Singapore, 2014, I{SCA}, pp.~338--342.

\bibitem{sundermeyer2012lstm}
{\sc M.~Sundermeyer, R.~Schl{\"u}ter, and H.~Ney}, {\em {LSTM} neural networks
  for language modeling}, in the 13th Annual Conference of the International
  Speech Communication Association, Portland, Oregon, 2012, I{SCA}.

%\bibitem{valiant1984theory}
%{\sc L.~G. Valiant}, {\em A theory of the learnable}, Commun. ACM, 27 (1984),
%  pp.~1134--1142, \url{https://doi.org/10.1145/1968.1972}.

\bibitem{ye2000multiplier}
{\sc J.~J. Ye}, {\em Multiplier rules under mixed assumptions of
  differentiability and {L}ipschitz continuity}, SIAM J. Control Optim., 39
  (2000), pp.~1441--1460, \url{https://doi.org/10.1137/S0363012999358476}.

\bibitem{zhang2016multi}
{\sc X.~Zhang, N.~Gu, and H.~Ye}, {\em Multi-{GPU} based recurrent neural
  network language model training}, in the International Conference of
  Pioneering Computer Scientists, Engineers and Educators, Springer, 2016,
  pp.~484--493.

\bibitem{zhang2017convergent}
{\sc Z.~Zhang and M.~Brand}, {\em Convergent block coordinate descent for
  training {T}ikhonov regularized deep neural networks}, in the 31st
  International Conference on Neural Information Processing Systems, NY, 2017,
  pp.~1719--1728.


%----------The new, need check
\bibitem{he2015delving}
{\sc K.~He, X.~Zhang, S.~Ren, and J.~Sun.}
{\em Delving deep into rectifiers: surpassing human-level performance on
  imagenet classification,} in the
2015 IEEE International Conference on Computer Vision (ICCV), Santiago, Chile, 2015, pp.~1026--1034.

\bibitem{glorot2010understanding}
{\sc X.~Glorot and Y.~Bengio.}
\newblock {\em Understanding the difficulty of training deep feedforward neural
  networks,} in the 13th International Conference on
  Artificial Intelligence and Statistics, Sardinia, Italy, 2010, pp.~249--256.

\bibitem{klambauer2017self}
{\sc G.~Klambauer, T.~Unterthiner, A.~Mayr, and S.~Hochreiter,}
\newblock {\em Self-normalizing neural networks,} in the 31st International Conference on Neural Information Processing Systems, California, USA, 2017, pp.~972--981.


\bibitem{bergstra2012random}
{\sc J.~Bergstra and Y. Bengio,}
\newblock {\em Random search for hyper-parameter optimization}, J. Mach. Learn. Res., 13 (2012), pp.~281--305.


\end{thebibliography}

\end{document}